\documentclass{amsart}

\usepackage[vcentermath]{youngtab}
\usepackage{amssymb}

\newtheorem{theorem}{Theorem}
\newtheorem{lemma}{Lemma}
\newtheorem{proposition}{Proposition}

\theoremstyle{definition}
\newtheorem{definition}{Definition}
\newtheorem{remark}{Remark}
\newtheorem{example}{Example}

\usepackage{amsmath}

\begin{document}

\title[Affine pavings for varieties of partial flags]{Existence of affine pavings for varieties of partial flags associated to nilpotent elements}
\thanks{Work supported in part by the ISF Grant Nr. 882/10 and by the ANR project NilpOrbRT}
\author{Lucas Fresse}
\address{Universit\'e de Lorraine, CNRS, Institut \'Elie Cartan de Lorraine, UMR 7502, Vandoeu\-vre-l\`es-Nancy, F-54506, France}
\date{\today}

\begin{abstract}
The flag variety of a complex reductive linear algebraic group $G$ is by definition the quotient $G/B$ by a Borel subgroup. It can be regarded as the set of Borel subalgebras of $Lie(G)$. Given a nilpotent element $e$ in $Lie(G)$, one calls Springer fiber the subvariety formed by the Borel subalgebras that contain $e$. Springer fibers have in general a complicated structure (not irreducible, singular). Nevertheless, a theorem by C. De Concini, G. Lusztig, and C. Procesi asserts that, when $G$ is classical, a Springer fiber can always be paved by finitely many subvarieties isomorphic to affine spaces. In this paper, we study varieties generalizing the Springer fibers to the context of partial flag varieties, that is, subvarieties of the quotient $G/P$ by a parabolic subgroup (instead of a Borel subgroup). The main result of the paper is a generalization of De Concini, Lusztig, and Procesi's theorem to this context.
\end{abstract}


\maketitle

\section{Introduction}

\label{section-1}

\subsection{Springer fibers}
Let $G$ be a reductive, connected, linear algebraic group over $\mathbb{C}$ (or any algebraically closed field of characteristic zero).
Let $\mathcal{B}$ be the variety of all Borel subalgebras of $\mathfrak{g}:=Lie(G)$.
Equivalently, $\mathcal{B}=G/B_0$ is the quotient by a Borel subgroup. This is an algebraic projective variety (the {\em full flag variety}).

Let $\mathcal{N}\subset\mathfrak{g}$ be the nilpotent variety (the set of nilpotent elements). It is a union of finitely many orbits ({\em nilpotent orbits}) for the adjoint action $G\times\mathfrak{g}\to\mathfrak{g}$, $(g,x)\mapsto g\cdot x$.

Given $e\in\mathcal{N}$, the set
\[\mathcal{B}_e=\{\mathfrak{b}\in\mathcal{B}:e\in\mathfrak{b}\}\]
is a closed (projective) subvariety of $\mathcal{B}$.\ It is called a {\em Springer fiber} as it coincides with the fiber $\pi_2^{-1}(e)$ of the Springer resolution
\[\pi_2:T^*\mathcal{B}=\{(\mathfrak{b},e)\in\mathcal{B}\times\mathcal{N}:e\in\mathfrak{b}\}\to\mathcal{N},\ (\mathfrak{b},e)\mapsto e\]
(see \cite{Slodowy,Springer1,Steinberg}).
If $e=0$, then $\mathcal{B}_e=\mathcal{B}$. At the other extreme, if $e$ is regular, then $\mathcal{B}_e$ consists of one point. In general, $\mathcal{B}_e$ is not irreducible (though always connected). It is equidimensional and one has $\dim\mathcal{B}_e=\frac{1}{2}(\dim Z_G(e)-\mathrm{rank}\,G)$, where $Z_G(e)=\{g\in G:g\cdot e=e\}$ (see \cite{Spaltenstein-Topology,Steinberg}).

Springer fibers are classical objects in representation theory (see, e.g., \cite{Bezrukavnikov-Mirkovic-Rumynin,Kazhdan-Lusztig,Springer2}).
The goal of this paper is to study parabolic analogues of the Springer fibers.

\subsection{The variety $\mathcal{P}_{e,\mathfrak{i}}$}
\label{section-1-2}
Fix a parabolic subgroup $P\subset G$ and let $\mathfrak{n}_P\subset\mathfrak{p}\subset\mathfrak{g}$ be the corresponding parabolic subalgebra and its nilradical.
Let $\mathcal{P}=G/P$.\ Equivalently, $\mathcal{P}$ can be regarded as the variety of all parabolic subalgebras of $\mathfrak{g}$ of same type as $\mathfrak{p}$
(a {\em partial flag variety}).

Our main object of study is introduced in the following definition.

\begin{definition}
\label{definition-1}
Given $e\in\mathcal{N}$ and a $P$-stable subspace $\mathfrak{i}\subset\mathfrak{p}$, we define
\[\mathcal{P}_{e,\mathfrak{i}}=\{gP\in\mathcal{P}:g^{-1}\cdot e\in\mathfrak{i}\}.\]
\end{definition}

Clearly $\mathcal{P}_{e,\mathfrak{i}}$ is a closed subvariety of $\mathcal{P}$, nonempty if and only if $e\in G\cdot\mathfrak{i}$.
The next examples correspond to particular cases of   $\mathcal{P}_{e,\mathfrak{i}}$, all related to resolutions of the nilpotent variety or of nilpotent orbit closures.

\begin{example}
\label{example-1}
(a) $\mathcal{P}_{e,\mathfrak{p}}$ (corresponding to $\mathfrak{i}=\mathfrak{p}$) was studied by R.\ Steinberg \cite{Steinberg} and is sometimes called a {\em Steinberg variety}. It coincides with the fiber over $e$ of the map
\[\{(gP,x)\in \mathcal{P}\times \mathcal{N}:g^{-1}\cdot x\in\mathfrak{p}\}\to\mathcal{N},\ (gP,x)\mapsto x,\]
which is a partial resolution of the nilpotent variety $\mathcal{N}$ in the sense of \cite{Borho-MacPherson}. \\[1mm]
(b) $\mathcal{P}_{e,\mathfrak{n}_P}$ (corresponding to $\mathfrak{i}=\mathfrak{n}_P$) is usually called a {\em Spaltenstein variety} (cf. \cite{Spaltenstein-Indag,Spaltenstein-book}).  It coincides with a fiber of the map
\[G\times_P\mathfrak{n}_P\to G\cdot \mathfrak{n}_P,\ (g,x)\mapsto g\cdot x.\]
Note that $G\cdot \mathfrak{n}_P$ is the closure of the Richardson nilpotent orbit attached to $P$.\ The previous map
is proper, surjective, generically finite, and is a resolution of $G\cdot\mathfrak{n}_P$ under some conditions on the stabilizer of $e$ in $G$
(see \cite{Borho-Kraft, Borho-MacPherson}). Spaltenstein varieties for $G=SL_n(\mathbb{C})$  arise in various problems such as the study of Springer representations (cf. \cite{Braverman-Gaitsgory,Ginzburg}), crystals (\cite{Malkin}), or quiver varieties (\cite{Nakajima-Duke,Nakajima-JDG}).

Note that $\mathcal{P}_{e,\mathfrak{p}}$ and $\mathcal{P}_{e,\mathfrak{n}_P}$ both coincide with the Springer fiber $\mathcal{B}_e$ in the particular case where $P=B_0$ is a Borel subgroup.\ In general, $\mathcal{P}_{e,\mathfrak{p}}$ and $\mathcal{P}_{e,\mathfrak{n}_P}$ may not coincide (for instance, $\mathcal{P}_{e,\mathfrak{p}}$
is always nonempty, unlike $\mathcal{P}_{e,\mathfrak{n}_P}$).
\\[1mm]
(c) Varieties of the form $\mathcal{P}_{e,\mathfrak{i}}$ also arise as fibers of resolutions of general nilpotent orbit closures.
Let a $\mathbb{Z}$-grading $\mathfrak{g}=\bigoplus_{i\in\mathbb{Z}}\mathfrak{g}_i$
(in particular,  $[\mathfrak{g}_i,\mathfrak{g}_j]\subset\mathfrak{g}_{i+j}$ for all $i,j$) with $\mathfrak{g}_0$ containing the center of $\mathfrak{g}$. Set $\mathfrak{g}_{\geq j}=\bigoplus_{i\geq j}\mathfrak{g}_i$.
There is a parabolic subgroup $P\subset G$ of Lie algebra $\mathfrak{g}_{\geq 0}$.
Let $e\in\mathcal{N}$.\ A grading such that
$e\in\mathfrak{g}_2$ and $\mathfrak{g}^e:=\{x\in\mathfrak{g}:[e,x]=0\}\subset\mathfrak{g}_{\geq 0}$ is said to be {\em good} for $e$ (see \cite{Elashvili-Kac}). Then, one has $G\cdot \mathfrak{g}_{\geq 2}=\overline{G\cdot e}$
and the map
\[G\times_P\mathfrak{g}_{\geq 2}\to \overline{G\cdot e},\ (g,x)\mapsto g\cdot x\]
is proper, surjective, generically finite (see \cite{Broer}), and its fiber over $e'$ is isomorphic to $\mathcal{P}_{e',\mathfrak{g}_{\geq 2}}$.
The map is a resolution for instance if the grading is a {\em Dynkin grading} (i.e., $\mathfrak{g}_i=\{x\in\mathfrak{g}:[h,x]=ix\}$ where $h$ is the semisimple member of a standard triple $(e,h,f)$; see \cite{Panyushev}).

More generally, for every $j\geq 1$, the image $G\cdot\mathfrak{g}_{\geq j}$ is the closure of a nilpotent orbit and the map $G\times_P\mathfrak{g}_{\geq j}\to G\cdot\mathfrak{g}_{\geq j}$ is proper, surjective, of fiber isomorphic to $\mathcal{P}_{e',\mathfrak{g}_{\geq j}}$, and under good conditions this map may be generically finite (see \cite{Broer}).
\end{example}

\begin{remark}
\label{remark-regular}
In the case where $e\in\mathfrak{g}$ is a regular nilpotent element, the variety $\mathcal{P}_{e,\mathfrak{i}}$ consists of at most one point. Indeed, then, there is a unique Borel subalgebra $\mathfrak{b}\subset\mathfrak{g}$ containing $e$. Let $B\subset G$ be the Borel subgroup of Lie algebra $\mathfrak{b}$.
Up to conjugation we may assume that $B\subset P$. If $gP\in\mathcal{P}_{e,\mathfrak{i}}$, then $g^{-1}\cdot e\in \mathfrak{p}$, and we get $g^{-1}\cdot \mathfrak{b}\subset \mathfrak{p}$.
Since $\mathfrak{b},g^{-1}\cdot\mathfrak{b}$ are two Borel subalgebras of $\mathfrak{p}$, there is $p\in P$ such that $pg^{-1}\cdot\mathfrak{b}=\mathfrak{b}$.
Thus $pg^{-1}\in B$ (because $B$ is self-normalizing). Hence $g\in P$. We conclude that $\mathcal{P}_{e,\mathfrak{i}}\subset\{P\}$ and this shows our claim.
\end{remark}

\subsection{Statement of main result}
Let $X$ be an algebraic variety.
A partition of $X$ into subsets that can be indexed $X_1,\ldots,X_k$ so that $X_1\cup\ldots\cup X_l$ is closed for every $l\in\{1,\ldots,k\}$ is called an \textit{$\alpha$-partition} (cf. \cite{DeConcini-Lusztig-Procesi}).
We say that an $\alpha$-partition is a {\em smooth paving} (resp. an {\em affine paving}) if every subset $X_l$ is a smooth subvariety of $X$ (resp. is isomorphic to an affine space $\mathbb{C}^{d_l}$).

The existence of an affine paving guarantees good (co)homological properties for the variety $X$ (see \cite[\S1.6--1.10]{DeConcini-Lusztig-Procesi}).
It is an especially desirable property for varieties that arise as fibers of a resolution (see \cite[\S 2.4, \S 4.3.2]{Juteau-Mautner-Williamson}).

There are classical constructions of affine pavings:

\begin{example}
\label{example-2}
(a) If $B\subset G$ is a Borel subgroup, then the partition of the variety $\mathcal{P}=G/P$ into the various $B$-orbits is an affine paving (the \textit{Schubert decomposition}). \\[1mm]
(b) (Bialynicki-Birula's theorem) Let $X$ be a projective variety equipped with an algebraic action of $\mathbb{C}^*$. This action gives rise to 
a fixed point set $X^{\mathbb{C}^*}\subset X$ and
a well-defined retraction map
\[\xi:X\to X^{\mathbb{C}^*},\ x\mapsto \lim_{t\to 0}t\cdot x,\]
which is not algebraic in general. Let $Y\subset X^{\mathbb{C}^*}$ be a connected component such that every $y\in Y$ is a regular point of $X$. Then the restriction $\xi:\xi^{-1}(Y)\to Y$ is an algebraic affine bundle
(i.e., a fiber bundle of typical fiber isomorphic to an affine space) (see \cite{Bialynicki-Birula}). In particular, if $X$ is a smooth projective variety and the fixed point set $X^{\mathbb{C}^*}$ is finite, then the decomposition $X=\bigsqcup_{y\in X^{\mathbb{C}^*}}\{x\in X:\lim_{t\to 0}t\cdot x=y\}$ is an affine paving (see \cite[\S 1.2--1.3]{DeConcini-Lusztig-Procesi}).
\end{example}

However, showing the existence of an affine paving for the variety $\mathcal{P}_{e,\mathfrak{i}}$ requires other methods: though $\mathcal{P}_{e,\mathfrak{i}}$ is a subset of $\mathcal{P}$, it is not in general stable by any Borel subgroup, so it is not a union of Schubert cells. Moreover, since $\mathcal{P}_{e,\mathfrak{i}}$ is not smooth in general and since there is in general no known action of $\mathbb{C}^*$ on $\mathcal{P}_{e,\mathfrak{i}}$ with a finite number of fixed points, Bialynicki-Birula's theorem cannot be directly applied.

In some situations, it is already known that Springer fibers admit affine pavings.
This was first shown by N.~Spaltenstein \cite[\S II.5]{Spaltenstein-book} for Springer fibers in the case of $G=SL_n(\mathbb{C})$. Other constructions of affine pavings for Springer fibers in the case of $SL_n(\mathbb{C})$ are given in \cite{Fresse-jalgebra,Nakajima-JDG,Xi}.
In fact, C.~De Concini, G.~Lusztig and C.~Procesi \cite{DeConcini-Lusztig-Procesi} showed that $\mathcal{B}_e$ admits an affine paving whenever $G$ is a classical simple algebraic group (see also \cite[\S11]{Jantzen}). The same property holds when $G$ is of type $G_2$, $F_4$, or $E_6$
(see \cite{DeConcini-Lusztig-Procesi,Spaltenstein-E6,Xi}) and one can expect it also for $G$ of type $E_7$ or $E_8$.
Finally, it is known that Steinberg varieties and Spaltenstein varieties admit affine pavings in the case of $G=SL_n(\mathbb{C})$ (see \cite{Brundan-Ostrik,Shimomura}).

In this paper, we extend the previous results to the case of the varieties $\mathcal{P}_{e,\mathfrak{i}}$ and for $G$ classical:

\begin{theorem}
\label{theorem-1}
Let $P,e,\mathfrak{i}$ be as in Definition \ref{definition-1}.
Assume that the minimal Levi subalgebra of $\mathfrak{g}$ containing $e$ has no nonregular 
component of exceptional type. Then, the variety $\mathcal{P}_{e,\mathfrak{i}}$ admits an affine paving.
\end{theorem}

Let us explain the assumption made in the theorem.
For $e\in\mathfrak{g}$ nilpotent, there is a minimal Levi subalgebra $\hat{\mathfrak{g}}\subset\mathfrak{g}$ containing $e$, which is unique up to conjugation.
Its semisimple part decomposes as $[\hat{\mathfrak{g}},\hat{\mathfrak{g}}]=\mathfrak{s}_1\times\cdots\times\mathfrak{s}_k$ with $\mathfrak{s}_i$ simple Lie algebras.
We have $e=(e_1,\ldots,e_k)$, where each $e_i$ is a distinguished nilpotent element of $\mathfrak{s}_i$.
By saying in the theorem that $\hat{\mathfrak{g}}$ has no nonregular component of exceptional type, we mean that
$\mathfrak{s}_i$ is of type $A$--$D$ whenever $e_i$ is not regular in $\mathfrak{s}_i$.
This assumption is more general than assuming that the semisimple part of $\mathfrak{g}$ itself has no component of exceptional type. In particular, the theorem is valid when $G$ is $GL_n(\mathbb{C})$, $SL_n(\mathbb{C})$,
$Sp_{2m}(\mathbb{C})$, $SO_{n}(\mathbb{C})$.
In light of Example \ref{example-1}\,(c),
the theorem answers affirmatively Question 4.19 in \cite{Juteau-Mautner-Williamson} for nilpotent orbits of classical type.
I would like to thank Daniel Juteau, Carl Mautner, and Geordie Williamson who informed me of this application of the theorem.


\subsection{Organization of the paper}
The remainder of the paper comprises seven sections and is organized as follows.
Section \ref{section-2} contains preliminary facts on parabolic orbits in partial flag varieties, which will be basic ingredients in the next sections.
In Section \ref{section-3}, we construct a smooth paving of the variety $\mathcal{P}_{e,\mathfrak{i}}$
and show that $\mathcal{P}_{e,\mathfrak{i}}$ will have an affine paving provided that its fixed point set $(\mathcal{P}_{e,\mathfrak{i}})^S$ under a certain torus $S$ does.
The main fact pointed out in Section \ref{section-4} is that the proof of Theorem \ref{theorem-1} can be reduced to the case of distinguished nilpotent elements.
Moreover, in Section \ref{section-5}, we explain how the proof can be reduced to the case of almost simple groups.

In Section \ref{section-6}, we recall the description of partial flag varieties in the classical cases, that is, in terms of partial flags (in type $A$) and isotropic partial flags (in types $B$, $C$, $D$).
In Section \ref{section-7}, we describe the form taken by the $P$-stable subspaces $\mathfrak{i}\subset\mathfrak{p}$ in the classical cases.
The conclusion of Sections \ref{section-6}--\ref{section-7} is that the variety $\mathcal{P}_{e,\mathfrak{i}}$ takes an elementary form in the classical cases.

Finally, Section \ref{section-8} contains a proof by induction of Theorem \ref{theorem-1} for distinguished nilpotent elements and almost simple classical groups, relying on the elementary form of the variety $\mathcal{P}_{e,\mathfrak{i}}$ in this situation.
This final argument is easy in type $A$ but quite involved in the other classical types.

The proof of Theorem \ref{theorem-1} that we give here is widely inspired by the proof given in \cite{DeConcini-Lusztig-Procesi} in the case of Springer fibers.
In Sections \ref{section-3}--\ref{section-5}, the arguments follow the same scheme as in \cite{DeConcini-Lusztig-Procesi}.
The final computational argument given in Section \ref{section-8} is however more involved here than in the case of Springer fibers.

In what follows, unless otherwise specified, $G$ is a reductive connected linear algebraic group.

\section{Preliminaries on parabolic orbits}

\label{section-2}

In this section, we recall (for later use) elementary properties of parabolic subgroups $Q\subset G$ and well-known properties of $Q$-orbits of the partial flag variety $\mathcal{P}=G/P$.
Proofs are provided for the sake of completeness.

\subsection{Parabolic subgroups, cocharacters, and $\mathbb{Z}$-gradings}
\label{section-2-1}
Recall that a parabolic subgroup $Q\subset G$ admits a Levi decomposition \[Q=L_Q\ltimes U_Q\]
(with $U_Q\subset Q$ the unipotent radical and $L_Q\subset Q$ a Levi factor).

\subsubsection{}
\label{section-2-1-1}
A (Levi decomposition of a) parabolic subgroup can always be induced by a cocharacter: given a Levi decomposition as above, we can find a cocharacter $\lambda:\mathbb{C}^*\to G$ (that is, a morphism of algebraic groups)  such that $Q$, $L_Q$, $U_Q$  are characterized by:
\begin{eqnarray}
\label{equation-1}
& & Q=\{g\in G:\lim_{t\to 0}\lambda(t)g\lambda(t)^{-1}\mbox{ exists}\},  \\[1mm]
\label{equation-2}
& & U_Q=\{g\in G:\lim_{t\to 0}\lambda(t)g\lambda(t)^{-1}=1_G\}, \\[1mm]
\label{equation-3}
& & L_Q=\{g\in G:\lambda(t)g\lambda(t)^{-1}=g\ \ \forall t\in\mathbb{C}^*\}
\end{eqnarray}
(cf. \cite[\S 8.4]{Springer-book}).

\subsubsection{} \label{section-2-1-2}
A (Levi decomposition of a) parabolic subgroup can also be induced by a $\mathbb{Z}$-grading
\[\mathfrak{g}=\bigoplus_{i\in\mathbb{Z}}\mathfrak{g}_i.\]
Here, $[\mathfrak{g}_i,\mathfrak{g}_j]\subset\mathfrak{g}_{i+j}$ for all $i,j$. We assume that $\mathfrak{g}_0$ contains the center of $\mathfrak{g}$ (by convention, all gradings in the rest of the paper will be subject to this assumption).
Write $\mathfrak{g}_{\geq j}=\bigoplus_{i\geq j}\mathfrak{g}_i$.\ Then $\mathfrak{g}_{\geq 0}\subset\mathfrak{g}$ is a parabolic subalgebra with Levi decomposition $\mathfrak{g}_{\geq 0}=\mathfrak{g}_0\oplus\mathfrak{g}_{\geq 1}$,
and there is a parabolic subgroup $Q\subset G$ with Levi decomposition $Q=L_Q\ltimes U_Q$ such that the Lie algebras of $Q,U_Q,L_Q$ are respectively $\mathfrak{g}_{\geq 0}$, $\mathfrak{g}_{\geq 1}$, $\mathfrak{g}_0$.

The map $d:\mathfrak{g}\to\mathfrak{g}$ defined by $d(x)=ix$ for $x\in\mathfrak{g}_i$ is a derivation of $\mathfrak{g}$, which restricts to a derivation of $[\mathfrak{g},\mathfrak{g}]$.
Since any derivation of $[\mathfrak{g},\mathfrak{g}]$ is inner and the center of $\mathfrak{g}$ lies in $\mathfrak{g}_0$, there is $h\in[\mathfrak{g},\mathfrak{g}]$ such that $[h,x]=ix$ for all $x\in\mathfrak{g}_i$, $i\in\mathbb{Z}$, and we find a cocharacter $\lambda:\mathbb{C}^*\to G$ with $\lambda'(t)=th$, so
\begin{equation}
\label{equation-4}
\mathfrak{g}_i=\{x\in\mathfrak{g}:\lambda(t)\cdot x=t^ix\ \ \forall t\in\mathbb{C}^*\}\quad\forall i\in\mathbb{Z}.
\end{equation}
Clearly, $Q,U_Q,L_Q$ correspond to the cocharacter $\lambda$ in the sense of relations (\ref{equation-1})--(\ref{equation-3}).

\subsubsection{Basic setting}
\label{section-basic-setting}
We will often consider the following situation, which combines the previous remarks:
\begin{itemize}
\item[(a)] $Q=L_Q\ltimes U_Q$ is a Levi decomposition of a parabolic subgroup of $G$.
\item[(b)] $\lambda:\mathbb{C}^*\to G$ is a cocharacter inducing this Levi decomposition, in the sense of (\ref{equation-1})--(\ref{equation-3}). Let $S=\{\lambda(t):t\in\mathbb{C}^*\}$, so that $L_Q$ is the centralizer of $S$ in $G$.
\item[(c)] $\mathfrak{g}=\bigoplus_{i\in\mathbb{Z}}\mathfrak{g}_i$ is the $\mathbb{Z}$-grading corresponding to $\lambda$ in the sense of (\ref{equation-4}).
\end{itemize}

\subsection{Parabolic orbits of a partial flag variety}

A parabolic subgroup $Q\subset G$ acts on the partial flag variety $\mathcal{P}=G/P$ with finitely many orbits.
In what follows, we describe the structure of these orbits.

Let $Q=L_Q\ltimes U_Q$, $\lambda:\mathbb{C}^*\to G$, and $S=\{\lambda(t):t\in\mathbb{C}^*\}$ be as in Section \ref{section-basic-setting}. In particular the cocharacter $\lambda$ gives rise to an algebraic action of $\mathbb{C}^*$
on $\mathcal{P}$. Since $\mathcal{P}$ is smooth, projective, we obtain a map
\[\rho:\mathcal{P}\to\mathcal{P}^S:=\{gP\in\mathcal{P}:s(gP)=gP,\ \forall s\in S\},\ gP\mapsto \lim_{t\to 0}\lambda(t)gP\]
which, by Bialynicki-Birula's theorem (cf. Example \ref{example-2}\,(b)), is an algebraic affine bundle over each connected component of the fixed point set $\mathcal{P}^S$. Proposition~\ref{proposition-1} below gives a different proof of this property, and shows in addition that $\rho$ is locally trivial over each connected component of $\mathcal{P}^S$ and is intimately related to the structure of the $Q$-orbits.

Given a $Q$-orbit $\mathcal{O}\subset \mathcal{P}$, we let $\mathcal{O}^S:=\{gP\in\mathcal{O}:s(gP)=gP,\ \forall s\in S\}$
be its $S$-fixed point set.
The structure of $\mathcal{O}^S$ is described in the next lemma.

\begin{lemma}
\label{lemma-1}
\begin{itemize}
\item[(a)] One has $\mathcal{O}^S\not=\emptyset$.
\end{itemize}
Fix an element $g_0P\in\mathcal{O}^S$.\ Hence $P_0:=g_0Pg_0^{-1}$ contains $S$ and $L_Q\cap P_0$ is a parabolic subgroup of $L_Q$ (cf. \cite[\S 6.4.7]{Springer-book}).
\begin{itemize}
\item[(b)] $\mathcal{O}^S$ consists of a unique $L_Q$-orbit. In fact, the map
\[\xi:L_Q/(L_Q\cap P_0)\to \mathcal{O}^S,\ \ell(L_Q\cap P_0)\mapsto \ell (g_0P)\]
is well defined and is an isomorphism of algebraic varieties.
\item[(c)]
In particular, $\mathcal{O}^S$ is a partial flag variety of $L_Q$ (thus a smooth, connected, projective variety).
Hence, the subsets $\mathcal{O}^S$, corresponding to the $Q$-orbits $\mathcal{O}\subset\mathcal{P}$, are exactly
the connected components of $\mathcal{P}^S$.
\end{itemize}
\end{lemma}



\begin{proof}
(a) Let $T\subset B\subset Q$ be a maximal torus and a Borel subgroup such that $S\subset T$.
Let $g\in G$ be such that $B\subset gPg^{-1}$.
By the Bruhat decomposition, the orbit $\mathcal{O}$ takes the form $\mathcal{O}=QwgP$ with $w\in N_G(T)$. For all $s\in S$, one has $swgP=w(w^{-1}sw)gP=wgP$. Thus, $wgP\in\mathcal{O}^S$.

(b) The inclusion $\{\ell(g_0P):\ell\in L_Q\}\subset\mathcal{O}^S$ is easy. Conversely, suppose $q(g_0P)\in\mathcal{O}^S$ (where $q\in Q$). Write $q=\ell u$ with $\ell\in L_Q$ and $u\in U_Q$. Since $g_0P$ and $q(g_0P)$ are fixed by $S$ (and using (\ref{equation-3})), we deduce the equality
\[q(g_0P)=\ell\lambda(t)u\lambda(t)^{-1}(g_0P)\quad\forall t\in\mathbb{C}^*.\]
Finally letting $t\to 0$ (and using (\ref{equation-2})) we infer that $q(g_0P)=\ell(g_0P)$. This shows the equality $\mathcal{O}^S=\{\ell(g_0P):\ell\in L_Q\}$.

This equality implies that the map $L_Q\to\mathcal{O}^S$, $\ell\to\ell(g_0P)$ is surjective.
Note that $L_Q\cap P_0=\{\ell\in L_Q:\ell(g_0P)=g_0P\}$.
This readily implies that
 $\xi$ is a well-defined isomorphism (see \cite[\S 2.11]{Steinberg-book} for instance).
\end{proof}

As in Lemma \ref{lemma-1}, we fix $g_0P\in\mathcal{O}^S$ and let $P_0=g_0Pg_0^{-1}$. In addition to the isomorphism
$\xi:L_Q/(L_Q\cap P_0)\to \mathcal{O}^S$
of Lemma \ref{lemma-1}\,(b), we dispose of the surjective maps
\begin{eqnarray*}
 & & \varphi:Q=L_Q\ltimes U_Q\to L_Q/(L_Q\cap P_0),\ \ell u\mapsto\ell(L_Q\cap P_0) \\[1mm]
 & \mbox{and} & \psi:Q\to\mathcal{O},\ q\mapsto q(g_0P).
\end{eqnarray*}

\begin{proposition}
\label{proposition-1}
\begin{itemize}
\item[(a)] There is a (unique)  map $\zeta:\mathcal{O}\to\mathcal{O}^S$ such that the diagram
\[\begin{array}{ccl}
\ \ Q & \stackrel{\psi}{\longrightarrow} & \mathcal{O} \\
\mbox{\scriptsize $\varphi$}\downarrow & & \downarrow \mbox{\scriptsize $\zeta$} \\
L_Q/(L_Q\cap P_0) & \stackrel{\xi}{\longrightarrow} & \mathcal{O}^S
\end{array}\]
commutes. Moreover, $\zeta$ is an algebraic locally trivial affine bundle.
\item[(b)] $\zeta$ is the restriction to $\mathcal{O}$ of the map $\rho:\mathcal{P}\to\mathcal{P}^S$, $gP\mapsto \lim_{t\to 0}\lambda(t)gP$.
In particular,
\begin{itemize}
\item[$\bullet$] $\zeta$ is intrinsic (i.e., it does not depend on the choice of $g_0P\in\mathcal{O}^S$),
\item[$\bullet$] $\rho$ is an algebraic locally trivial affine bundle over each connected component of $\mathcal{P}^S$.
\end{itemize}
\end{itemize}
\end{proposition}

\begin{proof}
(a) Observe that we have $Q\cap P_0=\{q\in Q:q(g_0P)=g_0P\}$.\ Hence
the map $\psi:Q\to\mathcal{O}$ induces an isomorphism of algebraic varieties $\psi_1:Q/(Q\cap P_0)\to\mathcal{O}$.

We claim that
\[Q\cap P_0=(L_Q\cap P_0)\ltimes(U_Q\cap P_0).\]
It is enough to check the inclusion $Q\cap P_0\subset (L_Q\cap P_0)(U_Q\cap P_0)$. So, take $q\in Q\cap P_0$.\ There are $\ell\in L_Q$, $u\in U_Q$ such that $q=\ell u$. Since $q\in P_0$, we have $q(g_0P)=g_0P$. Furthermore, recall that $g_0P\in\mathcal{O}^S$. It follows (by (\ref{equation-3})):
\[g_0P=\lambda(t)(g_0P)=\lambda(t)q(g_0P)=\ell(\lambda(t)u\lambda(t)^{-1})(g_0P)\quad \forall t\in\mathbb{C}^*.\]
Letting $t\to 0$ and invoking (\ref{equation-2}), we get $g_0P=\ell(q_0P)$.\ Whence $\ell\in L_Q\cap P_0$. We derive $u\in U_Q\cap P_0$. Thereby, $q\in(L_Q\cap P_0)(U_Q\cap P_0)$ and the claim is established.

The relations $Q=L_Q\ltimes U_Q$ and $Q\cap P_0=(L_Q\cap P_0)\ltimes(U_Q\cap P_0)$ yield a natural isomorphism
\[\zeta_1:Q/(Q\cap P_0)\stackrel{\sim}{\to}L_Q\times_{L_Q\cap P_0}(U_Q/(U_Q\cap P_0)).\]
We get a commutative diagram
\[
\begin{array}{rcccccl}
& & Q & \stackrel{\psi_2}{\longrightarrow} & Q/(Q\cap P_0) & \stackrel{\psi_1}{\longrightarrow} & \mathcal{O} \\
& & \mbox{\scriptsize $\varphi$}\downarrow & & \downarrow\mbox{\scriptsize $\zeta_1$}\\
\mathcal{O}^S & \stackrel{\xi}{\longleftarrow} & L_Q/(L_Q\cap P_0) & \stackrel{\zeta_2}{\longleftarrow} & L_Q\times_{L_Q\cap P_0}(U_Q/(U_Q\cap P_0))
\end{array}
\]
where  $\psi_2:Q\to Q/(Q\cap P_0)$ and $\zeta_2:L_Q\times_{L_Q\cap P_0}(U_Q/(U_Q\cap P_0))\to L_Q/(L_Q\cap P_0)$ are the natural surjections. Let $\zeta=\xi\zeta_2\zeta_1\psi_1^{-1}$.
Recall that the maps $\psi_1,\zeta_1,\xi$ involved in this composition are isomorphisms of varieties.
Note that $\zeta_2$ is a locally trivial fiber bundle whose typical fiber $U_Q/(U_Q\cap P_0)$ is isomorphic (as a variety) to an affine space.
Therefore, $\zeta$ is an algebraic locally trivial affine bundle.

(b) Let $q(g_0P)\in \mathcal{O}$. There are $\ell\in L_Q$, $u\in U_Q$ such that $q=\ell u$. Then,
\[\rho(q(g_0P))=\lim_{t\to 0}\lambda(t)q(g_0P)=\lim_{t\to 0}\ell(\lambda(t)u\lambda(t)^{-1})(g_0P)=\ell(g_0P)=\zeta(q(g_0P)).\]
(cf. (\ref{equation-2}), (\ref{equation-3})).
Thus $\zeta=\rho|_\mathcal{O}$. This implies that $\zeta$ does not depend on the choice of $g_0P$.
Combined with Lemma \ref{lemma-1}\,(c), this implies that $\rho$ is an algebraic locally trivial affine bundle over each connected component of $\mathcal{P}^S$.
\end{proof}

\section{Construction of smooth pavings}

\label{section-3}

The purpose of this section is to show the next statement, which will be used in the proof of Theorem \ref{theorem-1}:

\begin{proposition}
\label{proposition-section-3}
Let $Q,S$ be as in Section \ref{section-basic-setting}\,(a)--(b) and let $e,P,\mathfrak{i}$ be as in Definition~\ref{definition-1}.
Assume that the grading of Section \ref{section-basic-setting}\,(c) is good for $e$.
Let $\mathcal{O}\subset\mathcal{P}$ be a $Q$-orbit such that $\mathcal{P}_{e,\mathfrak{i}}\cap\mathcal{O}\not=\emptyset$. Then:

\begin{itemize}
\item[(a)] $\mathcal{P}_{e,\mathfrak{i}}\cap\mathcal{O}$ is smooth.

\item[(b)] $(\mathcal{P}_{e,\mathfrak{i}}\cap\mathcal{O})^S$ is nonempty and is a smooth, projective variety.

\item[(c)] The restriction of the map $\zeta:\mathcal{O}\to\mathcal{O}^S$ of Proposition \ref{proposition-1} yields a well-defined map 
\[\mathcal{P}_{e,\mathfrak{i}}\cap\mathcal{O}\to (\mathcal{P}_{e,\mathfrak{i}}\cap\mathcal{O})^S\]
whose restriction over each connected component of $(\mathcal{P}_{e,\mathfrak{i}}\cap\mathcal{O})^S$
is an algebraic affine bundle.
\end{itemize}
\end{proposition}

The assumptions of the proposition involve the notion of good grading (see Example \ref{example-1}\,(c)), which is recalled in Section \ref{section-3-1}. The choice of $Q$ arising from a good grading (like in the proposition) is suitable for applying the results of Section \ref{section-2} to the study of the variety $\mathcal{P}_{e,\mathfrak{i}}$.
Relying on this observation,
Proposition \ref{proposition-section-3} is proved in Section \ref{section-3-2}. The proof of Proposition \ref{proposition-section-3}\,{\rm (a)},\,{\rm (b)} follows the same reasoning as in \cite{DeConcini-Lusztig-Procesi},
the proof of Proposition \ref{proposition-section-3}\,{\rm (c)} is somewhat different.

\subsection{Preliminaries on good gradings}

\label{section-3-1}

Recall (cf. Example \ref{example-1}\,(c)) that a $\mathbb{Z}$-grading $\mathfrak{g}=\bigoplus_{i\in\mathbb{Z}}\mathfrak{g}_i$, such that $\mathfrak{g}_0$ contains the center of $\mathfrak{g}$, is said to be good for the nilpotent element $e\in\mathfrak{g}$ if
\begin{eqnarray}
& & e\in\mathfrak{g}_2 \label{nequation-4} \\
& \mbox{and} & \mathfrak{g}^e:=\{x\in\mathfrak{g}:[e,x]=0\}\subset\mathfrak{g}_{\geq 0}. \label{nequation-5}
\end{eqnarray}
(We write $\mathfrak{g}_{\geq j}=\bigoplus_{i\geq j}\mathfrak{g}_i$.) Good gradings always exist:

\begin{example}
\label{example-3}
Take $h,f\in\mathfrak{g}$ such that $(e,h,f)$ is a standard triple (i.e., $[h,e]=2e$, $[h,f]=-2f$, $[e,f]=h$).
Then, the grading obtained by letting $\mathfrak{g}_i=\{x\in\mathfrak{g}:[h,x]=ix\}$ is good for $e$.
\end{example}

There are equivalent definitions of good gradings (see \cite[Theorem 1.3]{Elashvili-Kac}):

\begin{lemma}
Let $\mathfrak{g}=\bigoplus_{i\in\mathbb{Z}}\mathfrak{g}_i$ be a $\mathbb{Z}$-grading such that $\mathfrak{g}_0$ contains the center of $\mathfrak{g}$ and $e\in\mathfrak{g}_2$. Then, the following conditions are equivalent: \\
{\rm (i)} the grading is good for $e$; \\
{\rm (ii)} $\mathrm{ad}\,e:\mathfrak{g}_i\to\mathfrak{g}_{i+2}$ is injective for all $i\leq -1$ and surjective for all $i\geq -1$; \\
{\rm (iii)} $\mathrm{ad}\,e:\mathfrak{g}_i\to\mathfrak{g}_{i+2}$ is injective for all $i\leq -1$; \\
{\rm (iv)} $\mathrm{ad}\,e:\mathfrak{g}_i\to\mathfrak{g}_{i+2}$ is surjective for all $i\geq -1$.
\end{lemma}

We refer to \cite{Elashvili-Kac} for the main properties of good gradings and classification of good gradings for simple Lie algebras. We just emphasize in the next lemma the two properties that we will need in our study of the varieties $\mathcal{P}_{e,\mathfrak{i}}$.

\begin{lemma}
\label{lemma-2}
Let $Q,S,\lambda$ be as in Section \ref{section-basic-setting}\,(a)--(b).
Assume that the grading of Section \ref{section-basic-setting}\,(c) is good for $e$.
Then:
\begin{itemize}
\item[(a)] $\overline{Q\cdot e}=\mathfrak{g}_{\geq 2}$. In particular, $\overline{Q\cdot e}$ is a vector subspace of $\mathfrak{g}$.
\item[(b)] $\lambda(t)\cdot e=t^2e$ for all $t\in\mathbb{C}^*$.\ In particular, the variety $\mathcal{P}_{e,\mathfrak{i}}$ is stable by the natural action of $S$ on $\mathcal{P}$.
\end{itemize}
\end{lemma}

\begin{proof}
(a)
It follows from Section \ref{section-basic-setting} and relation (\ref{nequation-4}) that $\overline{Q\cdot e}$ is a closed subvariety of $\mathfrak{g}_{\geq 2}$.
By $\mathfrak{sl}_2$-theory, we have $\dim \mathfrak{g}^e=\dim\mathfrak{g}_0+\dim\mathfrak{g}_1$.
By relation (\ref{nequation-5}), we have $\mathfrak{g}^e\subset\mathfrak{g}_{\geq 0}=Lie(Q)$. This yields
$\dim \overline{Q\cdot e}=\dim\mathfrak{g}_{\geq 0}-\dim\mathfrak{g}^e=\dim\mathfrak{g}_{\geq 2}$. Whence the claimed equality. Part (b) follows from (\ref{equation-4}) and (\ref{nequation-4}).
\end{proof}

\subsection{Proof of Proposition \ref{proposition-section-3}}

\label{section-3-2}

Let $gP\in\mathcal{P}_{e,\mathfrak{i}}\cap\mathcal{O}$.
Thus $g^{-1}\cdot e\in\mathfrak{i}$. Lemma \ref{lemma-2}\,(a) guarantees that the intersection $(g\cdot\mathfrak{i})\cap(Q\cdot e)$ is a smooth, closed, irreducible subvariety of $Q\cdot e$.
The maps $\psi:Q\to \mathcal{O}$, $q\mapsto q(gP)$ and $\chi:Q\to Q\cdot e$, $q\mapsto q^{-1}\cdot e$ are algebraic, smooth. Moreover, one has
\[\psi^{-1}(\mathcal{P}_{e,\mathfrak{i}}\cap\mathcal{O})=\chi^{-1}((g\cdot\mathfrak{i})\cap(Q\cdot e))=\{q\in Q:q^{-1}\cdot e\in g\cdot\mathfrak{i}\}.\]
Since $(g\cdot\mathfrak{i})\cap(Q\cdot e)$ is smooth, we infer that $\mathcal{P}_{e,\mathfrak{i}}\cap\mathcal{O}$ is smooth.
This shows (a).

Let $L_Q,U_Q$ be as in Section \ref{section-basic-setting}.
By Lemma \ref{lemma-1}, we have
$\mathcal{O}=Qg_0P$ and $\mathcal{O}^S=\{\ell(g_0P):\ell\in L_Q\}$, where $P_0:=g_0Pg_0^{-1}\supset S$.
The map $\zeta:\mathcal{O}\to\mathcal{O}^S$ of Proposition \ref{proposition-1} is such that
\[\zeta(\ell u(g_0P))=\ell(g_0P)\mbox{ \ for all $\ell\in L_Q$, $u\in U_Q$}.\]

Let $q(g_0P)\in\mathcal{P}_{e,\mathfrak{i}}\cap\mathcal{O}$. Write $q=\ell u$ with $\ell\in L_Q$, $u\in U_Q$. Invoking (\ref{equation-3}) and Lemma \ref{lemma-2}\,(b), we have
\[\ell(\lambda(t)u\lambda(t)^{-1})g_0P=\lambda(t)q(g_0P)\in\mathcal{P}_{e,\mathfrak{i}}\quad\forall t\in\mathbb{C}^*.\]
Letting $t\to 0$ and invoking (\ref{equation-2}), we get $\zeta(q(g_0P))=\ell(g_0P)\in(\mathcal{P}_{e,\mathfrak{i}}\cap\mathcal{O})^S$.
Whence the inclusion $\zeta(\mathcal{P}_{e,\mathfrak{i}}\cap\mathcal{O})\subset(\mathcal{P}_{e,\mathfrak{i}}\cap\mathcal{O})^S$.
The restriction of $\zeta$ to $\mathcal{O}^S$ being the identity, we obtain in fact
\[\zeta(\mathcal{P}_{e,\mathfrak{i}}\cap\mathcal{O})=(\mathcal{P}_{e,\mathfrak{i}}\cap\mathcal{O})^S.\]
In particular, $(\mathcal{P}_{e,\mathfrak{i}}\cap\mathcal{O})^S\not=\emptyset$.
Since $\mathcal{P}_{e,\mathfrak{i}}\cap\mathcal{O}$ is smooth, it follows that $(\mathcal{P}_{e,\mathfrak{i}}\cap\mathcal{O})^S$ is smooth. By Lemma \ref{lemma-1}\,{\rm (c)}, $(\mathcal{P}_{e,\mathfrak{i}}\cap \mathcal{O})^S=\mathcal{P}_{e,\mathfrak{i}}\cap \mathcal{O}^S$ is also projective, whence {\rm (b)}.

Let
\[\zeta_0:\mathcal{P}_{e,\mathfrak{i}}\cap\mathcal{O}\to(\mathcal{P}_{e,\mathfrak{i}}\cap\mathcal{O})^S\]
be the restriction of $\zeta$. For proving part {\rm (c)}, we need to show that, for every connected component $C\subset (\mathcal{P}_{e,\mathfrak{i}}\cap\mathcal{O})^S$, the restriction of $\zeta_0$ gives rise to an algebraic affine bundle $\zeta_0^{-1}(C)\to C$. To do this, we aim to apply Bialynicki-Birula's theorem (see Example \ref{example-2}\,{\rm (b)}) to the map
\[\rho_0:\mathcal{P}_{e,\mathfrak{i}}\cap \overline{\mathcal{O}}\to(\mathcal{P}_{e,\mathfrak{i}}\cap\overline{\mathcal{O}})^S,\ x\mapsto \lim_{t\to 0}\lambda(t)x.\] 
Recall that the orbit $\mathcal{O}$ is open in its closure
 $\overline{\mathcal{O}}$,
and $\overline{\mathcal{O}}\setminus\mathcal{O}$ is a union of finitely many $Q$-orbits of lower dimension.
It easily follows that $C$ is also a connected component of $(\mathcal{P}_{e,\mathfrak{i}}\cap\overline{\mathcal{O}})^S$.
Moreover, $C$ is contained in $\mathcal{P}_{e,\mathfrak{i}}\cap\mathcal{O}$, which is a smooth, open subset of $\mathcal{P}_{e,\mathfrak{i}}\cap\overline{\mathcal{O}}$, thereby $C$ lies in the regular locus of $\mathcal{P}_{e,\mathfrak{i}}\cap\overline{\mathcal{O}}$.
This allows us to apply Bialynicki-Birula's theorem, from which we obtain that the restriction of $\rho_0$ is an algebraic affine bundle $\rho_0^{-1}(C)\to C$.

According to Proposition \ref{proposition-1},
for every $Q$-orbit $\mathcal{O}'\subset\overline{\mathcal{O}}$, we have $\rho_0(\mathcal{P}_{e,\mathfrak{i}}\cap\mathcal{O}')\subset (\mathcal{P}_{e,\mathfrak{i}}\cap\mathcal{O}')^S$.
It follows that $\rho_0^{-1}((\mathcal{P}_{e,\mathfrak{i}}\cap\mathcal{O})^S)=\mathcal{P}_{e,\mathfrak{i}}\cap\mathcal{O}$. From Proposition \ref{proposition-1}\,{\rm (b)}, we also have that the restriction $\rho_0:\mathcal{P}_{e,\mathfrak{i}}\cap\mathcal{O}\to (\mathcal{P}_{e,\mathfrak{i}}\cap\mathcal{O})^S$
coincides with $\zeta_0$. In particular, we have $\zeta_0^{-1}(C)=\rho_0^{-1}(C)$, and the restriction $\zeta_0|_{\zeta_0^{-1}(C)}$
is then an algebraic affine bundle.
The proof of Proposition \ref{proposition-section-3} is complete.

\section{Reduction to distinguished case}

\label{section-4}

Through Bala-Carter theory, one attaches to a nilpotent element $e\in\mathfrak{g}$ a Levi subgroup $\hat{G}\subset G$ whose Lie algebra $\hat{\mathfrak{g}}$ contains $e$ as a distinguished element. The precise notation is given in Section \ref{section-4-1}.
In this section, we combine this classical construction of Bala-Carter theory with the construction of Section \ref{section-3},
which involves a parabolic subgroup $Q$ with Levi factor $Z_G(S)$, those data arising from a good grading for $e$
(see Section \ref{section-3-1}).
In order to make both constructions compatible, we will assume that $\hat{G}$ contains the subtorus $S$
(one always can find $\hat{G}$ with this property: see Section~\ref{section-4-1});
in this manner, $Q$ contains the center of $\hat{G}$,
so that $\hat{G}\cap Q$ is a parabolic subgroup of $\hat{G}$
and the datum $(\hat{G}\cap Q,\lambda,S)$ corresponds to a good grading of $\hat{G}$.

Our objective in this section is to prove the next proposition, which relates the variety
$\mathcal{P}_{e,\mathfrak{i}}$ to a variety called $\hat{\mathcal{P}}_{e,\hat{\mathfrak{i}}}$\,,
of the same type but relative to the group $\hat{G}$.

\begin{proposition}
\label{proposition-section-4}
Let $e,P,\mathfrak{i}$ be as in Definition~\ref{definition-1}.
Let $Q,\lambda,S$ be as in Section \ref{section-basic-setting} and assume that the grading of Section \ref{section-basic-setting}\,(c) is good for $e$.
Let $\hat{G}\subset G$ be a Levi subgroup whose Lie algebra contains $e$ as a distinguished element and assume that $S\subset\hat{G}$.
Let $\hat{Z}$ be the identity component of the center of $\hat{G}$, so that $\hat{G}=Z_G(\hat{Z})$. Hence $S\hat{Z}$ is a subtorus of both $Q$ and $\hat{G}$. Then: \\
{\rm (a)} There is a map
\[(\mathcal{P}_{e,\mathfrak{i}})^S\to(\mathcal{P}_{e,\mathfrak{i}})^{S \hat{Z}}\]
which is an algebraic affine bundle over each connected component. \\
{\rm (b)}
For each connected component $\mathcal{C}\subset \mathcal{P}^{\hat{Z}}$, there are\begin{itemize}
\item a parabolic subgroup $\hat{P}\subset\hat{G}$ and a $\hat{P}$-stable subspace $\hat{\mathfrak{i}}\subset Lie(\hat{P})$,
giving rise to a variety
\[\hat{\mathcal{P}}_{e,\hat{\mathfrak{i}}}=\{g\hat{P}\in\hat{G}/\hat{P}:g^{-1}\cdot e\in\hat{\mathfrak{i}}\}\subset\hat{G}/\hat{P},\]
\item and an isomorphism
\[\mathcal{C}\cap(\mathcal{P}_{e,\mathfrak{i}})^{S \hat{Z}}\to(\hat{\mathcal{P}}_{e,\hat{\mathfrak{i}}})^S.\]
\end{itemize}
In particular, if we know that the variety $(\hat{\mathcal{P}}_{e,\hat{\mathfrak{i}}})^S$ admits an affine paving for all choices of $\hat{P},\hat{\mathfrak{i}}$, then we can conclude that the variety $(\mathcal{P}_{e,\mathfrak{i}})^S$ itself admits an affine paving.
\end{proposition}

Hereafter we fix $e,P,\mathfrak{i}$ like in the proposition.
The group $\hat{G}$ involved in the statement of the proposition is unique up to conjugation.
In Section \ref{section-4-1} we review the construction of $\hat{G}$
and fix the notation.
The proof of the proposition is then given in Section \ref{section-4-2}.

\subsection{Notation and preliminary facts}

\label{section-4-1}

To start with, we recall the construction of a minimal Levi subalgebra of $\mathfrak{g}$ containing the nilpotent element $e$.
We refer to \cite[\S3 and \S8]{Collingwood-McGovern} and \cite[\S4]{Jantzen} for more details.

First, we embed $e$ in a standard triple $\phi=(e,h,f)$, so that $[h,e]=2e$, $[h,f]=-2f$, and $[e,f]=h$.
Then,
\begin{eqnarray*}
 & & \mathfrak{z}_{\mathfrak{g}}(\phi):=\{x\in\mathfrak{g}:[x,y]=0,\ \forall y\in\{e,h,f\}\}, \\[1mm]
 & & Z_G(\phi):=\{g\in G:g\cdot y=y,\ \forall y\in\{e,h,f\}\}
\end{eqnarray*}
are a reductive Lie subalgebra of $\mathfrak{g}$ (the quotient of $\mathfrak{z}_{\mathfrak{g}}(e):=\{x\in\mathfrak{g}:[x,e]=0\}$
by its nilradical) and a reductive subgroup of $G$, respectively.
Let $Z_G(\phi)^0$ denote the identity component.
Pick up a maximal torus $T_\phi\subset Z_G(\phi)^0$. Thus $\mathfrak{t}_\phi:=Lie(T_\phi)$
is a maximal toral subalgebra of $\mathfrak{z}_{\mathfrak{g}}(\phi)$.
Let $Z_G(T_\phi)\subset G$ (resp. $\mathfrak{z}_{\mathfrak{g}}(\mathfrak{t}_\phi)\subset\mathfrak{g}$)
be the corresponding centralizers.

\begin{lemma}
\label{lemma-4-1}
$\hat{G}:=Z_G(T_\phi)$ is a Levi subgroup of $G$ whose Lie algebra
$\hat{\mathfrak{g}}:=\mathfrak{z}_{\mathfrak{g}}(\mathfrak{t}_\phi)$ contains $e$ as a distinguished element.
\end{lemma}

The next lemma shows how to adapt the construction of the group $\hat{G}$ to the parabolic subgroup $Q$ and
the subtorus $S\subset Q$ arising from a good grading for $e$, involved in the constructions of Section \ref{section-3}
and in the statement of Proposition \ref{proposition-section-4}.

\begin{lemma}
\label{lemma-5}
Let $Q,\lambda,S$ be as in Section \ref{section-basic-setting} and assume that the grading of Section \ref{section-basic-setting}\,(c) is good for $e$.
Then, there is a Levi subgroup $\hat{G}\subset G$ whose Lie algebra contains $e$ as a distinguished element
and which satisfies $S\subset\hat{G}$.
\end{lemma}

\begin{proof}
Let $\mathfrak{g}=\bigoplus_{i\in\mathbb{Z}}\mathfrak{g}_i$ be the grading of Section \ref{section-basic-setting}\,(c).
By assumption, $e\in\mathfrak{g}_2$.
By \cite[Lemma 1.1]{Elashvili-Kac}, we can choose a standard triple $\phi=(e,h,f)$ with $h\in\mathfrak{g}_0$ and $f\in\mathfrak{g}_{-2}$.
Let $T_\phi\subset Z_G(\phi)$, $\mathfrak{t}_\phi\subset\mathfrak{z}_{\mathfrak{g}}(\phi)$
be a maximal torus and a maximal toral subalgebra, as above Lemma \ref{lemma-4-1}.
In particular,
\begin{equation}
\label{relation-4-1}
h\in \mathfrak{z}_{\mathfrak{g}}(\mathfrak{t}_\phi).
\end{equation}
Letting $H=\lambda'(1)\in\mathfrak{g}$, relation (\ref{equation-4}) implies $\mathfrak{g}_i=\{x\in\mathfrak{g}:[H,x]=ix\}$ for all $i\in\mathbb{Z}$.
By \cite[Theorem 1.1]{Elashvili-Kac}, $H-h$ lies in the center of $\mathfrak{z}_{\mathfrak{g}}(\phi)$.
This ensures that
\begin{equation}
\label{relation-4-2}
H-h\in\mathfrak{z}_{\mathfrak{g}}(\mathfrak{t}_\phi).
\end{equation}
Combining (\ref{relation-4-1}) and (\ref{relation-4-2}), we obtain $H\in \mathfrak{z}_{\mathfrak{g}}(\mathfrak{t}_\phi)$.
Since $H$ generates the Lie algebra of $S$, this implies
that the tori $S$ and $T_\phi$ commute.
By Lemma \ref{lemma-4-1}, $\hat{G}:=Z_G(T_\phi)$ satisfies the desired properties.
\end{proof}

Let $\hat{G}\subset G$ be a Levi subgroup satisfying the conditions of Lemma \ref{lemma-5}.
Saying that $\hat{G}$ is a Levi subgroup of $G$ means that it arises as a Levi factor of some
parabolic subgroup $\hat{Q}\subset G$.
Let $U_{\hat{Q}}$ be the unipotent radical of $\hat{Q}$, so that we have the Levi decomposition
\[\hat{Q}=\hat{G}\ltimes U_{\hat{Q}}.\]
We fix a cocharacter $\hat{\lambda}:\mathbb{C}^*\to G$ that induces this Levi decomposition in the sense
of Section \ref{section-2-1-1}, that is, such that
\begin{eqnarray*}
& & \hat{Q}=\{g\in G:\lim_{t\to 0}\hat\lambda(t)g\hat\lambda(t)^{-1}\mbox{ exists}\},  \\[1mm]
& & U_{\hat{Q}}=\{g\in G:\lim_{t\to 0}\hat\lambda(t)g\hat\lambda(t)^{-1}=1_G\}, \\[1mm]
& & \hat{G}=\{g\in G:\hat\lambda(t)g\hat\lambda(t)^{-1}=g\}.
\end{eqnarray*}
Let $\hat{S}=\{\hat\lambda(t):t\in\mathbb{C}^*\}$. Thus, $\hat{G}=Z_G(\hat{S})$.
The fact that $S$ is contained in $\hat{G}$ implies that
\begin{equation}
\label{equation-13}
\mbox{the tori $S$ and $\hat{S}$ commute,}
\end{equation}
thus they generate a torus $S\,\hat{S}\subset\hat{G}$.
It also implies that the adjoint action of $\hat{S}$ fixes every element
of the Lie algebra $\hat{\mathfrak{g}}$. In particular,
$s\cdot e=e$ for all $s\in\hat{S}$. Consequently,
\begin{equation}
\label{equation-14}
\mbox{the action of $\hat{S}$ on $\mathcal{P}=G/P$ leaves the subvariety $\mathcal{P}_{e,\mathfrak{i}}$ stable.}
\end{equation}

We will need the following lemma.

\begin{lemma}
\label{new-lemma-6}
We have the following equality between fixed point sets $\mathcal{P}^{\hat{S}}=\mathcal{P}^{\hat{Z}}$.
\end{lemma}

\begin{proof}
We check the equality $\mathcal{O}^{\hat{Z}}=\mathcal{O}^{\hat{S}}$ for all $\hat{Q}$-orbit $\mathcal{O}\subset\mathcal{P}$. The inclusion $\subset$ is immediate. For checking the other inclusion, we first note as in the proof of Lemma \ref{lemma-1}\,{\rm (a)} that the orbit $\mathcal{O}$ contains at least one element $x_0$ fixed by $\hat{Z}$. Then, Lemma \ref{lemma-1}\,{\rm (b)} implies that $\mathcal{O}^{\hat{S}}$ is exactly the $\hat{G}$-orbit of $x_0$. Since $\hat{G}=Z_G(\hat{Z})$, this yields the desired inclusion $\mathcal{O}^{\hat{S}}\subset\mathcal{O}^{\hat{Z}}$.
\end{proof}

\subsection{Proof of Proposition \ref{proposition-section-4}}
\label{section-4-2}

Part {\rm (a)} of the statement is established as follows. From Proposition \ref{proposition-section-3}, we know that the variety $(\mathcal{P}_{e,\mathfrak{i}})^S$ is projective and smooth. Moreover, by (\ref{equation-13}) and (\ref{equation-14}), this variety is stable by the natural action of the rank one torus $\hat{S}=\{\hat{\lambda}(t):t\in\mathbb{C}^*\}$
on $\mathcal{P}$.
Thereby, we can apply Bialynicki-Birula's theorem (see Example \ref{example-2}\,{\rm (b)}) which says that the retraction map
\[
\hat\rho:
(\mathcal{P}_{e,\mathfrak{i}})^S\to ((\mathcal{P}_{e,\mathfrak{i}})^{S})^{\hat{S}},\ \ gP\mapsto \lim_{t\to 0}\hat\lambda(t)gP
\]
is an algebraic affine bundle over each connected component. Using Lemma \ref{new-lemma-6}, we finally note that
$((\mathcal{P}_{e,\mathfrak{i}})^{S})^{\hat{S}}=(\mathcal{P}_{e,\mathfrak{i}})^{S}\cap\mathcal{P}^{\hat{S}}=(\mathcal{P}_{e,\mathfrak{i}})^{S}\cap\mathcal{P}^{\hat{Z}}
=(\mathcal{P}_{e,\mathfrak{i}})^{S\hat{Z}}$. The proof of {\rm (a)} is complete.

Now, let us prove part {\rm (b)} of the statement. Let a connected component $\mathcal{C}\subset\mathcal{P}^{\hat{Z}}$.
From Lemma \ref{lemma-1} (and Lemma \ref{new-lemma-6}), we know that there is $g_0P\in \mathcal{P}^{\hat{S}}$ such that $\mathcal{C}=\hat{G}g_0P$.
The fact that $g_0P$ is $\hat{S}$-fixed implies that the parabolic subgroup $g_0Pg_0^{-1}$ contains the torus $\hat{S}$. Thereby, $\hat{P}:=\hat{G}\cap g_0Pg_0^{-1}$ is a parabolic subgroup of $\hat{G}=Z_G(\hat{S})$. Set $\hat{\mathfrak{i}}=\hat{\mathfrak{g}}\cap (g_0\cdot\mathfrak{i})$. This is clearly a $\hat{P}$-stable subspace of $\hat{\mathfrak{p}}:=Lie(\hat{P})=\hat{\mathfrak{g}}\cap (g_0\cdot\mathfrak{p})$. From Lemma \ref{lemma-1}\,{\rm (b)}, we know that the map
\[
\xi:\hat{G}/\hat{P}\to \mathcal{C},\ \ \hat{g}\hat{P}\mapsto \hat{g}g_0P
\]
is an isomorphism of algebraic varieties. The map $\xi$ is clearly $S$-equivariant, hence it satisfied
\begin{equation}
\label{new-11}
\xi((\hat{G}/\hat{P})^S)=\mathcal{C}^S.
\end{equation}
Given $\hat{g}\hat{P}\in \hat{G}/\hat{P}$, we have
\begin{equation}
\label{new-12}\hat{g}\hat{P}\in\hat{\mathcal{P}}_{e,\hat{\mathfrak{i}}}\ \Leftrightarrow\ \hat{g}^{-1}\cdot e\in\hat{\mathfrak{g}}\cap (g_0\cdot\mathfrak{i})\ \Leftrightarrow\ (\hat{g}g_0)^{-1}\cdot e\in\mathfrak{i}\ \Leftrightarrow\ \hat{g}g_0P\in\mathcal{P}_{e,\mathfrak{i}}
\end{equation}
where we use that $e\in\hat{\mathfrak{g}}$ (so every $\hat{g}\in \hat{G}$ satisfies that $\hat{g}^{-1}\cdot e\in\hat{\mathfrak{g}}$).
Relations (\ref{new-11}) and (\ref{new-12}) imply
that $\xi$ restricts to an isomorphism
between $(\hat{\mathcal{P}}_{e,\hat{\mathfrak{i}}})^S$ and $(\mathcal{P}_{e,\mathfrak{i}})^S\cap\mathcal{C}$.
This completes the proof of {\rm (b)}.

\section{Reduction to almost simple classical groups}

\label{section-5}

It is convenient to formalize the following property:

\begin{definition}
\label{definition-PGe}
Given a reductive, connected group $G$ and a nilpotent element $e$ in its Lie algebra $\mathfrak{g}$, we say that property $\mathrm{P}(G,e)$ is satisfied if, for some standard triple $\{e,h,f\}\subset\mathfrak{g}$, letting $S=\{\lambda(t):t\in\mathbb{C}^*\}\subset G$ be the subtorus corresponding to $h$ in the sense of Section \ref{section-2-1-2},
for every parabolic subgroup $P\subset G$ and every $P$-stable subspace $\mathfrak{i}\subset \mathfrak{p}=Lie(P)$, the variety $(\mathcal{P}_{e,\mathfrak{i}})^S$ admits an affine paving.
\end{definition}

\begin{remark}
If $\{e,h,f\},\{e,h',f'\}\subset\mathfrak{g}$ are two standard triples containing $e$, then there is $g_0\in G$ such that $g_0\cdot e=e$ and $h'=g_0\cdot h$ (see \cite[Theorem 3.4.10]{Collingwood-McGovern}).
If $\lambda:\mathbb{C}^*\to G$ is the cocharacter corresponding to $h$ in the sense of Section \ref{section-2-1-2}, then
$\mu:\mathbb{C}^*\to G$ defined by $\mu(t)=g_0\lambda(t)g_0^{-1}$ corresponds to $h'$.
Let $S=\{\lambda(t):t\in\mathbb{C}^*\}$ and $S'=\{\mu(t):t\in\mathbb{C}^*\}$.
Since $g_0$ stabilizes $e$, it induces a well-defined automorphism $\mathcal{P}_{e,\mathfrak{i}}\to\mathcal{P}_{e,\mathfrak{i}}$, $gP\mapsto g_0gP$. This automorphism restricts to an isomorphism between the fixed point sets $(\mathcal{P}_{e,\mathfrak{i}})^S\stackrel{\sim}{\to}(\mathcal{P}_{e,\mathfrak{i}})^{S'}$.
We conclude from this that, if property ${\rm P}(G,e)$ is satisfied with respect to a standard triple $\{e,h,f\}$, then it holds with respect to any other standard triple $\{e,h',f'\}$ containing $e$.
\end{remark}

In these terms, Proposition \ref{proposition-section-4} and Example \ref{example-3} show that $\mathrm{P}(\hat{G},e)$ implies $\mathrm{P}(G,e)$, whereas Proposition \ref{proposition-section-3} implies that Theorem \ref{theorem-1} will be proved once we know that property $\mathrm{P}(G,e)$ holds for all $e$ such that $Lie(\hat{G})$ has no nonregular component of exceptional type.
The purpose of this section is to point out other situations where property $\mathrm{P}(G,e)$ is transmitted from a pair $(G,e)$ to another.

\subsection{Products}
Here we assume that $G=G^1\times\cdots\times G^k$ where $G^1,\ldots,G^k$ are reductive connected groups. Then, letting $\mathfrak{g}^i$ be the Lie algebra of $G^i$, we have $\mathfrak{g}=\mathfrak{g}^1\times\cdots\times\mathfrak{g}^k$.
Thus, any element $e\in\mathfrak{g}$ can be uniquely written $e=(e^1,\ldots,e^k)$ and $e$ is nilpotent if and only if $e^i$ is nilpotent for all $i\in\{1,\ldots,k\}$.
It is also clear that $e$ is distinguished in $\mathfrak{g}$ if and only if $e^i$ is distinguished in $\mathfrak{g}^i$ for all $i\in\{1,\ldots,k\}$.

\begin{proposition}
\label{proposition-section-5-1}
Let $G=G^1\times\cdots\times G^k$ and $e=(e^1,\ldots,e^k)\in\mathfrak{g}$ nilpotent. Then, $\mathrm{P}(G,e)$ holds whenever $\mathrm{P}(G^i,e^i)$ holds for all $i\in\{1,\ldots,k\}$.
\end{proposition}

\begin{proof}
For every $i\in\{1,\ldots,k\}$, fix a standard triple $\{e^i,h^i,f^i\}\subset\mathfrak{g}^i$
containing $e^i$ and a cocharacter $\lambda_i:\mathbb{C}^*\to G^i$, of image $S^i=\{\lambda_i(t):t\in\mathbb{C}^*\}$, which corresponds to $h^i$ in the sense of Section \ref{section-2-1-2}.
Then, letting $h=(h^1,\ldots,h^k)$ and $f=(f^1,\ldots,f^k)$, the elements $\{e,h,f\}$ form a standard triple of $\mathfrak{g}$ and the cocharacter
$\lambda:=(\lambda_1,\ldots,\lambda_k):\mathbb{C}^*\to G$ corresponds to $h$. Let $S:=\{\lambda(t):t\in\mathbb{C}^*\}\subset S^1\times\cdots\times S^k$.

Any parabolic subgroup $P\subset G$ can be written $P=P^1\times\cdots\times P^k$ where $P^i\subset G^i$ are parabolic subgroups and any $P$-stable subspace $\mathfrak{i}\subset Lie(P)$ can be written $\mathfrak{i}=\mathfrak{i}^1\times\cdots\times\mathfrak{i}^k$ where $\mathfrak{i}^i\subset Lie(P^i)$ are $P^i$-stable subspaces.
For $i\in\{1,\ldots,k\}$, let $\mathcal{P}^i_{e^i,\mathfrak{i}^i}=\{gP^i\in G^i/P^i:g^{-1}\cdot e^i\in\mathfrak{i}^i\}$.
The map
\[\Phi:G^1/P^1\times\cdots\times G^k/P^k\to G/P,\ (g_1P^1,\ldots,g_kP^k)\mapsto (g_1,\ldots,g_k)P\]
is an isomorphism.
It is easy to check that
\[
\Phi((\mathcal{P}^1_{e^1,\mathfrak{i}^1})^{S^1}\times\cdots\times(\mathcal{P}^k_{e^k,\mathfrak{i}^k})^{S^k})=(\mathcal{P}_{e,\mathfrak{i}})^S.
\]
Therefore, if $(\mathcal{P}^i_{e^i,\mathfrak{i}^i})^{S^i}$
admits an affine paving for all $i$, then $(\mathcal{P}_{e,\mathfrak{i}})^S$ admits an affine paving, too. This shows the proposition.
\end{proof}

\subsection{Central extensions}

Let $\check{G}$ be another reductive, connected, linear algebraic group over $\mathbb{C}$, equipped with a surjective morphism of algebraic groups
\[\pi:G\rightarrow\check{G}\]
whose kernel is contained in the center of $G$. By derivation, we get a surjective morphism of Lie algebras $d\pi:\mathfrak{g}\to\check{\mathfrak{g}}=Lie(\check{G})$  whose kernel lies in the center of $\mathfrak{g}$.
This implies that $d\pi$ restricts to an isomorphism between the semisimple Lie algebras $[\mathfrak{g},\mathfrak{g}]$ and $[\check{\mathfrak{g}},\check{\mathfrak{g}}]$.
Thus, $d\pi$ restricts to a bijection between the nilpotent cones $\mathcal{N}\subset\mathfrak{g}$ and $\check{\mathcal{N}}\subset\check{\mathfrak{g}}$,
and we have that $e\in\mathcal{N}$ is distinguished in $\mathfrak{g}$ if and only if $\check{e}:=d\pi(e)\in\check{\mathcal{N}}$
is distinguished in $\check{\mathfrak{g}}$.

\begin{proposition}
\label{proposition-section-5-2}
Let $\pi:G\to\check{G}$ and $d\pi:\mathfrak{g}\to\check{\mathfrak{g}}$ be central extensions as above. Let $e\in\mathfrak{g}$ be nilpotent and $\check{e}=d\pi(e)$.
Then,  ${\rm P}(G,e)$ holds if and only if  ${\rm P}(\check{G},\check{e})$ holds.
\end{proposition}

\begin{proof}
The maps $P\mapsto\pi(P)$ and $\check{P}\mapsto\pi^{-1}(\check{P})$ are pairwise inverse bijections between the set of closed subgroups of $G$ containing the center $Z(G)$ and the set of closed subgroups of $\check{G}$ containing the center $Z(\check{G})$. Moreover, if $\check{P}=\pi(P)$, then $\pi$ induces a bijection morphism of varieties $\varphi:G/P\to \check{G}/\check{P}$.
Similarly, the map $\mathfrak{p}\mapsto\check{\mathfrak{p}}=d\pi(\mathfrak{p})$ is a bijection between the set of subalgebras of $\mathfrak{g}$ containing the center of $\mathfrak{g}$ and the set of subalgebras of $\check{\mathfrak{g}}$ containing the center of $\check{\mathfrak{g}}$, and the induced linear morphism $\mathfrak{g}/\mathfrak{p}\to\check{\mathfrak{g}}/\check{\mathfrak{p}}$ is bijective. By \cite[\S 5.3.2\,{\rm (iii)} and \S 6.2.1]{Springer-book}, $P$ is a parabolic (resp. Borel) subgroup of $G$ if and only if $\check{P}$ is a parabolic (resp.\ Borel) subgroup of $\check{G}$, and in this case the map
\[
\varphi:G/P\to \check{G}/\check{P},\ gP\mapsto\pi(g)\check{P}
\]
is an isomorphism of $G$-homogeneous varieties.
We fix $P\subset G$ and $\check{P}=\pi(P)\subset\check{G}$ parabolic, and write $\mathfrak{p}=Lie(P)$ and $\check{\mathfrak{p}}=d\pi(\mathfrak{p})=Lie(\check{P})$.

If $\check{\mathfrak{i}}\subset\check{\mathfrak{p}}$ is a $\check{P}$-stable subspace, then $d\pi^{-1}(\check{\mathfrak{i}})\subset\mathfrak{p}$ is a $P$-stable subspace and we have $\check{\mathfrak{i}}=d\pi(d\pi^{-1}(\check{\mathfrak{i}}))$.
Conversely,
if $\mathfrak{i}\subset\mathfrak{p}$ is a $P$-stable subspace, then $d\pi(\mathfrak{i})\subset\check{\mathfrak{p}}$ is a $\check{P}$-stable subspace, and we have $d\pi^{-1}(d\pi(\mathfrak{i}))=\mathfrak{i}+\ker d\pi$. More generally, fix subspaces  $\mathfrak{i}\subset\mathfrak{p}$ and $\check{\mathfrak{i}}\subset\check{\mathfrak{p}}$, respectively $P$- and $\check{P}$-stable, such that $\mathfrak{i}\subset d\pi^{-1}(\check{\mathfrak{i}})\subset\mathfrak{i}+\ker d\pi$.
We claim that
\begin{equation}
\label{ideals-nilpotent}
\mathcal{N}\cap\mathfrak{i}=\mathcal{N}\cap(\mathfrak{i}+\ker d\pi).
\end{equation}
The inclusion $\subset$ is immediate. For checking the other inclusion,
let $x\in\mathfrak{i}$, $z\in\ker d\pi$, assume that $x+z$ is nilpotent, and let us show that $x+z\in\mathfrak{i}$.
Note that $x+z$ is a nilpotent element of $\mathfrak{p}$, hence it is contained in the nilradical $\mathfrak{n}$ of some Borel subalgebra
$\mathfrak{b}\subset\mathfrak{p}$. Fix a Cartan subalgebra $\mathfrak{t}$ with $\mathfrak{b}=\mathfrak{t}\oplus\mathfrak{n}$ and let
$\mathfrak{g}=\mathfrak{t}\oplus \bigoplus_{\alpha\in\Phi}\mathfrak{g}_\alpha$ be the root space decomposition with respect to $\mathfrak{t}$. Thus $\mathfrak{n}=\bigoplus_{\alpha\in\Phi^+}\mathfrak{g}_\alpha$ for a system of positive roots $\Phi^+$.
We can write $x+z=\sum_{\alpha\in I}x_\alpha$ where $x_\alpha\in\mathfrak{g}_\alpha$, $x_\alpha\not=0$,
for a subset $I\subset\Phi^+$. Hence $\mathfrak{i}\ni x=-z+\sum_{\alpha\in I}x_\alpha$ and we know that $-z\in\mathfrak{t}$. Note that, being $P$-stable, $\mathfrak{i}$ is also $\mathfrak{t}$-stable, so it is the sum of its root spaces. This yields $\mathfrak{g}_\alpha\subset\mathfrak{i}$ for all $\alpha\in I$. Whence $x+z\in\mathfrak{i}$. The checking of (\ref{ideals-nilpotent}) is complete.

By virtue of (\ref{ideals-nilpotent}), for $gP\in G/P$, we have
$g^{-1}\cdot e\in\mathfrak{i}\Leftrightarrow d\pi(g^{-1}\cdot e)\in d\pi(\mathfrak{i})\Leftrightarrow \pi(g)^{-1}\cdot\check{e}\in\check{\mathfrak{i}}$,
thereby
\[\varphi(\mathcal{P}_{e,\mathfrak{i}})=\check{\mathcal{P}}_{\check{e},\check{\mathfrak{i}}}:=\{\check{g}\check{P}\in \check{G}/\check{P}:\check{g}^{-1}\cdot\check{e}\in\check{\mathfrak{i}}\}.\]
If $S\subset G$ is any closed subgroup and $\check{S}=\pi(S)$, then we clearly have $\varphi((G/P)^S)=(\check{G}/\check{P})^{\check{S}}$, so $\varphi$ restricts to an isomorphism
\begin{equation}
\label{extension-21}
\mathcal{P}_{e,\mathfrak{i}}\cap(G/P)^S\stackrel{\sim}{\to}\check{\mathcal{P}}_{\check{e},\check{\mathfrak{i}}}\cap(\check{G}/\check{P})^{\check{S}}.
\end{equation}

Let $\{e,h,f\}\subset\mathfrak{g}$ be a standard triple and let $\lambda:\mathbb{C}^*\to G$ be a cocharacter corresponding to $h$ in the sense of Section \ref{section-2-1-2}. Let $S=\{\lambda(t):t\in\mathbb{C}^*\}$.
Setting $\check{h}=d\pi(h)$ and $\check{f}=d\pi(f)$, it is clear that $\{\check{e},\check{h},\check{f}\}$ is a standard triple in $\check{\mathfrak{g}}$.
The cocharacter corresponding to $\check{h}$ is $\check{\lambda}=\pi\circ\lambda:\mathbb{C}^*\to\check{G}$.
Let $\check{S}=\pi(S)=\{\check{\lambda}(t):t\in\mathbb{C}^*\}$.
By (\ref{extension-21}), the varieties $(\mathcal{P}_{e,\mathfrak{i}})^S$ and $(\check{\mathcal{P}}_{\check{e},\check{\mathfrak{i}}})^{\check{S}}$ are isomorphic, hence one admits an affine paving if and only if the other one does. Conditions ${\rm P}(G,e)$ and ${\rm P}(\check{G},\check{e})$ are therefore equivalent.
\end{proof}

\subsection{Conclusion} The conclusion of this section is the following:

\begin{proposition}
\label{proposition-conclusion-section-5}
In order to prove Theorem \ref{theorem-1}, it suffices to show that property $\mathrm{P}(G,e)$ is satisfied whenever $G$ is
$SL(V)$ (for $V=\mathbb{C}^n$), $Sp(V,\omega)$ (for $V=\mathbb{C}^{2n}$ endowed with a symplectic form $\omega$), or $SO(V,\omega)$ (for $V=\mathbb{C}^m$ endowed with a nondegenerate symmetric form $\omega$), and $e\in Lie(G)$ is a distinguished nilpotent element.
\end{proposition}

\begin{proof}
As observed at the beginning of Section \ref{section-5}, Propositions \ref{proposition-section-3} and \ref{proposition-section-4} imply that, in order to prove Theorem \ref{theorem-1}, it suffices to know that $\mathrm{P}(G,e)$ holds whenever $G$ is reductive connected, $e\in Lie(G)$ is a distinguished nilpotent element, and $Lie(G)$ has no nonregular component of exceptional type with respect to $e$
(in the sense of Theorem \ref{theorem-1}). By \cite[\S 8.1.5]{Springer-book}, there is a central extension $G^1\times\cdots\times G^k\to G$ where, for every $i\in\{1,\ldots,k\}$, $G^i$ is either a torus or an almost simple group. Thus, by Propositions \ref{proposition-section-5-1} and \ref{proposition-section-5-2}, we may assume that $G$ itself is a torus or an almost simple group.
If $G$ is a torus, then every partial flag variety $G/P$ is a single point, so property ${\rm P}(G,e)$ is trivially true. If $G$ is an almost simple group whose Lie algebra is of exceptional type, 
then by assumption $e$ is regular in $Lie(G)$, so the variety $\mathcal{P}_{e,\mathfrak{i}}$ is at most one point (see Remark \ref{remark-regular}) and thus property ${\rm P}(G,e)$ is trivially true.
If $G$ is an almost simple group whose Lie algebra is of type $A$--$D$, then invoking again Proposition \ref{proposition-section-5-2},
we may assume that $G$ is of the form $SL(V)$, $Sp(V,\omega)$, or $SO(V,\omega)$. The proof of the proposition is then complete.
\end{proof}

From now on, we focus on the case where $G$ is one of the classical groups $SL(V)$, $Sp(V,\omega)$, and $SO(V,\omega)$.\ The structure of these groups is recalled in the next section.

\section{Classical partial flag varieties}

\label{section-6}

In the case where the group $G$ is classical, a partial flag variety of the form $\mathcal{P}=G/P$ can be identified with a set of partial flags. This well-known fact is recalled in this section.

\subsection{Partial flag variety of type $A$}

Let $G=SL(V)$ be the group of linear
automorphisms of the space $V:=\mathbb{C}^n$ of determinant $1$. Its
Lie algebra $\mathfrak{g}=\mathfrak{sl}(V)$ consists of all linear
endomorphisms of $V$ of trace $0$. A {\it partial flag} of $V$ is a
chain of subspaces $(V_0=0\subset V_1\subset \ldots\subset V_k=V)$.
Given a sequence $\underline{d}=(d_0=0< d_1<\ldots<d_k=n)$, we
denote by $\mathcal{F}_{\underline{d}}$ the set of all partial flags
of $V$ such that $\dim V_p=d_p$ for all $p\in\{0,1,\ldots,k\}$. The
set $\mathcal{F}_{\underline{d}}$ has a natural structure of
algebraic projective variety, on which the group $SL(V)$ acts
transitively.

\begin{proposition}
\label{proposition-classical-1}
{\rm (a)} If $F=(V_0,\ldots,V_k)$ is a partial flag of $V=\mathbb{C}^n$, then $P_F:=\{g\in SL(V):g(V_p)=V_p\ \forall p=0,\ldots,k\}$ is a parabolic subgroup of $SL(V)$ and its Lie algebra is
$\mathfrak{p}_F:=\{x\in \mathfrak{sl}(V):x(V_p)\subset V_p\ \forall p=0,\ldots,k\}$.\ Any parabolic subgroup of $SL(V)$ (resp. any parabolic subalgebra of $\mathfrak{sl}(V)$) is of this form. \\
{\rm (b)} Let $F\in\mathcal{F}_{\underline{d}}$.
The map $gP_F\mapsto g(F)$ is an isomorphism of $SL(V)$-homogeneous varieties between the partial flag variety $SL(V)/P_F$ and the variety of partial flags $\mathcal{F}_{\underline{d}}$.
\end{proposition}

\begin{remark}
Part (b) of the statement (or equivalently the $SL(V)$-homogeneity of the flag variety $\mathcal{F}_{\underline{d}}$) shows that the map $F\mapsto P_F$ between partial flags of $V$ and parabolic subgroups of $SL(V)$ is also injective (hence bijective).
\end{remark}

\subsection{Partial flag varieties of types $B$, $C$, and $D$}
Here we consider the space $V=\mathbb{C}^m$ equipped with a
nondegenerate bilinear form $\omega$ which is either symmetric, or
skew-symmetric (i.e., symplectic). Let $G\subset SL(V)$ be the
subgroup of automorphisms $g$ preserving $\omega$ (i.e.,
$\omega(gv,gv')=\omega(v,v')$ for all $v,v'\in V$). Its Lie algebra
 $\mathfrak{g}\subset\mathfrak{sl}(V)$ is the subspace of
endomorphisms $x$ that are antiadjoint with respect to $\omega$ (i.e.,
$\omega(xv,v')=-\omega(v,xv')$ for all $v,v'\in V$). If the form
$\omega$ is symmetric, then $G$ (resp. $\mathfrak{g}$) is denoted by
$SO(V,\omega)$ (resp. $\mathfrak{so}(V,\omega)$) and called a
special orthogonal group (resp. Lie algebra): a classical group of
type $B$ or $D$ depending on whether $\dim V$ is odd or even. If the
form $\omega$ is symplectic (which forces $\dim V$ to be even), then
$G$ (resp. $\mathfrak{g}$) is denoted by $Sp(V,\omega)$ (resp.
$\mathfrak{sp}(V,\omega)$) and called symplectic group (resp.\ Lie
algebra): a classical group of type $C$.

For a subspace $V'\subset V$, we set $V'^\perp:=\{v\in V:\omega(v,v')=0\ \forall v'\in V'\}$.
An {\it isotropic partial flag} of $(V,\omega)$ is a chain of subspaces $(V_0=0\subset V_1\subset\ldots\subset V_{k-1})$ such that $V_p$ is isotropic for all $p=0,\ldots,k-1$, i.e., we have $V_p\subset V_p^\perp$.
Given a sequence of integers $\underline{d}=(d_0=0<d_1<\ldots<d_{k-1}\leq\frac{\dim V}{2})$, we denote by $\mathcal{F}_{\underline{d}}^\omega$ the set of all isotropic partial flags of $(V,\omega)$ such that $\dim V_p=d_p$ for all $p\in\{0,\ldots,k-1\}$. The set $\mathcal{F}_{\underline{d}}^\omega$ has a natural structure of algebraic projective variety. If $G$ is of type $B$ or $C$ (i.e., $\dim V$ is odd or $\omega$ is symplectic) or $d_{k-1}<\frac{\dim V}{2}$, then $\mathcal{F}_{\underline{d}}^\omega$ is connected and $G$-homogeneous.\ If $G$ is of type $D$ and $d_{k-1}=\frac{\dim V}{2}$, then $\mathcal{F}_{\underline{d}}^\omega$ has exactly two connected components, which are $G$-homogeneous.

\begin{proposition}
\label{proposition-classical-2}
Let $(V,\omega)$, $G$, and $\mathfrak{g}$ be as above. \\
{\rm (a)} If $F=(V_0,\ldots,V_{k-1})$ is an isotropic partial flag of $(V,\omega)$, then $P_F:=\{g\in G: g(V_p)=V_p\ \forall p=0,\ldots,k-1\}$ is a parabolic subgroup of $G$.\ Its Lie algebra is $\mathfrak{p}_F:=\{x\in\mathfrak{g}:x(V_p)\subset V_p\ \forall p=0,\ldots,k-1\}$. Any parabolic subgroup of $G$ (resp. any parabolic subalgebra of $\mathfrak{g}$) is of this form. \\
{\rm (b)}
Assume that $\omega$ is symplectic or $d_{k-1}<\frac{\dim V}{2}$.
Then, for $F\in\mathcal{F}_{\underline{d}}^\omega$,
the map $gP_F\mapsto g(F)$ is an isomorphism of $G$-homogeneous varieties between the partial flag variety $G/P_F$ and the variety of isotropic partial flags $\mathcal{F}_{\underline{d}}^\omega$. \\
{\rm (c)}
Assume that $\omega$ is symmetric and $d_{k-1}=\frac{\dim V}{2}$.
Then, for $F\in\mathcal{F}_{\underline{d}}^\omega$,
the map $gP_F\mapsto g(F)$ is an isomorphism of $G$-homogeneous varieties between $G/P_F$ and the connected component of $\mathcal{F}_{\underline{d}}^\omega$
containing $F$.
\end{proposition}

\begin{remark}
\label{newnew-remark-3}
{\rm (a)}
In the case where $\omega$ is symplectic or $d_{k-1}<\frac{\dim V}{2}$, 
the map $F\mapsto P_F$ between flags in $\mathcal{F}_{\underline{d}}^\omega$ and parabolic subgroups of $G$ is injective (this follows from
the $G$-homogeneity of $\mathcal{F}_{\underline{d}}^\omega$). \\
{\rm (b)} In the case where $\omega$ is symmetric and $d_{k-1}=\frac{\dim V}{2}$, for every flag $F\in\mathcal{F}_{\underline{d}}^\omega$, there is exactly one element $\tilde{F}\in\mathcal{F}_{\underline{d}}^\omega$ different of $F$ such that $P_F=P_{\tilde{F}}$. Then, the maps $gP_F\mapsto g(F)$ and $gP_F\mapsto g(\tilde{F})$ are isomorphisms between $G/P_F$ and the two connected components of $\mathcal{F}_{\underline{d}}^\omega$. The flags $F,\tilde{F}$ are explicitly described in terms of adapted bases in Remark \ref{remark-a-faire}.
\end{remark}

\subsection{Notation for standard parabolic subalgebras}
\label{section-classical-1-3}
Let $G$ be a reductive group, let $\mathfrak{g}$ be its Lie algebra
with Cartan subalgebra $\mathfrak{t}$, corresponding root system
$\Phi=\Phi(\mathfrak{g},\mathfrak{t})$, and root space decomposition
\[\mathfrak{g}=\mathfrak{t}\oplus\bigoplus_{\alpha\in\Phi}\mathfrak{g}_\alpha.\]
Let $\Delta\subset\Phi$ be a basis and $\Phi^+\subset\Phi$ the
corresponding set of positive roots. Given a subset
$I\subset\Delta$, we let $\Phi_I=\Phi\cap\langle
I\rangle_\mathbb{R}$. Then,
\begin{equation}
\label{classical-1}
\mathfrak{p}_I:=\mathfrak{t}\oplus\bigoplus_{\alpha\in\Phi_I\cup\Phi^+}\mathfrak{g}_\alpha,\quad
\mathfrak{l}_I:=\mathfrak{t}\oplus\bigoplus_{\alpha\in\Phi_I}\mathfrak{g}_\alpha,\quad\mbox{and}\quad
\mathfrak{n}_I:=\bigoplus_{\alpha\in\Phi^+\setminus\Phi_I}\mathfrak{g}_\alpha
\end{equation}
are respectively a parabolic subalgebra of $\mathfrak{g}$, a Levi
factor of it, and its nilradical. The parabolic subalgebra
$\mathfrak{p}_I$ is called {\it standard} with respect to the basis
$\Delta$. Let $P_I\subset G$ be the corresponding parabolic
subgroup. Any parabolic subalgebra of $\mathfrak{g}$ (resp. subgroup
of $G$) is conjugate to a standard one.

\subsection{Proof of Propositions \ref{proposition-classical-1} and \ref{proposition-classical-2}}
We briefly review the matrix representation, the root systems, and
the form of the parabolic subalgebras of the classical groups and
Lie algebras. This description easily yields Propositions
\ref{proposition-classical-1}\,{\rm (a)} and \ref{proposition-classical-2}\,{\rm
(a)}, which, in turn, imply Propositions
\ref{proposition-classical-1}\,{\rm (b)} and \ref{proposition-classical-2}\,{\rm
(b)}--{\rm (c)}. It will also be useful in Section \ref{section-classical-2}.

\subsubsection{Notation for matrices} By $\varepsilon_i(h)$ we denote the $i$-th coefficient
of a diagonal matrix $h$. Let $E_{i,j}$ be the elementary matrix
with $1$ in the position $(i,j)$ and $0$'s elsewhere. Given a matrix
$x=(x_{i,j})\in\mathcal{M}_{n}(\mathbb{C})$, we denote by
${}^tx:=(x_{j,i})$ its transpose by the diagonal and by ${}^\delta
x:=(x_{n-j+1,n-i+1})$ its symmetric by the antidiagonal. Let
$I_n,J_n\in\mathcal{M}_{n}(\mathbb{C})$ respectively denote the
identity matrix and the matrix with $1$'s on the antidiagonal and
$0$'s elsewhere. Clearly, $J_n{}^txJ_n={}^\delta x$.

\subsubsection{Type $A$ case}
\label{section-classical-1-4-2} Via the natural basis
$(v_1,\ldots,v_n)$ of $V=\mathbb{C}^n$, the group $SL(V)$ is
isomorphic to $SL_n(\mathbb{C})$. The Lie algebra is
$\mathfrak{sl}_n(\mathbb{C}):=\{x\in\mathcal{M}_n(\mathbb{C}):\mathrm{Tr}\,x=0\}$,
the space of $(n\times n)$-sized matrices of trace zero. A Cartan
subalgebra $\mathfrak{t}$ is formed by the diagonal matrices of trace zero. Then,
the root system
$\Phi=\Phi(\mathfrak{sl}_n(\mathbb{C}),\mathfrak{t})$ consists of
the roots $\varepsilon_i-\varepsilon_j$ for $1\leq i\not=j\leq n$,
with corresponding root vectors $E_{i,j}$. A basis of $\Phi$ is
$\Delta:=\{\alpha_1,\ldots,\alpha_{n-1}\}$, where
$\alpha_i=\varepsilon_i-\varepsilon_{i+1}$.

A subset $I\subset\Delta$ is equivalent to the datum of a sequence
of integers $\underline{d}=(d_0=0<d_1<\ldots<d_k=n)$ such that
$I=I_{\underline{d}}:=\{\alpha_i:i\in\{1,\ldots,n\}\setminus\{d_1,\ldots,d_k\}\}$.
Then, the standard parabolic subalgebra
$\mathfrak{p}_{\underline{d}}:=\mathfrak{p}_I$ is the space of
blockwise upper triangular matrices of trace zero, whose blocks along
the diagonal have respective sizes $d_p-d_{p-1}$ (for $p=1,\ldots,k$). The standard parabolic group $P_{\underline{d}}$ is the subgroup
of $SL_n(\mathbb{C})$ with the same form.

Set $V_p=\langle v_i:1\leq i\leq d_p\rangle_\mathbb{C}$ and
$F=(V_0,\ldots,V_k)$. Then, we have
$\mathfrak{p}_{\underline{d}}=\mathfrak{p}_F$ and
$P_{\underline{d}}=P_F$, where $\mathfrak{p}_F$ and $P_F$ are the
parabolic subalgebra and the parabolic subgroup corresponding to $F$
in the sense of Proposition \ref{proposition-classical-1}\,{\rm
(a)}.

\subsubsection{Type $C$ case}
\label{section-classical-1-4-3}
The space $V=\mathbb{C}^{2n}$ is endowed with a symplectic form
$\omega$. There exists a basis $(v_1,\ldots,v_{2n})$ of $V$ such
that $\omega(v_i,v_{2n+1-i})=1=-\omega(v_{2n+1-i,i})$ for
$i\in\{1,\ldots,n\}$ and $\omega(v_i,v_j)=0$ for any other couple
$(i,j)$. Via the basis $(v_1,\ldots,v_{2n})$, we identify
$Sp(V,\omega)$ with the subgroup $Sp_{2n}(\mathbb{C}):=\{g\in
SL_{2n}(\mathbb{C}):{}^tgK_ng=K_n\}$, where
\[
K_n=\left(\begin{array}{cc} 0 & J_n \\ -J_n & 0
\end{array}\right).
\]
In turn, $\mathfrak{sp}(V,\omega)$ identifies with the Lie algebra
$\mathfrak{sp}_{2n}(\mathbb{C}):=\{x\in\mathfrak{sl}_{2n}(\mathbb{C}):{}^txK_n+K_nx=0\}$.
Thus, an element $x\in\mathfrak{sp}_{2n}(\mathbb{C})$ is a matrix of
the form $x=\left(\begin{array}{cc} A & B \\ C & D
\end{array}\right)$
with $A,B,C,D\in\mathcal{M}_n(\mathbb{C})$, $B={}^\delta B$,
$C={}^\delta C$, and $D=-{}^\delta A$.

A Cartan subalgebra
$\mathfrak{t}\subset\mathfrak{sp}_{2n}(\mathbb{C})$ is formed by the
diagonal matrices of $\mathfrak{sp}_{2n}(\mathbb{C})$. The root system
$\Phi=\Phi(\mathfrak{sp}_{2n}(\mathbb{C}),\mathfrak{t})$ consists of
the following roots: $\pm(\varepsilon_i\pm\varepsilon_j)$ (for
$1\leq i<j\leq n$) and $\pm 2\varepsilon_i$ (for $1\leq i\leq n$). A
root vector corresponding to $\varepsilon_i-\varepsilon_j$ is
$E_{i,j}-E_{2n-j+1,2n-i+1}$. Root vectors corresponding to
$\varepsilon_i+\varepsilon_j$ and $-(\varepsilon_i+\varepsilon_j)$
are respectively $E_{i,2n-j+1}+E_{j,2n-i+1}$ and its transpose. Root
vectors corresponding to $2\varepsilon_i$ and $-2\varepsilon_i$ are
respectively $E_{i,2n-i+1}$ and its transpose. A basis of $\Phi$ is
$\Delta=\{\alpha_1,\ldots,\alpha_n\}$ where
$\alpha_i=\varepsilon_i-\varepsilon_{i+1}$ for
$i\in\{1,\ldots,n-1\}$ and $\alpha_n=2\varepsilon_n$.

Any subset $I\subset\Delta$ can be written
$I=I_{\underline{d}}:=\{\alpha_i:i\in\{1,\ldots,n\}\setminus\{d_1,\ldots,d_{k-1}\}\}$
for a sequence $\underline{d}=(d_0=0<d_1<\ldots<d_{k-1}\leq n\}$.
Then, the parabolic subalgebra
$\mathfrak{p}_{\underline{d}}:=\mathfrak{p}_I$ is the space of
blockwise matrices of the form
\begin{equation}
\label{new-classical-2}
x(A,B,C)=\left(
\begin{array}{ccc|ccc}
A_{1,1} & \cdots & A_{1,k} & B_{1,k} & \cdots & B_{1,1} \\
      0 & \ddots & \vdots  & \vdots  &        & \vdots \\
    0   & 0      & A_{k,k} & B_{k,k} & \cdots & B_{k,1} \\
\hline
 0      & 0      & C_{k,k} & -{}^\delta A_{k,k} & \cdots & -{}^\delta A_{1,k} \\
  0     &  0     & 0       & 0                  & \ddots & \vdots \\
   0    &   0    & 0       &       0            & 0      & -{}^\delta A_{1,1}
\end{array}
\right)
\end{equation}
where $A_{p,q}\in
\mathcal{M}_{d_p-d_{p-1},d_q-d_{q-1}}(\mathbb{C})$ (using the
convention $d_k=n$), $B\in\mathcal{M}_{n}(\mathbb{C})$ and
$C_{k,k}\in\mathcal{M}_{n-d_{k-1}}(\mathbb{C})$ satisfy ${}^\delta
B=B$ and ${}^\delta C_{k,k}=C_{k,k}$. In the case where $d_{k-1}=n$,
the blocks $A_{p,k}$, $B_{p,k}$, $B_{k,p}$, and $C_{k,k}$ are empty.
The parabolic subgroup $P_{\underline{d}}:=P_I$ is formed by the blockwise upper triangular matrices of $Sp_{2n}(\mathbb{C})$ with the same frame.

For $p\in\{0,\ldots,k-1\}$, set $V_p=\langle v_i:1\leq i\leq
d_p\rangle_\mathbb{C}$. Thus,
$F:=(V_0,\ldots,V_{k-1})\in\mathcal{F}_{\underline{d}}^\omega$.
Then, we have $\mathfrak{p}_{\underline{d}}=\mathfrak{p}_F$ and
$P_{\underline{d}}=P_F$, where $\mathfrak{p}_F$ and $P_F$ correspond
to $F$ in the sense of Proposition
\ref{proposition-classical-2}\,{\rm (a)}.

\subsubsection{Types $B$ and $D$ cases}
\label{section-classical-1-4-4}
The space $V=\mathbb{C}^m$
is endowed with a nondegenerate symmetric bilinear form $\omega$.
There is a basis $(v_1,\ldots,v_m)$ of $V$ such that
\begin{equation}
\label{newnewnew-17}
\omega(v_i,v_j)=1\ \mbox{ if $i+j=m+1$ \ and }\ \omega(v_i,v_j)=0\ \ \mbox{otherwise}.
\end{equation}
Through the basis
$(v_1,\ldots,v_m)$, the group $SO(V,\omega)$ identifies with the
group of matrices $SO_m(\mathbb{C}):=\{g\in
SL_m(\mathbb{C}):{}^tgJ_mg=J_m\}$. Its Lie algebra is the orthogonal
Lie algebra
$\mathfrak{so}_m(\mathbb{C})=\{x\in\mathfrak{sl}_m(\mathbb{C}):-{}^\delta
x=x\}$, formed by matrices which are antisymmetric by the
antidiagonal.

A Cartan subalgebra
$\mathfrak{t}\subset\mathfrak{so}_{m}(\mathbb{C})$ is formed by the
diagonal matrices of $\mathfrak{so}_{m}(\mathbb{C})$. Let $n=\lfloor\frac{m}{2}\rfloor$. The root
system $\Phi=\Phi(\mathfrak{so}_m(\mathbb{C}),\mathfrak{t})$
consists of the following roots:
$\pm(\varepsilon_i\pm\varepsilon_j)$ (for $1\leq i<j\leq n$), and
$\pm\varepsilon_i$ (for $1\leq i\leq n$ but only in the case where
$m$ is odd, i.e., $m=2n+1$). A root vector corresponding to
$\varepsilon_i-\varepsilon_j$ is $E_{i,j}-E_{m-j+1,m-i+1}$. Root
vectors corresponding to $\varepsilon_i+\varepsilon_j$ and
$-(\varepsilon_i+\varepsilon_j)$ are $E_{i,m-j+1}-E_{j,m-i+1}$ and
its transpose. Root vectors corresponding to $\varepsilon_i$ and
$-\varepsilon_i$ (in the case $m=2n+1$) are $E_{i,n+1}-E_{n+1,m-i+1}$
and its transpose. A basis of $\Phi$ is
$\Delta=\{\alpha_1,\ldots,\alpha_n\}$ where
$\alpha_i=\varepsilon_i-\varepsilon_{i+1}$ (for $i\in\{1,\ldots,n-1\}$)
and $\alpha_n=\varepsilon_n$ (if $m=2n+1$) or
$\alpha_n=\varepsilon_{n-1}+\varepsilon_n$ (if $m=2n$).

A subset $I\subset \Delta$ can be written
$I=I_{\underline{d}}:=\{\alpha_i:i\in\{1,\ldots,n\}\setminus\{d_1,\ldots,d_{k-1}\}\}$
for a sequence $\underline{d}=(d_0=0<d_1<\ldots<d_{k-1}\leq n)$. 

In the case where $m$ is even (i.e., $m=2n$), the group $SO(V,\omega)$ and the Lie algebra
$\mathfrak{so}(V,\omega)$ can as well be identified to $SO_m(\mathbb{C})$ and $\mathfrak{so}_m(\mathbb{C})$ through the basis $(v_1,\ldots,v_{n-1},v_{n+1},v_n,v_{n+2},\ldots,v_{2n})$
(obtained by switching $v_n$ and $v_{n+1}$). This change of basis induces an automorphism of $\mathfrak{so}_m(\mathbb{C})$, which exchanges the simple roots $\alpha_n$ and $\alpha_{n-1}$.
Thereby, up to invoking this automorphism, we may assume without loss of generality that
the set $I_{\underline{d}}$ fulfills the property: $\alpha_n\in I_{\underline{d}}$ $\Rightarrow$ $\alpha_{n-1}\in I_{\underline{d}}$. In other words,
\begin{equation}
\label{assumption-D}
\mbox{in the case where $m$ is even, we may assume that $d_{k-1}\not=n-1$.}
\end{equation}

In the general case (with the assumption made in (\ref{assumption-D})),
the parabolic subalgebra
$\mathfrak{p}_{\underline{d}}:=\mathfrak{p}_I$ is then the subalgebra of
blockwise upper triangular matrices of the form
\begin{equation}
\label{new-classical-3}
x(B)=\left( \begin{array}{cccccc} B_{1,2k-1} & \cdots & B_{1,k} &
\cdots & B_{1,1} \\
0 & \ddots & \vdots & & \vdots \\
0 & 0 & B_{k,k} & \cdots & B_{k,1} \\
0 & \ddots & 0 & \ddots & \vdots \\
0 & 0 & 0 & 0 & B_{2k-1,1}
\end{array} \right)
\end{equation}
where the diagonal blocks $B_{p,2k-p}$ have sizes $d_p-d_{p-1}$ for $p\in\{1,\ldots,k-1\}$ and where the
full matrix is antisymmetric with respect to the antidiagonal, i.e.,
$-{}^\delta B=B$, in particular we have $B_{2k-p,p}=-{}^\delta
B_{p,2k-p}$ for all $p$. In particular, the diagonal block $B_{k,k}$ has size $m-2d_{k-1}$. In the case where $d_{k-1}=n$ and $m$ is even (i.e., $m=2n$), the blocks
$B_{p,k}$ and $B_{k,q}$ are all empty. The parabolic group
$P_{\underline{d}}:=P_I$ is the subgroup of $SO_m(\mathbb{C})$ with
the same frame.

For $p\in\{0,\ldots,k-1\}$, set $V_p=\langle v_i:1\leq i\leq
d_p\rangle_\mathbb{C}$. Then, letting $F=(V_0,\ldots,V_{k-1})$, we
have $F\in\mathcal{F}_{\underline{d}}^\omega$, and it is clear that
$\mathfrak{p}_{\underline{d}}=\mathfrak{p}_F$ and
$P_{\underline{d}}=P_F$, where $\mathfrak{p}_F$ and $P_F$ are as in
Proposition \ref{proposition-classical-2}\,{\rm (a)}.

\begin{remark}
\label{remark-a-faire}
If $d_{k-1}<\frac{m}{2}$, then $F$ is the only flag such that $P_{\underline{d}}=P_F$.
Assume now that $m=2n$ and $d_{k-1}=n$.
Let $\phi:V\to V$ be the linear automorphism such that $\phi(v_i)=v_i$ for all $i\notin\{n,n+1\}$, $\phi(v_n)=v_{n+1}$, and $\phi(v_{n+1})=v_n$.
Set $\tilde{V}_p=\phi(V_p)$ and $\tilde{F}=(\tilde{V}_0,\ldots,\tilde{V}_{k-1})$.
Then, we have $\tilde{F}\in\mathcal{F}_{\underline{d}}^\omega$ and it is readily seen that $P_{\underline{d}}=P_{\tilde{F}}$. As mentioned in Remark \ref{newnew-remark-3}, $F$ and $\tilde{F}$ are the only two flags such that $P_{\underline{d}}=P_F=P_{\tilde{F}}$.
\end{remark}

\section{Parabolic ideals and classical form of the variety $\mathcal{P}_{e,\mathfrak{i}}$}

\label{section-7}
\label{section-classical-2}

Let $G$ be a reductive connected group over $\mathbb{C}$. Apart from a parabolic subgroup $P\subset G$, its Lie algebra $\mathfrak{p}\subset\mathfrak{g}$, and a nilpotent element $e\in\mathfrak{g}$, the definition of the variety $\mathcal{P}_{e,\mathfrak{i}}$ studied in this paper also involves a $P$-stable subspace $\mathfrak{i}\subset\mathfrak{p}$.\ The purpose of this section is to describe the form taken by such subspaces $\mathfrak{i}$. In particular, the following fact will be noticed (see Section \ref{section-classical-2-2}).

\begin{lemma}
\label{lemma-classical-1}
Given a linear subspace $\mathfrak{i}\subset\mathfrak{p}$, the following conditions are equivalent: \\
{\rm (i)} $\mathfrak{i}$ is stable by the adjoint action of $P$ (i.e., $P\cdot\mathfrak{i}\subset\mathfrak{i}$); \\
{\rm (ii)} $\mathfrak{i}$ is an ideal of $\mathfrak{p}$ (i.e., $[\mathfrak{p},\mathfrak{i}]\subset\mathfrak{i}$).
\end{lemma}

In the classical cases, we can propose a description of the ideals $\mathfrak{i}$ based on the elementary form of the parabolic subgroups and subalgebras in terms of automorphisms and endomorphisms preserving a given partial flag (see Propositions \ref{proposition-classical-1}--\ref{proposition-classical-2}).
Relying on the interpretation of partial flag varieties as varieties of partial flags, we also deduce an elementary description of the corresponding variety $\mathcal{P}_{e,\mathfrak{i}}$.\ The first statement focuses on the type $A$ case.

\begin{proposition}
\label{proposition-classical-3}
Let $G=SL(V)$ where $V=\mathbb{C}^n$. We consider a partial flag
$F=(V_0=0\subset V_1\subset\ldots\subset V_k=V)\in\mathcal{F}_{\underline{d}}$, the corresponding parabolic subgroup $P=P_F\subset SL(V)$, and the Lie algebra $\mathfrak{p}=\mathfrak{p}_F\subset\mathfrak{sl}(V)$ (see Proposition \ref{proposition-classical-1}\,{\rm (a)}).
Let $e\in\mathfrak{sl}(V)$ be nilpotent (a nilpotent endomorphism of $V$).
\\
{\rm (a)}  Given a sequence of integers $\underline{c}=(c_0\leq\ldots\leq c_k)$ such that $0\leq c_p\leq p$ for all $p$, the space
\[\mathfrak{i}^F_{\underline{c}}:=\{x\in\mathfrak{sl}(V):x(V_p)\subset V_{c_p}\ \forall p=0,\ldots,k\}\]
is a $P$-stable subspace of $\mathfrak{p}$. The map $gP\mapsto g(F)$ is an isomorphism between the variety $\mathcal{P}_{e,\mathfrak{i}^F_{\underline{c}}}$ and the variety
\[\mathcal{F}_{e,\underline{d},\underline{c}}:=\{(W_0,\ldots,W_k)\in\mathcal{F}_{\underline{d}}:e(W_p)\subset W_{c_p}\ \forall p=0,\ldots,k\}.\]
{\rm (b)} Moreover, any ideal is essentially of this form in the sense that: for any $P$-stable subspace $\mathfrak{i}\subset\mathfrak{p}$, we can find a sequence $\underline{c}$ such that the sets of nilpotent elements of $\mathfrak{i}$ and $\mathfrak{i}^F_{\underline{c}}$ coincide.\ This implies $\mathcal{P}_{e,\mathfrak{i}}=\mathcal{P}_{e,\mathfrak{i}^F_{\underline{c}}}$.
\end{proposition}

For stating an analogous result in the other three classical cases,
we need to introduce a piece of notation. Given
$F=(V_0,\ldots,V_{k-1})$ an isotropic partial flag of a space
$(V,\omega)$ equipped with a nondegenerate bilinear form, its {\it
completion}
$\overline{F}=(\overline{V}_0,\ldots,\overline{V}_{2k-1})$ is the
sequence given by
\[\overline{V}_p=V_p\ \mbox{ for $p\in\{0,\ldots,k-1\}$}\quad\mbox{and}\quad \overline{V}_p=V_{2k-1-p}^\perp\ \mbox{ for $p\in\{k,\ldots,2k-1\}$}.\]
If $\underline{c}=(c_0\leq \ldots\leq c_{2k-1})$ is a sequence of integers with $0\leq c_p\leq p$ for all $p$, then we let $\underline{c}^*=(c_0^*\leq \ldots\leq c^*_{2k-1})$ be the {\it dual} sequence given by $c_p^*=|\{q=1,\ldots,2k-1:c_q\geq 2k-p\}|$.\ It also satisfies $0\leq c_p^*\leq p$ for all $p$.

\begin{proposition}
\label{proposition-classical-4} Let $G\subset SL(V)$ be the subgroup
of automorphisms which preserve a nondegenerate symmetric or
skew-symmetric bilinear form $\omega$ and let
$\mathfrak{g}\subset\mathfrak{sl}(V)$ be its Lie algebra. We
consider an isotropic partial flag
$F=(V_0,\ldots,V_{k-1})\in\mathcal{F}_{\underline{d}}^\omega$, the
parabolic subgroup $P=P_F\subset G$, and the Lie algebra
$\mathfrak{p}=\mathfrak{p}_F\subset\mathfrak{g}$ (see Proposition \ref{proposition-classical-2}\,{\rm (a)}). Let
$e\in\mathfrak{g}$  be nilpotent (a nilpotent antiadjoint endomorphism
of $V$).
\\
{\rm (a)} Given a sequence of integers $\underline{c}=(c_0\leq\ldots\leq c_{2k-1})$ such that $0\leq c_p\leq p$ for all $p$, the space
\[\mathfrak{i}^F_{\underline{c}}:=\{x\in\mathfrak{g}: x(\overline{V}_p)\subset\overline{V}_{c_p}\ \ \forall p=0,\ldots,2k-1\}\]
is a $P$-stable subspace of $\mathfrak{p}$. The map $gP\mapsto g(F)$ is an isomorphism between the variety $\mathcal{P}_{e,\mathfrak{i}^F_{\underline{c}}}$ and the variety
\[
\hat{\mathcal{F}}^\omega_{e,\underline{d},\underline{c}}:=\{(W_0,\ldots,W_{k-1})\in\hat{\mathcal{F}}_{\underline{d}}^\omega:e(\overline{W}_p)\subset \overline{W}_{c_p}\ \forall p=0,\ldots,2k-1\},
\]
where $\hat{\mathcal{F}}_{\underline{d}}^\omega$ denotes the
connected component of $\mathcal{F}_{\underline{d}}^\omega$
containing $F$ (we have
$\hat{\mathcal{F}}_{\underline{d}}^\omega=\mathcal{F}_{\underline{d}}^\omega$
unless $\omega$ is symmetric and $d_{k-1}=\frac{\dim V}{2}$).
\\
{\rm (b)} 
For every $P$-stable subspace $\mathfrak{i}\subset\mathfrak{p}$, there is a sequence $\underline{c}$ as above, with $\underline{c}^*=\underline{c}$, such that one of the following situations occurs:
\begin{itemize}
\item[\rm (i)] $\omega$ is symplectic or $d_{k-2}<\frac{\dim V}{2}-1$, and the sets of nilpotent elements of $\mathfrak{i}$ and $\mathfrak{i}^F_{\underline{c}}$ coincide.\ Thus, $\mathcal{P}_{e,\mathfrak{i}}=\mathcal{P}_{e,\mathfrak{i}^F_{\underline{c}}}$. \item[\rm (ii)] $\omega$ is symmetric and $d_{k-2}=\frac{\dim V}{2}-1$ (and so $d_{k-1}=\frac{\dim V}{2}$), and the set of nilpotent elements of $\mathfrak{i}$
coincides with the one of $\mathfrak{i}_{\underline{c}}^F$ or the one of $\mathfrak{i}_{\underline{c}}^{\tilde{F}}$, 
where $\tilde{F}\in\mathcal{F}_{\underline{d}}^\omega$ is the unique element different of $F$ satisfying $P=P_{\tilde{F}}$ (see Remark \ref{newnew-remark-3}\,{\rm (b)}). Thus, $\mathcal{P}_{e,\mathfrak{i}}=\mathcal{P}_{e,\mathfrak{i}_{\underline{c}}^F}$ or $\mathcal{P}_{e,\mathfrak{i}}=\mathcal{P}_{e,\mathfrak{i}_{\underline{c}}^{\tilde{F}}}$.
\end{itemize} 

\end{proposition}

\begin{remark}
\label{remark-x}
The variety $\hat{\mathcal{F}}^\omega_{e,\underline{d},\underline{c}}$ introduced in Proposition \ref{proposition-classical-4} is an open and closed subvariety of
\[
\mathcal{F}^\omega_{e,\underline{d},\underline{c}}:=\{(W_0,\ldots,W_{k-1})\in\mathcal{F}_{\underline{d}}^\omega:e(\overline{W}_p)\subset \overline{W}_{c_p}\ \forall p=0,\ldots,2k-1\}.
\]
In fact, the equality $\hat{\mathcal{F}}^\omega_{e,\underline{d},\underline{c}}=\mathcal{F}^\omega_{e,\underline{d},\underline{c}}$ holds unless $\omega$ is symmetric and $d_{k-1}=\frac{\dim V}{2}$.
\end{remark}

Propositions \ref{proposition-classical-3}--\ref{proposition-classical-4}\,{\rm (a)} are easy consequences of Propositions \ref{proposition-classical-1}--\ref{proposition-classical-2}.
The proofs of Propositions \ref{proposition-classical-3}--\ref{proposition-classical-4}\,{\rm (b)} will be given in Sections \ref{section-classical-2-3}--\ref{section-classical-2-4}. 

\subsection{Preliminaries}
In this section,
we show some general properties of the ideals of parabolic subalgebras in the reductive case.

\subsubsection{Notation}
\label{section-classical-2-1-1} We use the notation of Section
\ref{section-classical-1-3}, in particular $G$ is a connected
reductive group of Lie algebra $\mathfrak{g}$. We denote by
$T\subset G$ a maximal torus, $\mathfrak{t}\subset\mathfrak{g}$ the
corresponding Cartan subalgebra,
$\Phi=\Phi(\mathfrak{g},\mathfrak{t})$ the root system. Let
$\Delta\subset\Phi^+\subset\Phi$ be a system of positive roots and
the corresponding basis.

Let $P\subset G$ be a parabolic subgroup of Lie algebra $\mathfrak{p}\subset\mathfrak{g}$.\ Up to conjugation, we may assume that $P$ and $\mathfrak{p}$ are standard, i.e., $P=P_I$ and $\mathfrak{p}=\mathfrak{p}_I$ for some subset $I\subset\Delta$.
We have the Levi decomposition $\mathfrak{p}_I=\mathfrak{l}_I\oplus\mathfrak{n}_I$ (see (\ref{classical-1})).
We abbreviate $\mathfrak{l}=\mathfrak{l}_I$.
The Levi subalgebra $\mathfrak{l}$ decomposes as
\[
\mathfrak{l}=\mathfrak{z}\oplus[\mathfrak{l},\mathfrak{l}]=\mathfrak{z}\oplus\mathfrak{l}_1\oplus\ldots\oplus\mathfrak{l}_k
\]
where $\mathfrak{z}$ denotes the center of $\mathfrak{l}$ and $\mathfrak{l}_i\subset\mathfrak{l}$ (for $i\in\{1,\ldots,k\}$) are the simple ideals of the semisimple Lie algebra $[\mathfrak{l},\mathfrak{l}]$. Set $\mathfrak{t}_i=\mathfrak{t}\cap\mathfrak{l}_i$. We have
\[
\mathfrak{l}_i=\mathfrak{t}_i\oplus\bigoplus_{\alpha\in\Phi_{I_i}}\mathfrak{g}_\alpha
\]
where $I=I_1\sqcup\ldots\sqcup I_k$ is a partition into (nonempty) pairwise orthogonal maximal subsets.

An ideal $\mathfrak{i}\subset\mathfrak{p}$ is in particular stable by the adjoint action of the Cartan subalgebra $\mathfrak{t}$. We conclude that $\mathfrak{i}$ admits a decomposition into weight spaces:
\begin{equation}
\label{classical-2}
\mathfrak{i}=\mathfrak{t}\cap\mathfrak{i}\oplus\bigoplus_{\alpha\in\Phi(\mathfrak{i})}\mathfrak{g}_\alpha
\end{equation}
where we write $\Phi(\mathfrak{i})=\{\alpha\in\Phi:\mathfrak{g}_\alpha\subset\mathfrak{i}\}$.

\subsubsection{Elementary ideals}
Given $\alpha\in\Phi^+\cup\Phi_I$, we denote by $\mathfrak{p}(\mathfrak{g}_\alpha)\subset\mathfrak{p}$
the smallest ideal containing $\mathfrak{g}_\alpha$. Our aim is to describe $\mathfrak{p}(\mathfrak{g}_\alpha)$. We distinguish two cases depending on whether $\alpha\in\Phi_I$ or $\alpha\in\Phi^+\setminus\Phi_I$.

First, assume that $\alpha\in\Phi_I$. So, there is $i\in\{1,\ldots,k\}$ such that $\alpha\in\Phi_{I_i}$; equivalently, $\mathfrak{g}_\alpha\subset\mathfrak{l}_i$. The simplicity of $\mathfrak{l}_i$ imposes $\mathfrak{l}_i\subset\mathfrak{p}(\mathfrak{g}_\alpha)$. Thus $\mathfrak{p}(\mathfrak{g}_\alpha)$ is also the smallest ideal of $\mathfrak{p}$ that contains $\mathfrak{l}_i$.
Set $I_i^\perp=\{\beta\in \Delta:(\beta,\gamma)=0\ \mbox{ for all $\gamma\in I_i$}\}$ and let
\[\hat{\mathfrak{n}}_i=\bigoplus_{\beta\in\Phi^+\setminus\Phi_{I_i\cup I_i^\perp}}\mathfrak{g}_\beta,\]
which is the nilradical of the standard parabolic subalgebra $\hat{\mathfrak{p}}_{i}:=\mathfrak{p}_{I_i\cup I_i^\perp}$.

\begin{lemma}
\label{lemma-classical-2}
Assume $\alpha\in\Phi_{I_i}$.
Then,
$\mathfrak{p}(\mathfrak{g}_\alpha)=\mathfrak{l}_i\oplus\hat{\mathfrak{n}}_i$.
\end{lemma}

\begin{proof}
Note that $\mathfrak{p}\subset\hat{\mathfrak{p}}_i$ and $\hat{\mathfrak{n}}_i\subset\mathfrak{n}_I$.
The space $\mathfrak{l}_i\oplus\hat{\mathfrak{n}}_i$ is an ideal in $\hat{\mathfrak{p}}_i$ hence it is a fortiori an ideal in $\mathfrak{p}$. Thereby, $\mathfrak{p}(\mathfrak{g}_\alpha)\subset\mathfrak{l}_i\oplus\hat{\mathfrak{n}}_i$. It remains to show the inverse inclusion. The inclusion $\mathfrak{l}_i\subset\mathfrak{p}(\mathfrak{g}_\alpha)$ is noticed above. Now, take $\beta\in\Phi^+\setminus\Phi_{I_i\cup I_i^\perp}$. Since $\beta\in\Phi^+\setminus\Phi_{I_i^\perp}$, we can find $\gamma\in I_i$ such that $(\beta,\gamma)<0$. This implies that $\beta+\gamma\in\Phi^+$. Since $-\gamma,\beta+\gamma,\beta$ all belong to $\Phi$, we conclude that $\mathfrak{g}_\beta=[\mathfrak{g}_{\beta+\gamma},\mathfrak{g}_{-\gamma}]\subset [\mathfrak{p},\mathfrak{l}_i]\subset\mathfrak{p}(\mathfrak{g}_\alpha)$. Therefore, $\hat{\mathfrak{n}}_i\subset\mathfrak{p}(\mathfrak{g}_\alpha)$. The proof is complete.
\end{proof}

Second, assume that $\alpha\in\Phi^+\setminus\Phi_I$; equivalently, $\mathfrak{g}_\alpha\subset\mathfrak{n}_I$, and so $\mathfrak{p}(\mathfrak{g}_\alpha)\subset\mathfrak{n}_I$.
We describe $\mathfrak{p}(\mathfrak{g}_\alpha)$ in terms of a partial order $\preceq_\mathfrak{p}$ on roots: write $\alpha\preceq_\mathfrak{p}\beta$ if there exist $\gamma_1,\ldots,\gamma_m\in\Phi^+\cup\Phi_I$ ($m\geq 0$) such that
\begin{equation}
\label{classical-3}
\alpha+\sum_{i=1}^p\gamma_i\in\Phi\ \mbox{ for all $p\in\{1,\ldots,m\}$}
\quad\mbox{and}\quad
\beta=\alpha+\gamma_1+\ldots+\gamma_m.
\end{equation}
Set
$\Phi_\mathfrak{p}(\alpha)=\{\beta\in\Phi:\alpha\preceq_\mathfrak{p}\beta\}$.

\begin{lemma}
\label{lemma-classical-3}
Assume that $\alpha\in\Phi^+\setminus\Phi_I$. Then,
\[\mathfrak{p}(\mathfrak{g}_\alpha)=\bigoplus_{\beta\in\Phi_\mathfrak{p}(\alpha)}\mathfrak{g}_\beta.\]
\end{lemma}

\begin{proof}
Let $\beta\in\Phi_\mathfrak{p}(\alpha)$ and take $\gamma_1,\ldots,\gamma_m\in\Phi^+\cup\Phi_I$ satisfying (\ref{classical-3}). The two relations in (\ref{classical-3}) imply \[\mathfrak{g}_\beta=[\mathfrak{g}_{\gamma_m},[\mathfrak{g}_{\gamma_{m-1}},[\ldots,[\mathfrak{g}_{\gamma_1},\mathfrak{g}_\alpha]\ldots]]]\subset[\mathfrak{p},[\mathfrak{p},[\ldots,[\mathfrak{p},\mathfrak{g}_\alpha]\ldots]]]\subset\mathfrak{p}(\mathfrak{g}_\alpha).\]
This establishes the inclusion $V:=\bigoplus_{\beta\in\Phi_\mathfrak{p}(\alpha)}\mathfrak{g}_\beta\subset\mathfrak{p}(\mathfrak{g}_\alpha)$. Note that this inclusion yields in particular $\mathfrak{g}_\beta\subset\mathfrak{n}_I$ for all $\beta\in\Phi_\mathfrak{p}(\alpha)$, hence
\begin{equation}
\label{classical-4}
\Phi_\mathfrak{p}(\alpha)\subset\Phi^+\setminus\Phi_I.
\end{equation}
In order to show the desired equality, it remains to check that
the space $V$ is $\mathfrak{p}$-stable.
Let $\beta\in\Phi_\mathfrak{p}(\alpha)$ and $\gamma\in\Phi^+\cup\Phi_I$. Since $\beta\in\Phi^+\setminus\Phi_I$ (by (\ref{classical-4})), we have $\beta+\gamma\not=0$. Then, either $\beta+\gamma$ is not a root, in which case $[\mathfrak{g}_\gamma,\mathfrak{g}_\beta]=0$, or $\beta+\gamma$ is a root, in which case $\beta+\gamma\in\Phi_\mathfrak{p}(\alpha)$
(by definition of the order $\preceq_\mathfrak{p}$) and so $[\mathfrak{g}_\gamma,\mathfrak{g}_\beta]=\mathfrak{g}_{\beta+\gamma}\subset V$. In both cases, we conclude that $[\mathfrak{g}_\gamma,V]\subset V$ for all $\gamma\in\Phi^+\cup\Phi_I$, so $[\mathfrak{p},V]\subset V$. The proof is now complete.
\end{proof}

\begin{remark}
\label{remark-classical-6-new}
If $\mathfrak{i}\subset\mathfrak{p}$ is an ideal, then the proof of Lemma \ref{lemma-classical-3} shows that the set $\Phi(\mathfrak{i})$ is stable by the order $\preceq_{\mathfrak{p}}$  in the sense that, if $\alpha\preceq_{\mathfrak{p}}\beta$ and $\alpha\in\Phi(\mathfrak{i})$, then $\beta\in\Phi(\mathfrak{i})$. In other words, $\Phi_{\mathfrak{p}}(\alpha)\subset\Phi(\mathfrak{i})$ whenever $\alpha\in\Phi(\mathfrak{i})$.
\end{remark}

\subsubsection{General ideals}
Given a subset $J\subset\Phi^+\cup\Phi_I$, the space
$\sum_{\alpha\in J}\mathfrak{p}(\mathfrak{g}_\alpha)$
is an ideal of $\mathfrak{p}$ that we can describe thanks to Lemmas \ref{lemma-classical-2}--\ref{lemma-classical-3}. The next result shows that any ideal of $\mathfrak{p}$ is of this form, up to a subspace of the center $\mathfrak{z}$ of the Levi subalgebra~$\mathfrak{l}$.

\begin{lemma}
\label{lemma-classical-4}
{\rm (a)} Let $\mathfrak{i}\subset\mathfrak{p}$ be an ideal. Then
\[\mathfrak{i}=\mathfrak{z}\cap\mathfrak{i}\oplus \sum_{\alpha\in\Phi(\mathfrak{i})}\mathfrak{p}(\mathfrak{g}_\alpha),\]
where as before $\Phi(\mathfrak{i})=\{\alpha\in\Phi:\mathfrak{g}_\alpha\subset\mathfrak{i}\}$ and $\mathfrak{z}$ denotes the center of $\mathfrak{l}$. \\
{\rm (b)} Let $\mathfrak{i},\mathfrak{i}'\subset\mathfrak{p}$ be ideals such that $\mathfrak{i}+\mathfrak{z}=\mathfrak{i}'+\mathfrak{z}$. Let $x\in\mathfrak{g}$ be nilpotent. Then, $x\in\mathfrak{i}$ if and only if $x\in\mathfrak{i}'$.
\end{lemma}

\begin{proof}
{\rm (a)}
The inclusion
\[
\mathfrak{z}\cap\mathfrak{i}\oplus\sum_{\alpha\in\Phi(\mathfrak{i})}\mathfrak{p}(\mathfrak{g}_\alpha)\subset\mathfrak{i}
\]
is immediate. For showing the inverse inclusion,
according to (\ref{classical-2}),
it suffices to check that $\mathfrak{t}\cap\mathfrak{i}\subset\mathfrak{z}\cap\mathfrak{i}\oplus\sum_{\alpha\in\Phi(\mathfrak{i})}\mathfrak{p}(\mathfrak{g}_\alpha)$.
So, let $h\in\mathfrak{t}\cap\mathfrak{i}$, that we can write $h=z+h_1+\ldots+h_k$, where $z\in\mathfrak{z}$ and $h_i\in\mathfrak{t}_i$ for all $i\in\{1,\ldots,k\}$.

We claim that $h_1,\ldots,h_k\in\sum_{\alpha\in\Phi(\mathfrak{i})}\mathfrak{p}(\mathfrak{g}_\alpha)$.
For $h_i\not=0$, the center of $\mathfrak{l}_i$ being trivial, we find $\alpha\in\Phi_{I_i}$ such that $[h_i,\mathfrak{g}_\alpha]\not=0$.
So $\mathfrak{g}_\alpha=[h_i,\mathfrak{g}_\alpha]=[h,\mathfrak{g}_\alpha]\subset [h,\mathfrak{p}]\subset\mathfrak{i}$.
Thus, $\alpha\in \Phi(\mathfrak{i})$.
The fact that $\mathfrak{p}(\mathfrak{g}_\alpha)\cap\mathfrak{l}_i$ is a nontrivial ideal of $\mathfrak{l}_i$ and the simplicity of $\mathfrak{l}_i$ force $\mathfrak{l}_i\subset\mathfrak{p}(\mathfrak{g}_\alpha)$, hence $h_i\in\mathfrak{p}(\mathfrak{g}_\alpha)$.
This establishes our claim.

From the claim, we get $h_1,\ldots,h_k\in\mathfrak{i}$, whence $z=h-h_1-\ldots-h_k\in\mathfrak{i}$. Thus $h\in \mathfrak{z}\cap\mathfrak{i}\oplus\sum_{\alpha\in\Phi(\mathfrak{i})}\mathfrak{p}(\mathfrak{g}_\alpha)$.
The proof of {\rm (a)} is complete. \\
{\rm (b)} The assumption, together with part {\rm (a)} and Lemmas \ref{lemma-classical-2} and \ref{lemma-classical-3}, implies that $([\mathfrak{l},\mathfrak{l}]\oplus\mathfrak{n}_I)\cap\mathfrak{i}=([\mathfrak{l},\mathfrak{l}]\oplus\mathfrak{n}_I)\cap\mathfrak{i}'=:\mathfrak{j}$.
Let $x\in\mathfrak{g}$ be nilpotent and assume that $x\in\mathfrak{i}$.
In particular, $x\in\mathfrak{p}$. Let $L\subset P$ be the Levi factor of Lie algebra $\mathfrak{l}$.
Since all the Borel subalgebras of $\mathfrak{p}$ are conjugate under the adjoint action of $L$, we can find $\ell\in L$ such that $\ell\cdot x\in\bigoplus_{\alpha\in\Phi^+}\mathfrak{g}_\alpha\subset[\mathfrak{l},\mathfrak{l}]\oplus \mathfrak{n}_I$. The fact that $[\mathfrak{l},\mathfrak{l}]$ and $\mathfrak{n}_I$ are both $L$-stable yields $x\in[\mathfrak{l},\mathfrak{l}]\oplus\mathfrak{n}_I$. Whence $x\in\mathfrak{j}\subset\mathfrak{i}'$. 
This establishes {\rm (b)}.
\end{proof}

\subsection{Proof of Lemma \ref{lemma-classical-1}}

\label{section-classical-2-2}
The implication ${\rm (i)}\Rightarrow{\rm (ii)}$ is obtained by differentiation. Let us show the implication ${\rm (ii)}\Rightarrow{\rm (i)}$. So we assume that $\mathfrak{i}$ is an ideal of $\mathfrak{p}$.

We use the notation of Section \ref{section-classical-2-1-1}. For
each root $\beta\in\Phi$, there is a unique closed unipotent
subgroup $U_\beta\subset G$ such that
$Lie(U_\beta)=\mathfrak{g}_\beta$. The torus $T$ and the subgroups
$\{U_\beta\}_{\beta\in\Delta\cup(- I)}$ generate $P$. In order to
check that $\mathfrak{i}$ is $P$-stable, it suffices to check that
$\mathfrak{i}$ is stable by $U_\beta$ whenever
$\beta\in\Delta\cup(-I)$. To do this, in view of
(\ref{classical-2}), we need to check that
\begin{equation}
\label{classical-5}
U_\beta\cdot(\mathfrak{t}\cap\mathfrak{i})\subset\mathfrak{i}\quad\mbox{and}\quad U_\beta\cdot\mathfrak{g}_\alpha\subset\mathfrak{i}\ \mbox{ for all }\ \alpha\in\Phi(\mathfrak{i}).
\end{equation}

Let $h\in\mathfrak{t}\cap\mathfrak{i}$. If $\beta(h)=0$, then we have $u\cdot h=h$ for all $u\in U_\beta$. If $\beta(h)\not=0$, then
we get on one hand $\mathfrak{g}_\beta=[\mathfrak{g}_\beta,h]\subset[\mathfrak{p},\mathfrak{i}]\subset\mathfrak{i}$ and we have on the other hand
$U_\beta\cdot h\subset h+\mathfrak{g}_\beta$ (see \cite[\S 3.3]{Steinberg-book}), so $U_\beta\cdot h\subset\mathfrak{i}$.
In both cases, we obtain $U_\beta\cdot h\subset\mathfrak{i}$.
This shows the first part of (\ref{classical-5}).

Let $\alpha\in\Phi(\mathfrak{i})$.
In particular $\alpha\in\Phi^+$ or $\alpha\in\Phi_I$.
We distinguish two cases depending on whether $\alpha+\beta=0$ or $\alpha+\beta\not=0$. The case $\alpha+\beta=0$ may occur only if $\alpha\in\Phi_I$, so $\alpha\in\Phi_{I_i}$ for some $i\in\{1,\ldots,k\}$. In this case, we know from Lemmas \ref{lemma-classical-2} and \ref{lemma-classical-4} that $\mathfrak{g}_\alpha\subset\mathfrak{l}_i\subset \mathfrak{i}$, whereas the fact that $\beta=-\alpha\in\Phi_{I_i}$ ensures that $\mathfrak{l}_i$ is $U_\beta$-stable. Whence, $U_\beta\cdot\mathfrak{g}_\alpha\subset\mathfrak{i}$ in this case.
Next, assume that $\alpha+\beta\not=0$.
Then,
there is an integer $k\geq 1$ such that $\alpha+i\beta$ is a root for all $i\in\{0,\ldots,k\}$ and $\alpha+i\beta$ is not a root for $i>k$ (see \cite[\S 9.4]{Humphreys-lie-algebras}).
Moreover we have (see \cite[\S 3.3]{Steinberg-book})
\[U_\beta\cdot\mathfrak{g}_\alpha\subset \sum_{i=0}^k\mathfrak{g}_{\alpha+i\beta}=\mathfrak{g}_\alpha+[\mathfrak{g}_\beta,\mathfrak{g}_\alpha]+\ldots+[\underbrace{\mathfrak{g}_\beta,[\ldots,[\mathfrak{g}_\beta}_{\mbox{\scriptsize $k$ terms}},\mathfrak{g}_\alpha]\ldots]].\]
Since $\mathfrak{g}_\beta\subset\mathfrak{p}$ and $\mathfrak{g}_\alpha\subset\mathfrak{i}$, we conclude that $U_\beta\cdot\mathfrak{g}_\alpha\subset\mathfrak{i}$ in this case, too. This shows the second part of (\ref{classical-5}). The proof of Lemma \ref{lemma-classical-1} is now complete.

\subsection{Proof of Proposition \ref{proposition-classical-3}\,{\rm (b)}}
\label{section-classical-2-3}
We first notice that the ideals $\mathfrak{i}^F_{\underline{c}}$ from Proposition \ref{proposition-classical-3}\,{\rm (a)} satisfy the following rule.\ Given two sequences $\underline{c}=(c_0\leq\ldots\leq c_{k})$ and $\underline{c}'=(c'_0\leq\ldots\leq c'_{k})$ such that $c_p,c'_p\in\{0,1,\ldots,p\}$ for all $p$, we denote by $\max\{\underline{c},\underline{c}'\}$ the sequence $(\max\{c_0,c'_0\}\leq\ldots\leq\max\{c_k,c'_k\})$. It is straightforward to check that
\begin{equation}
\label{classical-6}
\mathfrak{i}^F_{\underline{c}}+\mathfrak{i}^F_{\underline{c}'}=\mathfrak{i}^F_{\max\{\underline{c},\underline{c}'\}}.
\end{equation}

In view of relation (\ref{classical-6}) and Lemma \ref{lemma-classical-4}, in order to prove Proposition \ref{proposition-classical-3}\,{\rm (b)}, it suffices to show that every elementary ideal $\mathfrak{p}(\mathfrak{g}_\alpha)$ coincides with $\mathfrak{i}^F_{\underline{c}}$ for some sequence $\underline{c}$.\

Here, as in Section \ref{section-classical-1-4-2}, we identify
$SL(V)$ with the group $SL_n(\mathbb{C})$, its Lie algebra with the
space $\mathfrak{sl}_n(\mathbb{C})$ of matrices of trace zero, and
$\mathfrak{p}_F$ with the subspace of blockwise upper triangular
matrices with blocks of sizes $d_p-d_{p-1}$ along the diagonal. This
parabolic subalgebra corresponds to the set of simple roots
$I_{\underline{d}}=\{\alpha_i:i\in\{1,\ldots,n-1\}\setminus\{d_1,\ldots,d_k\}\}$.
Let $\alpha=\varepsilon_i-\varepsilon_j$, with $i,j\in\{1,\ldots,n\}$, $i\not=j$, so that the root space $\mathfrak{g}_\alpha$ is
generated by the elementary matrix $E_{i,j}$. For
$\ell\in\{1,\ldots,n\}$, let $p(\ell)\in\{1,\ldots,k\}$ be the
unique number such that $d_{p(\ell)-1}<\ell\leq d_{p(\ell)}$. Recall that $\alpha\in\Phi^+\cup\Phi_{I_{\underline{d}}}$. We
distinguish two cases.

First, assume that $\alpha\in \Phi_{I_{\underline{d}}}$.
Equivalently, we have $p(i)=p(j)=:p$. Then, it follows from Lemma
\ref{lemma-classical-2} that we have
$\mathfrak{p}(\mathfrak{g}_\alpha)=\mathfrak{i}^F_{\underline{c}}$
where $\underline{c}=(c_0,\ldots,c_k)$ is given by $c_q=0$ for $q<p$
and $c_q=p$ for $q\geq p$.

Second, assume that
$\alpha\in\Phi^+\setminus\Phi_{I_{\underline{d}}}$. Equivalently,
$p(i)<p(j)$. It is easy to see that the set
$\Phi_{\mathfrak{p}}(\alpha)$ consists of the roots
$\varepsilon_{i'}-\varepsilon_{j'}$ with $1\leq i'<j'\leq n$ such
that $p(i')\leq p(i)$ and $p(j')\geq p(j)$. Therefore, from Lemma
\ref{lemma-classical-3}, we see that
$\mathfrak{p}(\mathfrak{g}_\alpha)=\mathfrak{i}^F_{\underline{c}}$
where the sequence $\underline{c}=(c_0,\ldots,c_k)$ is given by
$c_q=0$ if $q<p(j)$ and $c_q=p(i)$ if $q\geq p(j)$. The proof of
Proposition \ref{proposition-classical-3}\,{\rm (b)} is complete.

\subsection{Proof of Proposition \ref{proposition-classical-4}\,{\rm (b)}}
\label{section-classical-2-4}

Let two sequences $\underline{c}=(c_0\leq\ldots\leq c_{2k-1})$ and
$\underline{c}'=(c'_0\leq\ldots\leq c'_{2k-1})$ such that $c_p,c'_p\in\{0,1,\ldots,p\}$ for all $p$ and assume that
$\underline{c}=\underline{c}^*$ and
$\underline{c}'=\underline{c}'^*$. As in Section
\ref{section-classical-2-3}, we let
$\max\{\underline{c},\underline{c}'\}=(\max\{c_p,c'_p\})_{p=0}^{2k-1}$.
Then, it is easy to see that
$\max\{\underline{c},\underline{c}'\}=\max\{\underline{c},\underline{c}'\}^*$.
We claim that
\begin{equation}
\label{classical-7}
\mathfrak{i}^F_{\underline{c}}+\mathfrak{i}^F_{\underline{c}'}=\mathfrak{i}^F_{\max\{\underline{c},\underline{c}'\}}.
\end{equation}
Note that the completion
$(\overline{V}_0,\ldots,\overline{V}_{2k-1})$ of the isotropic
partial flag $F$ and the sequence $\underline{c}$ also give rise to
a parabolic subalgebra of $\mathfrak{sl}(V)$ denoted by
$\mathfrak{q}_F:=\{x\in\mathfrak{sl}(V):x(\overline{V}_p)\subset\overline{V}_p\
\mbox{for all $p=0,\ldots,2k-1$}\}$ and to an ideal
$\mathfrak{j}^F_{\underline{c}}:=\{x\in\mathfrak{sl}(V):x(\overline{V}_p)\subset
\overline{V}_{c_p}\ \mbox{ for all $p$}\}$ of $\mathfrak{q}_F$.
Given $x\in\mathfrak{sl}(V)$, let $x^*$ be its adjoint with respect
to the form $\omega$. Then, the property that
$\underline{c}=\underline{c}^*$ implies that
$\mathfrak{j}^F_{\underline{c}}$ is stable by the map $x\mapsto
x^*$. This property (also applied to $\underline{c}'$), combined
with (\ref{classical-6}) and the equalities
$\mathfrak{i}^F_{\underline{c}}=\mathfrak{j}_{\underline{c}}^F\cap\mathfrak{g}$
and
$\mathfrak{i}^F_{\underline{c}'}=\mathfrak{j}_{\underline{c}'}^F\cap\mathfrak{g}$,
easily yields relation (\ref{classical-7}).

In view of (\ref{classical-7}) and Lemma \ref{lemma-classical-4},
the proof of Proposition \ref{proposition-classical-4}\,{\rm (b)} can
be carried out through a careful analysis of the elementary ideals $\mathfrak{p}(\mathfrak{g}_\alpha)$.
The proof is done in the following subsections.

\subsubsection{Type $C$ case}
\label{section-7-4-1}
As in Section \ref{section-classical-1-4-3}, we assume that the bilinear form $\omega$ is symplectic and we identify the Lie algebra $\mathfrak{g}=\mathfrak{sp}(V,\omega)$ with $\mathfrak{sp}_{2n}(\mathbb{C})$. The parabolic subalgebra $\mathfrak{p}=\mathfrak{p}_F$ coincides with the blockwise upper triangular subalgebra $\mathfrak{p}_{\underline{d}}$ described in relation (\ref{new-classical-2}) and corresponding to the sequence $\underline{d}=(d_0=0<\ldots<d_{k-1}\leq n)$, that is, to the subset of simple roots $I_{\underline{d}}=\{\alpha_i:i\in\{1,\ldots,n\}\setminus\{d_1,\ldots,d_{k-1}\}\}$.

Set by convention $d_k=n$.
Given a number $i\in\{1,\ldots,n\}$, we denote by $p(i)\in\{1,\ldots,k\}$ the unique number such that $d_{p(i)-1}<i\leq d_{p(i)}$.
Then, the set $\Phi_{I_{\underline{d}}}$ consists of the elements $\pm(\varepsilon_i-\varepsilon_j)$ for $1\leq i<j\leq n$ such that $p(i)=p(j)$, and $\pm(\varepsilon_i+\varepsilon_j)$ for $1\leq i\leq j\leq n$ such that $p(i)=p(j)=k$.

We focus on the elementary ideal $\mathfrak{p}(\mathfrak{g}_\alpha)$ for $\alpha\in\Phi^+\cup\Phi_{I_{\underline{d}}}$. We distinguish two cases.

First, assume that $\alpha\in\Phi_{I_{\underline{d}}}$.
By Lemma \ref{lemma-classical-2}, we may as well assume that $\alpha\in I_{\underline{d}}$, thus $\alpha=\alpha_i$ for some $i\notin\{d_1,\ldots,d_{k-1}\}$. If $i<n$, then the root space $\mathfrak{g}_\alpha$ has components in the blocks $A_{p(i),p(i)}$ and $-{}^\delta A_{p(i),p(i)}$ in the blockwise upper-triangular matrix representation of $\mathfrak{p}$ (see (\ref{new-classical-2})), whereas if $i=n$, then $\mathfrak{g}_\alpha$ is comprised in the block $B_{k,k}$. Applying Lemma \ref{lemma-classical-2}, we deduce that $\mathfrak{p}(\mathfrak{g}_\alpha)=\mathfrak{i}^F_{\underline{c}}$ where $\underline{c}=(c_0,\ldots,c_{2k-1})$ is given by $c_q=0$ if $q<p(i)$, $c_q=p(i)$ if $p(i)\leq q<2k-p(i)$, and $c_q=2k-p(i)$ if $2k-p(i)\leq q\leq 2k-1$.

Next, assume that $\alpha\in\Phi^+\setminus\Phi_{I_{\underline{d}}}$. We distinguish two subcases.

First, suppose that $\alpha=\varepsilon_i-\varepsilon_j$ for $1\leq i<j\leq n$ such that $p(i)<p(j)$; equivalently the root space $\mathfrak{g}_\alpha$ has components in the blocks $A_{p(i),p(j)}$ and $-{}^\delta A_{p(i),p(j)}$ of the representation of $\mathfrak{p}$ given in (\ref{new-classical-2}). In this situation, the set $\Phi_{\mathfrak{p}}(\alpha)$ consists of the following elements: $\varepsilon_{i'}-\varepsilon_{j'}$ for $1\leq i'<j'\leq n$ such that $p(i')\leq p(i)<p(j)\leq p(j')$; and $\varepsilon_{i'}+\varepsilon_{j'}$ for $1\leq i'\leq j'\leq n$ such that $p(i')\leq p(i)$. We conclude (in the light of (\ref{new-classical-2})) that we have $\mathfrak{p}(\mathfrak{g}_\alpha)=\mathfrak{i}^F_{\underline{c}}$ where $\underline{c}=(c_0,\ldots,c_{2k-1})$ is such that $c_q=0$ if $0\leq q<p(j)$, $c_q=p(i)$ if $p(j)\leq q<2k-p(i)$, and $c_q=2k-p(j)$ if $2k-p(i)\leq q\leq 2k-1$.

Second, suppose that $\alpha=\varepsilon_i+\varepsilon_j$ for $1\leq i\leq j\leq n$ such that $p(i)<k$; equivalently $\mathfrak{g}_\alpha$ has components in the blocks $B_{p(i),p(j)}$ and $B_{p(j),p(i)}$ of the representation of $\mathfrak{p}$ given in (\ref{new-classical-2}). Here, the set $\Phi_{\mathfrak{p}}(\alpha)$ consists of the elements: $\varepsilon_{i'}+\varepsilon_{j'}$ for $1\leq i'\leq j'\leq n$ such that $p(i')\leq p(i)$ and $p(j')\leq p(j)$; and $\varepsilon_{i'}-\varepsilon_{j'}$ for $1\leq i'<j'\leq n$ such that $p(i')\leq p(i)$ and $p(j')=p(j)=k$ (only in the case where $p(j)=k$). Thus, by (\ref{new-classical-2}), we see that $\mathfrak{p}(\mathfrak{g}_\alpha)=\mathfrak{i}^F_{\underline{c}}$ for $\underline{c}=(c_0,\ldots,c_{2k-1})$ such that $c_q=0$ if $0\leq q< 2k-p(j)$, $c_q=p(i)$ if $2k-p(j)\leq q<2k-p(i)$, and $c_q=p(j)$ if $2k-p(i)\leq q\leq 2k-1$. 

In each case, we obtain that $\mathfrak{p}(\mathfrak{g}_\alpha)$ is of the form
$\mathfrak{i}_{\underline{c}}^F$ for some sequence $\underline{c}$ with $\underline{c}=\underline{c}^*$. By (\ref{classical-7}) and Lemma \ref{lemma-classical-4},
this yields Proposition \ref{proposition-classical-4}\,{\rm (b)} in the case
where $\omega$ is symplectic.

\subsubsection{Types $B$ and $D$ cases}
\label{section-7-4-2}
In this section, we assume that the bilinear form $\omega$ is symmetric.
As in Section \ref{section-classical-1-4-4}, the Lie algebra $\mathfrak{so}(V,\omega)$ is identified with the orthogonal Lie algebra $\mathfrak{so}_m(\mathbb{C})$ and the parabolic subalgebra $\mathfrak{p}=\mathfrak{p}_F$ is the standard parabolic subalgebra corresponding to the set
of simple roots $I_{\underline{d}}=\{\alpha_i:i\in\{1,\ldots,n\}\setminus\{d_1,\ldots,d_{k-1}\}\}$,
for a sequence $\underline{d}=(d_0=0<d_1<\ldots<d_{k-1}\leq n)$, with $n=\lfloor\frac{m}{2}\rfloor$.
Moreover, in the case where $m$ is even, we may assume that $d_{k-1}\not=n-1$ (see (\ref{assumption-D})), so that $\mathfrak{p}$ coincides with the subalgebra $\mathfrak{p}_{\underline{d}}$ of blockwise upper triangular matrices described in (\ref{new-classical-3}).

Set by convention $d_k=n$. Given $i\in\{1,\ldots,n\}$, we let $p(i)\in\{1,\ldots,k\}$ be the unique number such that $d_{p(i)-1}<i\leq d_{p(i)}$. Then, the set $\Phi_{I_{\underline{d}}}$ consists of the roots: $\pm(\varepsilon_i-\varepsilon_j)$ for $1\leq i<j\leq n$ such that $p(i)=p(j)$; $\pm(\varepsilon_i+\varepsilon_j)$ for $1\leq i<j\leq n$ such that $p(i)=p(j)=k$; and, only in the case where $m$ is odd, i.e., $m=2n+1$: $\pm\varepsilon_i$ for $1\leq i\leq n$ such that $p(i)=k$.

We consider the elementary ideal $\mathfrak{p}(\mathfrak{g}_\alpha)$ for $\alpha\in\Phi^+\cup\Phi_{I_{\underline{d}}}$. We distinguish two cases.

First, assume that $\alpha\in\Phi_{I_{\underline{d}}}$.\ Thus, according to Lemma \ref{lemma-classical-2}, we may also assume that $\alpha\in I_{\underline{d}}$,
so $\alpha=\alpha_i$ for some $i\in\{1,\ldots,n\}\setminus\{d_1,\ldots,d_{k-1}\}$.
Then, each element in the root space $\mathfrak{g}_\alpha$ has components in the blocks $B_{p(i),2k-p(i)}$ and $B_{2k-p(i),p(i)}$ of the blockwise decomposition of $\mathfrak{p}$ given in (\ref{new-classical-3}). Applying Lemma \ref{lemma-classical-2}, we obtain that $\mathfrak{p}(\mathfrak{g}_\alpha)=\mathfrak{i}^F_{\underline{c}}$ where the sequence $\underline{c}=(c_0,\ldots,c_{2k-1})$ is defined by letting $c_q=0$ if $q<p(i)$, $c_q=p(i)$ if $p(i)\leq q<2k-p(i)$, and $c_q=2k-p(i)$ if $2k-p(i)\leq q\leq 2k-1$.

Next, assume that $\alpha\in\Phi^+\setminus\Phi_{I_{\underline{d}}}$. There are three subcases depending on the form of the root $\alpha$.

To start with, suppose that $\alpha=\varepsilon_i+\varepsilon_j$ for $1\leq i<j\leq n$ such that $p(i)<k$. Then, the root space $\mathfrak{g}_\alpha$ has components in the blocks $B_{p(i),p(j)}$ and $B_{p(j),p(i)}$ of the decomposition (\ref{new-classical-3}). The set $\Phi_{\mathfrak{p}}(\alpha)$ consists of the roots: $\varepsilon_{i'}+\varepsilon_{j'}$ for $1\leq i'<j'\leq n$ such that $p(i')\leq p(i)$ and $p(j')\leq p(j)$; $\varepsilon_{i'}-\varepsilon_{j'}$ for $1\leq i'<j'\leq n$ such that $p(i')\leq p(i)$ and $p(j')=p(j)=k$ (only in the case where $p(j)=k$); and $\varepsilon_{i'}$ for $1\leq i'\leq n$ such that $p(i')\leq p(i)$ (only in the case where $p(j)=k$ and $m=2n+1$). We easily obtain $\mathfrak{p}(\mathfrak{g}_\alpha)=\mathfrak{i}^F_{\underline{c}}$,
where the sequence $\underline{c}=(c_0,\ldots,c_{2k-1})$ is given by $c_q=0$ for $1\leq q<2k-p(j)$, $c_q=p(i)$ for $2k-p(j)\leq q<2k-p(i)$, and $c_q=p(j)$ for $2k-p(i)\leq q\leq 2k-1$.

Second, suppose that $\alpha=\varepsilon_i$ for $1\leq i\leq n$ such that $p(i)<k$ (this case occurs only if $m=2n+1$). So, the root space $\mathfrak{g}_\alpha$ has components in the blocks $B_{p(i),k}$ and $B_{k,p(i)}$ of the decomposition (\ref{new-classical-3}). In this situation, the set $\Phi_{\mathfrak{p}}(\alpha)$ comprises the following roots:
$\varepsilon_{i'}$ for $1\leq i'\leq n$ such that $p(i')\leq p(i)$;
$\varepsilon_{i'}+\varepsilon_{j'}$ for $1\leq i'<j'\leq n$ such that $p(i')\leq p(i)$; $\varepsilon_{i'}-\varepsilon_{j'}$ for $1\leq i'<j'\leq n$ such that $p(i')\leq p(i)$ and $p(j')=k$. Then, we can see that $\mathfrak{p}(\mathfrak{g}_\alpha)=\mathfrak{i}^F_{\underline{c}}$ for $\underline{c}=(c_0,\ldots,c_{2k-1})$ given by $c_q=0$ if $1\leq q<k$, $c_q=p(i)$ if $k\leq q<2k-p(i)$, and $c_q=k$ if $2k-p(i)\leq q\leq 2k-1$.

Third, suppose that $\alpha=\varepsilon_i-\varepsilon_j$ for $1\leq i<j\leq n$ such that $p(i)<p(j)$. Thus, the root space $\mathfrak{g}_\alpha$ has components in the blocks $B_{p(i),2k-p(j)}$ and $B_{2k-p(j),p(i)}$ of the decomposition (\ref{new-classical-3}). 
The study of $\mathfrak{p}(\mathfrak{g}_\alpha)$ in this case requires more care, in
particular we need to distinguish the two situations
\begin{itemize}
\item[\rm (i)] $d_{k-2}<\frac{m}{2}-1$; equivalently, $m=2n+1$ or ($m=2n$ and $d_{k-2}<n-1$);
\item[\rm (ii)] $d_{k-2}=\frac{m}{2}-1$; equivalently, $m=2n$ and $d_{k-2}=n-1$.
\end{itemize}
If $m=2n$, then (knowing that we assume $d_{k-1}\not=n-1$ in this case) condition {\rm (i)}
is equivalent to saying that $n-1\notin\{d_1,\ldots,d_{k-1}\}$, thus $\alpha_{n-1}\in I_{\underline{d}}$, whereas condition {\rm (ii)} is equivalent to saying that $\alpha_{n-1}\notin I_{\underline{d}}$.
Taking this into account, one can check that the set $\Phi_{\mathfrak{p}}(\alpha)$
consists of the following roots:
$\varepsilon_{i'}-\varepsilon_{j'}$ for $1\leq i'<j'\leq n$ such that $p(i')\leq p(i)$ and $p(j')\geq p(j)$; $\varepsilon_{i'}+\varepsilon_{j'}$ for $1\leq i'<j'< n$ such that $p(i')\leq p(i)$; 
only if {\rm (i)} holds or $j<n$:
$\varepsilon_{i'}+\varepsilon_{n}$ for $1\leq i'<n$ such that $p(i')\leq p(i)$; 
and only in the case where $m=2n+1$: $\varepsilon_{i'}$ for $1\leq i'\leq n$ such that $p(i')\leq p(i)$.
Let $\underline{c}=(c_0,\ldots,c_{2k-1})$ be given by $c_q=0$ for $0\leq q<p(j)$, $c_q=p(i)$ for $p(j)\leq q<2k-p(i)$, and $c_q=2k-p(j)$ for $2k-p(i)\leq q\leq 2k-1$.
If {\rm (i)} holds or $j<n$, then we obtain
$\mathfrak{p}(\mathfrak{g}_\alpha)=\mathfrak{i}_{\underline{c}}^F$.
If {\rm (ii)} holds and $j=n$, then, using the above description of the set $\Phi_{\mathfrak{p}}(\tilde\alpha)$ for $\tilde\alpha:=\varepsilon_i+\varepsilon_n$, we can see that 
$\mathfrak{p}(\mathfrak{g}_\alpha)+\mathfrak{p}(\mathfrak{g}_{\tilde\alpha})=\mathfrak{i}_{\underline{c}}^F$.

Finally, we have shown: for every root $\alpha\in\Phi^+\cup\Phi_{I_{\underline{d}}}$,
there is a sequence $\underline{c}$ with $\underline{c}=\underline{c}^*$ such that
\begin{equation}
\label{newnew-26}
\left\{
\begin{array}{ll}
\mathfrak{p}(\mathfrak{g}_\alpha)=\mathfrak{i}_{\underline{c}}^F & \mbox{if $d_{k-2}<\frac{m}{2}-1$ or $\alpha\notin\{\varepsilon_i-\varepsilon_n\}_{i=1}^{n-1}$,} \\
\mathfrak{p}(\mathfrak{g}_\alpha)+\mathfrak{p}(\mathfrak{g}_{\tilde\alpha})=\mathfrak{i}_{\underline{c}}^F & \mbox{if $d_{k-2}=\frac{m}{2}-1$ and $\alpha=\varepsilon_i-\varepsilon_n$,
for $\tilde{\alpha}:=\varepsilon_i+\varepsilon_n$.}
\end{array}
\right.
\end{equation}

Proposition \ref{proposition-classical-4}\,{\rm (b)}\,{\rm (i)} follows
by combining Lemma \ref{lemma-classical-4}, (\ref{classical-7}), and (\ref{newnew-26}).
It remains to prove Proposition \ref{proposition-classical-4}\,{\rm (b)}\,{\rm (ii)}.
So, we assume that $m=2n$ and $d_{k-2}=n-1$.

As before, given the ideal $\mathfrak{i}\subset\mathfrak{p}$, we write $\Phi(\mathfrak{i})=\{\alpha\in\Phi:\mathfrak{g}_\alpha\subset\mathfrak{i}\}$. 
Assume for the moment that the following condition holds:
\begin{equation}
\label{newnew-28}
i_0:=\max\{i:\varepsilon_i-\varepsilon_n\in\Phi(\mathfrak{i})\}\leq j_0:=\max\{i:\varepsilon_i+\varepsilon_n\in\Phi(\mathfrak{i})\}.
\end{equation}
Under this condition, we can see that, for every $i\in\{1,\ldots,n-1\}$, we have
\begin{equation}
\label{newnew-29}
\varepsilon_i-\varepsilon_n\in\Phi(\mathfrak{i})\ \Rightarrow\ \varepsilon_i+\varepsilon_n\in\Phi(\mathfrak{i}).
\end{equation}
Indeed, if $\varepsilon_i-\varepsilon_n\in\Phi(\mathfrak{i})$, then $i\leq i_0\leq j_0$ (by (\ref{newnew-28})). This implies that $\varepsilon_i+\varepsilon_n\in\Phi_{\mathfrak{p}}(\varepsilon_{j_0}+\varepsilon_n)$,
thus $\varepsilon_i+\varepsilon_n\in\Phi(\mathfrak{i})$ (see Remark \ref{remark-classical-6-new}), which establishes (\ref{newnew-29}). From Lemma \ref{lemma-classical-4}, (\ref{classical-7}), (\ref{newnew-26}), and (\ref{newnew-29}), we conclude that there is a sequence $\underline{c}$ with $\underline{c}=\underline{c}^*$ such that $\mathfrak{i}$ and $\mathfrak{i}_{\underline{c}}^F$ have the same nilpotent elements.
Whence
Proposition \ref{proposition-classical-4}\,{\rm (b)}\,{\rm (ii)} in this case.

Finally, it remains to treat the case where (\ref{newnew-28}) does not hold. As in Section \ref{section-classical-1-4-4}, we can find a basis $(v_1,\ldots,v_{2n})$ of the space $V$ satisfying (\ref{newnewnew-17}), and which is adapted to the flag $F$ in the sense that $F=(\langle v_i:1\leq i\leq d_p\rangle_{\mathbb{C}})_{p=0}^{k-1}$. The automorphism $\phi:V\to V$ given by
\[\phi(v_n)=v_{n+1},\ \phi(v_{n+1})=v_n,\ \mbox{ and }\phi(v_i)=v_i\mbox{ for $i\notin\{n,n+1\}$},\]
induces involutive automorphisms
\[\xi:SO(V,\omega)\to SO(V,\omega),\ g\mapsto \phi g \phi\quad\mbox{and}\quad d\xi:\mathfrak{so}(V,\omega)\to\mathfrak{so}(V,\omega).\]
We have $\tilde{F}=\phi(F)$ (with $\tilde{F}$ as in the statement of Proposition \ref{proposition-classical-4}\,{\rm (b)}\,{\rm (ii)} or as described in Remark \ref{remark-a-faire}), so that $\xi$ preserves the parabolic subgroup $P=P_F=P_{\tilde{F}}$ and $d\xi$ stabilizes the Lie algebra $\mathfrak{p}=\mathfrak{p}_F=\mathfrak{p}_{\tilde{F}}$. Thus $\tilde{\mathfrak{i}}:=d\xi(\mathfrak{i})$ is again an ideal of $\mathfrak{p}$. On the other hand, the automorphism $\xi$ exchanges the simple roots $\alpha_n$ and $\alpha_{n+1}$, hence it exchanges the roots $\varepsilon_i-\varepsilon_n$ and $\varepsilon_i+\varepsilon_n$. It follows that the ideal $\tilde{\mathfrak{i}}$ satisfies condition (\ref{newnew-28}) (because (\ref{newnew-28}) is not valid for $\mathfrak{i}$), thereby we can find a sequence $\underline{c}=\underline{c}^*$ such that $\tilde{\mathfrak{i}}$ and $\mathfrak{i}_{\underline{c}}^F$ have the same nilpotent elements. Hence, the sets of nilpotent elements of the ideals $\mathfrak{i}=d\xi(\tilde{\mathfrak{i}})$ and $\mathfrak{i}_{\underline{c}}^{\tilde{F}}=d\xi(\mathfrak{i}_{\underline{c}}^F)$ coincide. The proof of Proposition \ref{proposition-classical-4}\,{\rm (b)} is now complete.

\section{Calculations for classical groups}

\label{section-8}

By Proposition \ref{proposition-conclusion-section-5}, the proof of Theorem \ref{theorem-1} will be complete once we show:

\begin{proposition}
\label{proposition-section-8}
Let $G$ be one of the groups $SL(V)$ (for $V=\mathbb{C}^n$, $n\geq 2$), $Sp(V,\omega)$ (for $V=\mathbb{C}^{2n}$, $n\geq 2$, and $\omega:V\times V\to\mathbb{C}$ a symplectic form), or $SO(V,\omega)$ (for $V=\mathbb{C}^{m}$, $m\geq 3$, and $\omega:V\times V\to\mathbb{C}$ a nondegenerate symmetric bilinear form). Let $e\in\mathfrak{g}=Lie(G)$ be a distinguished nilpotent element. Then, property ${\rm P}(G,e)$ is satisfied.
\end{proposition}

The remainder of this section is devoted to the proof of Proposition \ref{proposition-section-8}.

First, we deal with the case where $G=SL(V)$.
From \cite[Theorem 8.2.14\,{\rm (i)}]{Collingwood-McGovern}, we know that any distinguished nilpotent element $e\in \mathfrak{sl}(V)$ is regular.
Then, by Remark \ref{remark-regular}, the variety $\mathcal{P}_{e,\mathfrak{i}}$ is either empty or a single point. Therefore, property ${\rm P}(G,e)$ trivially holds in this case.

The study of the cases where $G=Sp(V,\omega)$ or $G=SO(V,\omega)$ is much more involved. We deal with these two cases simultaneously.
In the remainder of this section, we assume that the space $V=\mathbb{C}^m$ ($m\geq 1$) is equipped with a nondegenerate bilinear form $\omega$, which can be symmetric or antisymmetric.
Let $G\subset SL(V)$ be the subgroup of automorphisms that preserve $\omega$ and
let $\mathfrak{g}\subset End(V)$ be its Lie algebra, that is, the subspace of endomorphisms that are antiadjoint with respect to $\omega$. A nilpotent element $e\in\mathfrak{g}$ is an antiadjoint nilpotent endomorphism.
The proof displays into several subsections.

\subsection{Review on nilpotent elements in types $B$, $C$, $D$}

\label{section-recalls-nilpotent}

Let $h,f\in\mathfrak{g}$ be such that $\{e,h,f\}$ form a standard triple.
Let $\mathfrak{s}$ be the Lie subalgebra generated by $e,h,f$.
Then, the space $V$ decomposes as direct sums
\[V=\bigoplus_{i\in\mathbb{Z},\, i\geq 0}M(i)=\bigoplus_{j\in\mathbb{Z}}E_j\]
where
$E_j$ is the eigenspace for $h$ corresponding to the eigenvalue $j$ and where
each $M(i)$ is a direct sum (possibly zero) of simple $\mathfrak{s}$-modules of dimension $i+1$ (i.e., of highest weight $i$).
We have $e(E_j)\subset E_{j+2}$.
The subspaces $M(i)$ are pairwise orthogonal in $(V,\omega)$ and the restriction of $\omega$ to $M(i)$ is nondegenerate.

Each simple summand of $M(i)$ is a Jordan block of $e$ of size $i+1$. In particular, for $i\geq 0$,
the number of Jordan blocks of $e$ of size $\geq i+1$ coincides with
$\dim (E_i+E_{i+1})$.

The following statement reviews the characterization of admissible Jordan forms and of distinguished nilpotent elements in types $B$, $C$, $D$ (see \cite[\S 5.1 and \S 8.2]{Collingwood-McGovern}).

\begin{proposition}
\label{proposition-distinguished}
{\rm (a)} Assume that $\omega$ is symplectic. Then, $M(i)$ has an even number of summands whenever $i$ is even. Moreover, $e$ is distinguished if and only if $M(i)$ is zero or simple for all $i$ (in particular, the Jordan form of $e$ is of the form $\mu(e)=(2n_1>2n_2>\ldots>2n_r)$). \\
{\rm (a)} Assume that $\omega$ is symmetric. Then, $M(i)$ has an even number of summands whenever $i$ is odd. Moreover, $e$ is distinguished if and only if $M(i)$ is zero or simple for all $i$ (in particular, the Jordan form of $e$ is of the form $\mu(e)=(2n_1+1>2n_2+1>\ldots>2n_r+1)$).
\end{proposition}

\subsection{Subvarieties $(\mathcal{F}_{e,\underline{d},\underline{c}}^\omega)^S$}

Let $n=\lfloor\frac{m}{2}\rfloor$ where, as above, $m=\dim V$.
Let $k\geq 2$ and let sequences of integers $\underline{d}=(d_0=0<d_1<\ldots<d_{k-1}\leq n)$ and $\underline{c}=(c_0\leq c_1\leq \ldots\leq c_{2k-1})$ with $0\leq c_p\leq p$ for all $p\in\{0,\ldots,2k-1\}$ and $\underline{c}^*=\underline{c}$, that is
\begin{equation}
\label{classicals-newnew-cdual}
c_p=c_p^*:=|\{q=1,\ldots,2k-1:c_q\geq 2k-p\}|\ \mbox{ for all $p=0,\ldots,2k-1$.}
\end{equation}
Recall the varieties of isotropic partial flags $\mathcal{F}_{\underline{d}}^\omega$
and $\mathcal{F}_{e,\underline{d},\underline{c}}^\omega$ introduced in Proposition \ref{proposition-classical-2} and Remark \ref{remark-x}.

Let $\lambda(t)\in G$ be the element given by $\lambda(t)(v)=t^jv$ whenever $v\in E_j$. Thus, $S:=\{\lambda(t):t\in\mathbb{C}^*\}$ is a subtorus corresponding to $h$ in the sense of Section \ref{section-2-1-2}.
The Levi subgroup $L:=Z_G(S)$ can be described as
\[L=\{g\in G:g(E_j)=E_j\ \mbox{ for all } j\}.\]

As usual, we denote by $(\mathcal{F}^\omega_{e,\underline{d},\underline{c}})^S\subset \mathcal{F}^\omega_{e,\underline{d},\underline{c}}$ the subvariety of the elements that are fixed by $S$. In the next statement, $e$ is any nilpotent element (not necessarily distinguished).

\begin{lemma}
\label{lemma-classicals-1}
The following conditions are equivalent: \\
{\rm (i)}
The variety $(\mathcal{F}_{e,\underline{d},\underline{c}}^\omega)^S$ admits an affine paving; \\
{\rm (ii)}
For every $L$-stable subvariety $\mathcal{Z}\subset(\mathcal{F}_{\underline{d}}^\omega)^S$, the variety $\mathcal{Z}\cap (\mathcal{F}_{e,\underline{d},\underline{c}}^\omega)^S$ admits an affine paving; \\
{\rm (iii)}
There are  $L$-stable subvarieties $\mathcal{Z}_1,\ldots,\mathcal{Z}_M\subset(\mathcal{F}_{\underline{d}}^\omega)^S$ such that the varieties $\mathcal{Z}_i\cap (\mathcal{F}_{e,\underline{d},\underline{c}}^\omega)^S$ (for $i=1,\ldots,M$) cover $(\mathcal{F}_{e,\underline{d},\underline{c}}^\omega)^S$ and admit affine pavings.
\end{lemma}

\begin{proof}
By Lemma \ref{lemma-1}, the fixed point set $(\mathcal{F}_{\underline{d}}^\omega)^S\subset\mathcal{F}_{\underline{d}}^\omega$ is a union of finitely many $L$-orbits, and every $L$-orbit of $(\mathcal{F}_{\underline{d}}^\omega)^S$ is a connected component of $(\mathcal{F}_{\underline{d}}^\omega)^S$ (hence is open and closed in $(\mathcal{F}_{\underline{d}}^\omega)^S$). Thus, for every $L$-stable subset $\mathcal{Z}\subset(\mathcal{F}_{\underline{d}}^\omega)^S$, the intersection $\mathcal{Z}\cap(\mathcal{F}_{e,\underline{d},\underline{c}}^\omega)^S$ is the union of some connected components of $(\mathcal{F}_{e,\underline{d},\underline{c}}^\omega)^S$. If {\rm (i)} holds, then each connected component of $(\mathcal{F}_{e,\underline{d},\underline{c}}^\omega)^S$
admits an affine paving, thus $\mathcal{Z}\cap(\mathcal{F}_{e,\underline{d},\underline{c}}^\omega)^S$ admits an affine paving, and this shows that {\rm (ii)} holds.

The implication {\rm (ii)}$\Rightarrow${\rm (iii)} is immediate. 

Assume that {\rm (iii)} holds. In order to show that {\rm (i)} holds, it suffices to check that every connected component $\mathcal{C}\subset(\mathcal{F}_{e,\underline{d},\underline{c}}^\omega)^S$ admits an affine paving.
There is an $L$-orbit $\mathcal{Z}\subset(\mathcal{F}_{\underline{d}}^\omega)^S$ such that $\mathcal{C}\subset \mathcal{Z}\cap(\mathcal{F}_{e,\underline{d},\underline{c}}^\omega)^S$ and we can find $i\in\{1,\ldots,M\}$ such that $\mathcal{Z}\cap(\mathcal{F}_{e,\underline{d},\underline{c}}^\omega)^S\subset \mathcal{Z}_i\cap(\mathcal{F}_{e,\underline{d},\underline{c}}^\omega)^S$. It follows that $\mathcal{C}$ is a connected component of $\mathcal{Z}_i\cap(\mathcal{F}_{e,\underline{d},\underline{c}}^\omega)^S$. Since $\mathcal{Z}_i\cap(\mathcal{F}_{e,\underline{d},\underline{c}}^\omega)^S$ admits an affine paving, we conclude that $\mathcal{C}$ admits an affine paving.\ The proof of the lemma is now complete.
\end{proof}

In order to show Proposition \ref{proposition-section-8} in the case of $Sp(V,\omega)$ and $SO(V,\omega)$, in view of Definition \ref{definition-PGe}, Proposition \ref{proposition-classical-4},
and Remark \ref{remark-x}, it is sufficient to establish the following statement.

\begin{proposition}
\label{proposition-claim-1}
Let $e\in\mathfrak{g}$ be distinguished. Then,
the variety $(\mathcal{F}^\omega_{e,\underline{d},\underline{c}})^S$ admits an affine paving.
\end{proposition}

The remainder of Section \ref{section-8} is devoted to the proof of 
Proposition \ref{proposition-claim-1}. The proof is made by induction on $m\geq 1$. If $m\in\{1,2\}$, then any distinguished nilpotent element in $\mathfrak{g}$ is regular and the property is true by Remark \ref{remark-regular}.
In what follows, let $m\geq 3$ such that Proposition \ref{proposition-claim-1} holds until the rank $m-1$. The nilpotent element $e\in\mathfrak{g}$ is now supposed to be distinguished.

\subsection{Notation}
\label{classicals-notation}

By Proposition \ref{proposition-distinguished}, the fact that $e$ is distinguished implies that the sizes of its Jordan blocks form a decreasing sequence
\[\mu_1>\mu_2>\ldots>\mu_r(>0)\]
where the $\mu_p$'s are even (resp. odd) numbers if $\omega$ is symplectic (resp. symmetric). 
Also by Proposition \ref{proposition-distinguished},
the space $V$ decomposes as
\[V=\bigoplus_{\ell=1}^r M(\ell)\]
where $M(\ell)$ is a simple $\mathfrak{s}$-module of dimension $\mu_\ell$.
Let $I_\ell=\{-\mu_\ell+1,-\mu_\ell+3,\ldots,\mu_\ell-1\}$.
There is a basis $\{v_{i}^\ell:i\in I_\ell\}$ of $M(\ell)$ such that
\[h(v_i^\ell)=iv_i^\ell\quad\mbox{and}\quad e(v_i^\ell)=\left\{\begin{array}{ll}
v_{i+2}^\ell & \mbox{if $i<\mu_\ell-1$,} \\ 0 & \mbox{if $i=\mu_\ell-1$.}
\end{array}\right.\
\]
The fact that $e,h$ are antiadjoint with respect to $\omega$ implies that
\begin{equation}
\label{omega-def}
\omega(v_i^\ell,v_j^q)\not=0\ \mbox{ if and only if }\ \ell=q\ \mbox{ and }\ i+j=0.
\end{equation}

\subsection{Induction hypothesis}
\label{section-classicals-1-1-5}
In this section,
we point out a preliminary fact, which is a consequence of the induction hypothesis. 

We focus on the $\ell$-th Jordan block $M(\ell)$ for $\ell\in\{1,\ldots,r\}$ and on its extremal vectors $v_{\mu_\ell-1}^\ell$ and $v_{-\mu_\ell+1}^\ell$. Set $\tilde{W}=\langle v_{\mu_\ell-1}^\ell,v_{-\mu_\ell+1}^\ell\rangle_\mathbb{C}$ and $\tilde{V}=\langle v_i^q:(q,i)\not=(\ell,\pm(\mu_\ell-1))\rangle_\mathbb{C}$, so that we have the orthogonal decomposition
\[V=\tilde{V}\oplus\tilde{W}.\]
Assume that $\mu_\ell\geq 2$, so that $\dim \tilde{W}=2$.
Recall that $\dim V=m$, thus 
\[\dim\tilde{V}=m-2.\]
The restriction of $\omega$ to $\tilde{V}$ (still denoted by $\omega$) is nondegenerate.
Let $\tilde{G}\subset SL(\tilde{V})$ be the subgroup of automorphisms preserving the form $\omega$ and let $\tilde{\mathfrak{g}}\subset End(\tilde{V})$ be the subspace of endomorphisms that are antiadjoint with respect to $\omega$. Let $\iota:\tilde{V}\to V$ denote the inclusion and let $\pi:V\to \tilde{V}$ denote the orthogonal projection.
The elements $e,h$ induce elements $\tilde{e}:=\pi e\iota$, $\tilde{h}:=\pi h\iota$ in $\tilde{\mathfrak{g}}$, and clearly, we can find $\tilde{f}\in\tilde{\mathfrak{g}}$ such that $\{\tilde{e},\tilde{h},\tilde{f}\}$ form a standard triple.
Let $\tilde\lambda(t):=\pi\lambda(t)\iota=\lambda(t)|_{\tilde{V}}$, so that the rank one torus $\tilde{S}:=\{\tilde{\lambda}(t):t\in\mathbb{C}^*\}\subset \tilde{G}$ corresponds to $\tilde{h}$ in the sense of Section \ref{section-2-1-2}.
We denote by $\tilde{\mathcal{F}}_{\underline{\tilde{d}}}^\omega$ and $(\tilde{\mathcal{F}}_{\tilde{e},\underline{\tilde{d}},\underline{\tilde{c}}}^\omega)^{\tilde{S}}$
 the analogues of the varieties $\mathcal{F}_{\underline{d}}^\omega$ and $(\mathcal{F}_{e,\underline{d},\underline{c}}^\omega)^S$ for the space $\tilde{V}$.

\begin{lemma}
\label{lemma-claim-2}
For any integer $\tilde{k}\geq 1$ and any sequences $\underline{\tilde{d}}=(\tilde{d}_0=0<\tilde{d}_1<\ldots<\tilde{d}_{\tilde{k}-1}\leq \lfloor\frac{m-2}{2}\rfloor)$ and $\underline{\tilde{c}}=(\tilde{c}_0\leq\ldots\leq\tilde{c}_{2\tilde{k}-1})$ with $0\leq \tilde{c}_p\leq p$ for all $p$,
the variety $(\tilde{\mathcal{F}}_{\tilde{e},\underline{\tilde{d}},\underline{\tilde{c}}}^\omega)^{\tilde{S}}$ admits an affine paving.
\end{lemma}

\begin{proof}
By induction hypothesis, Proposition \ref{proposition-claim-1} holds until the rank $m-1$. From this fact and in view of Proposition  \ref{proposition-classical-4} and Remark \ref{remark-x}, 
we know that property $\mathbf{P}(\hat{G},\hat{e})$ (introduced in Definition \ref{definition-PGe}) holds
whenever $\hat{G}$ is a group of the form $SL_q(\mathbb{C})$ (for all $q\geq 1$; see the beginning of Section \ref{section-8}), $Sp_{2q}(\mathbb{C})$ (for $2\leq 2q\leq m-1$), or $SO_q(\mathbb{C})$
(for $1\leq q\leq m-1$), and $\hat{e}\in\hat{\mathfrak{g}}:=Lie(\hat{G})$ is a distinguished nilpotent element.  
By Propositions \ref{proposition-section-5-1} and \ref{proposition-section-5-2}, it follows that $\mathbf{P}(\hat{G},\hat{e})$ holds whenever $\hat{G}$ is a Levi subgroup of $\tilde{G}$.
Then, invoking Proposition \ref{proposition-section-4} (and again Proposition \ref{proposition-classical-4} and Remark \ref{remark-x}) we conclude that $(\tilde{\mathcal{F}}_{\tilde{e},\underline{\tilde{d}},\underline{\tilde{c}}}^\omega)^{\tilde{S}}$ admits an affine paving.
\end{proof}

\subsection{Proof of Proposition \ref{proposition-claim-1} in the case where $c_1=1$}
Note that the first terms of the sequence $\underline{c}=(c_0,c_1,\ldots,c_{2k-1})$ satisfy $c_0=0$ and $c_1\in\{0,1\}$. In this section, we show Proposition \ref{proposition-claim-1} in the case where $c_1=1$. The case where $c_1=0$ will be addressed in Section \ref{big-proof}.

\subsubsection{The subvarieties $\mathcal{X}_j\subset(\mathcal{F}_{e,\underline{d},\underline{c}}^\omega)^S$}
\label{section-8-5-1}

For every $j\in\mathbb{Z}$, we let
\[\mathcal{Z}_j=\{F=(V_0,V_1,\ldots,V_{k-1})\in(\mathcal{F}_{\underline{d}}^\omega)^S:E_j\cap V_1\not=0\ \mbox{and}\ E_{j'}\cap V_1=0\ \forall j'>j\}.\]
Clearly, the set $\mathcal{Z}_j$ is $L$-stable, it is empty for all but finitely many $j$, and we have 
\[(\mathcal{F}_{e,\underline{d},\underline{c}}^\omega)^S=\bigcup_{j\in\mathbb{Z}}\mathcal{Z}_j\cap(\mathcal{F}_{e,\underline{d},\underline{c}}^\omega)^S.\]
For $j\in\mathbb{Z}$, we set
\[\mathcal{X}_j=\mathcal{Z}_j\cap(\mathcal{F}_{e,\underline{d},\underline{c}}^\omega)^S.\]
According to Lemma \ref{lemma-classicals-1}, 
the following implication holds:
\begin{equation}
\label{implication1}
\mbox{$\mathcal{X}_j$ admits an affine paving for all $j$}\ \Rightarrow\ \mbox{$(\mathcal{F}_{e,\underline{d},\underline{c}}^\omega)^S$ admits an affine paving.}
\end{equation}
Our next goal is then to show that each subvariety $\mathcal{X}_j$ has an affine paving.

Fix $j\in\mathbb{Z}$ such that $\mathcal{X}_j$ is nonempty (otherwise there is nothing to prove). 
The following claim uses the notation of Section \ref{classicals-notation}.

\medskip
\noindent
{\it Claim:} There is $\ell\in\{1,\ldots,r\}$ such that $j=\mu_\ell-1$ and $\mu_\ell\geq 2$, and we have
\begin{equation}
\label{classicals-newnew-2}
v^\ell_{\mu_\ell-1}\in V_1\ \mbox{ for all $F=(V_0,\ldots,V_{k-1})\in \mathcal{X}_j$.}
\end{equation}

\begin{proof}[Proof of the Claim]
Take $F=(V_0,\ldots,V_{k-1})\in \mathcal{X}_j$
and let $v\in V_1\cap E_j\setminus\{0\}$. Note that $e(v)\in V_1\cap E_{j+2}=\{0\}$ (since $F\in\mathcal{Z}_j$), hence $v\in V_1\cap E_j\cap\ker e$. According to Section \ref{classicals-notation}, we have $E_j\cap\ker e\not=0$ only if $j=\mu_\ell-1$ for some $\ell\in\{1,\ldots,r\}$
and in this case  $E_j\cap\ker e=\langle v_{j}^\ell\rangle_\mathbb{C}$. Relation (\ref{classicals-newnew-2}) ensues.
By (\ref{classicals-newnew-2}), we get in particular that the vector $v^\ell_{\mu_\ell-1}$ is isotropic, which forces $\mu_\ell\geq 2$ (see (\ref{omega-def})).
\end{proof}

\subsubsection{The variety $\tilde{\mathcal{X}}_j$}
We apply the construction of Section \ref{section-classicals-1-1-5} 
to the choice of $\ell$ made in the Claim of Section \ref{section-8-5-1}
(note that we have $\mu_\ell\geq 2$ as required in Section \ref{section-classicals-1-1-5}). As in Section \ref{section-classicals-1-1-5}, we deal with the space $\tilde{V}=\langle v_i^q:(q,i)\not=(\ell,\pm(\mu_\ell-1)\rangle_\mathbb{C}$, the nilpotent element $\tilde{e}=\pi e\iota$ (where $\iota:\tilde{V}\to V$ and $\pi:V\to \tilde{V}$ are respectively the inclusion and the orthogonal projection), and the torus $\tilde{S}=\{\tilde\lambda(t):t\in\mathbb{C}^*\}$ where $\tilde\lambda(t)=\pi\lambda(t)\iota=\lambda(t)|_{\tilde{V}}$. 
For later use, we note that the definition of $\tilde{e}$ guarantees that
\begin{equation}
\label{eetilde-1}
\mathrm{Im}(e|_{\tilde{V}}-\tilde{e})\subset\langle v_{\mu_\ell-1}^\ell\rangle_\mathbb{C}.
\end{equation}
Let $\underline{\tilde{d}}=(\tilde{d}_0=0\leq \tilde{d}_1<\ldots<\tilde{d}_{k-1})$ be the sequence defined by 
\[\tilde{d}_p=d_p-1\ \mbox{ for all $p\in\{1,\ldots,k-1\}$.}\]
Corresponding to this sequence, we consider the varieties $\tilde{\mathcal{F}}_{\underline{\tilde{d}}}^\omega$,
$\tilde{\mathcal{F}}_{\tilde{e},\underline{\tilde{d}},\underline{c}}^\omega$, and
$(\tilde{\mathcal{F}}_{\tilde{e},\underline{\tilde{d}},\underline{c}}^\omega)^{\tilde{S}}$
(defined like $\mathcal{F}_{\underline{d}}^\omega$, $\mathcal{F}_{e,\underline{d},\underline{c}}^\omega$, and $(\mathcal{F}_{e,\underline{d},\underline{c}}^\omega)^S$, but for the space $\tilde{V}$).
By Lemma \ref{lemma-claim-2},
\begin{equation}
\label{newnewnew-33}
\mbox{the variety $(\tilde{\mathcal{F}}_{\tilde{e},\underline{\tilde{d}},\underline{c}})^{\tilde{S}}$ admits an affine paving.}
\end{equation}

\begin{remark}
\label{remark-8-new}
Hereafter we make a slight abuse of notation since in Section \ref{section-classicals-1-1-5}
the variety
$\tilde{\mathcal{F}}_{\tilde{e},\underline{\tilde d},\underline{c}}^\omega$ is considered for an increasing sequence $\underline{\tilde{d}}$ whereas, here, the  sequence $\underline{\tilde{d}}$ may satisfy $\tilde{d}_0=\tilde{d}_1$. However, this is harmless since the definition of $\tilde{\mathcal{F}}_{\tilde{e},\underline{\tilde d},\underline{c}}^\omega$ still makes sense and the proof of Lemma \ref{lemma-claim-2} remains valid. Actually, if $\tilde{d}_1=\tilde{d}_0$, then we have $\tilde{\mathcal{F}}_{\tilde{e},\underline{\tilde{d}},\underline{c}}^\omega=\tilde{\mathcal{F}}_{\tilde{e},\underline{d}',\underline{{c}}'}^\omega$ where $\underline{d}':=(\tilde{d}_0=0,\tilde{d}_2,\tilde{d}_3,\ldots,\tilde{d}_{k-1})$ and $\underline{c}':=(c_0,c_2-1,c_3-1,\ldots,c_{2k-2}-1)$, so that we retrieve the situation of an increasing sequence $\underline{d}'$.
A similar abuse of notation will be made in Section \ref{big-proof} below.
\end{remark}

The set
\[\tilde{\mathcal{Z}}_j:=\{(W_0,\ldots,W_{k-1})\in(\tilde{\mathcal{F}}_{\underline{\tilde{d}}}^\omega)^{\tilde{S}}: E_{j'}\cap W_1=0\ \forall j'>j\}\]
is clearly stable by the subgroup $\tilde{L}:=Z_{\tilde{G}}(\tilde{S})$. Hence, (\ref{newnewnew-33}) and Lemma \ref{lemma-classicals-1} imply that
\begin{equation}
\label{classicals-newnewnew-31}
\mbox{the subvariety $\tilde{\mathcal{X}}_j:=\tilde{\mathcal{Z}}_j\cap(\tilde{\mathcal{F}}_{\tilde{e},\underline{\tilde{d}},\underline{c}}^\omega)^{\tilde{S}}$ admits an affine paving.}
\end{equation}

\subsubsection{An isomorphism between $\mathcal{X}_j$ and $\tilde{\mathcal{X}}_j$}
\label{section-8-5-3}
In view of (\ref{classicals-newnew-2}), we have $\mathcal{X}_j=\mathcal{Z}'_j\cap(\mathcal{F}_{e,\underline{d},\underline{c}}^\omega)^S$
where we denote
\[\mathcal{Z}'_j=\{(V_0,\ldots,V_{k-1})\in\mathcal{Z}_j:v_{\mu_\ell-1}^\ell\in V_1\}.\]
The maps
\[\Phi:\mathcal{Z}'_j\to\tilde{\mathcal{Z}}_j,\ F\mapsto (V_0\cap \tilde{V},\ldots,V_{k-1}\cap\tilde{V})\]
and 
\[\Psi:\tilde{\mathcal{Z}}_j\to \mathcal{Z}'_j,\ (W_0,\ldots,W_{k-1})\mapsto(W_0,\langle v_{\mu_\ell-1}^\ell\rangle_\mathbb{C}+W_1,\ldots,\langle v_{\mu_\ell-1}^\ell\rangle_\mathbb{C}+W_{k-1})\]
are mutually inverse isomorphisms of algebraic varieties. We claim that 
\begin{equation}
\label{newnewnew-35}
\Phi(\mathcal{X}_j)=\tilde{\mathcal{X}}_j\quad(\mbox{i.e., 
$\Phi$ restricts to an isomorphism $\mathcal{X}_j\stackrel{\sim}{\to}\tilde{\mathcal{X}}_j$).}
\end{equation}
To see this, take $F=(V_0,\ldots,V_{k-1})\in\mathcal{Z}_j'$ and $\Phi(F)=(W_0,\ldots,W_{k-1})\in\tilde{\mathcal{Z}}_j$. 

%

\medskip
\noindent
{\it Claim 1:} If $F$ belongs to $\mathcal{F}_{e,\underline{d},\underline{c}}^\omega$,
then $\Phi(F)$ belongs to $\tilde{\mathcal{F}}_{\tilde{e},\underline{\tilde{d}},\underline{c}}^\omega$.

\begin{proof}[Proof of Claim 1]
As in Proposition \ref{proposition-classical-4}, we consider the completions of $F$ and $\Phi(F)$ defined by
\[(\overline{V}_0,\ldots,\overline{V}_{2k-1}):=(V_0,V_1,\ldots,V_{k-1},V_{k-1}^\perp,\ldots,V_1^\perp,V_0^\perp)\]
and
\[(\overline{W}_0,\ldots,\overline{W}_{2k-1}):=(W_0,W_1,\ldots,W_{k-1},W_{k-1}^{\tilde\perp},\ldots,W_1^{\tilde\perp},W_0^{\tilde\perp}),\]
where the symbols $\perp$ and $\tilde\perp$ stand for the orthogonals in the spaces $(V,\omega)$ and $(\tilde{V},\omega)$, respectively.
Note that
\begin{equation}
\label{newnewnew-36}
\overline{W}_p=\overline{V}_p\cap\tilde{V}=\pi(\overline{V}_p)\ \mbox{ for all $p\in\{0,\ldots,2k-1\}$}
\end{equation}
and
\begin{equation}
\label{newnewnew-37}
\overline{V}_0=0,\ \ \overline{V}_{2k-1}=V,\ \ \mbox{and}\ \ \overline{V}_p=\overline{W}_p\oplus\langle v_{\mu_\ell-1}^\ell\rangle_{\mathbb{C}}\ \mbox{ for all $p\in\{1,\ldots,2k-2\}$.}
\end{equation}
The assumption that $F$ belongs to $\mathcal{F}_{e,\underline{d},\underline{c}}^\omega$
reads as
\begin{equation}
\label{newnewnew-38}
e(\overline{V}_p)\subset \overline{V}_{c_p}\ \mbox{ for all $p\in\{1,\ldots,2k-1\}$.}
\end{equation}
From (\ref{newnewnew-36}) and (\ref{newnewnew-38}), we derive
\[\tilde{e}(\overline{W}_p)=\pi (e(\overline{V}_p\cap \tilde{V}))
\subset \pi(\overline{V}_{c_p})=\overline{W}_{c_p}\ \mbox{ for all $p\in\{1,\ldots,2k-1\}$,}\]
whence $\Phi(F)\in \tilde{\mathcal{F}}_{\tilde{e},\underline{\tilde{d}},\underline{c}}^\omega$.
\end{proof}

\noindent
{\it Claim 2:} If $\Phi(F)$ belongs to $\tilde{\mathcal{F}}_{\tilde{e},\underline{\tilde{d}},\underline{c}}^\omega$, then $F$ belongs to $\mathcal{F}_{e,\underline{d},\underline{c}}^\omega$.

\begin{proof}[Proof of Claim 2]
The assumption on $\Phi(F)$ implies that
\begin{equation}
\label{newnewnew-39}
\tilde{e}(\overline{W}_p)\subset \overline{W}_{c_p}\ \mbox{ for all $p\in\{1,\ldots,2k-1\}$}.
\end{equation}
On the one hand, for $p\in\{1,\ldots,2k-2\}$, using (\ref{newnewnew-37}), 
the fact that $v_{\mu_\ell-1}^\ell\in\ker e$, (\ref{eetilde-1}), (\ref{newnewnew-39}), and the fact that $c_p\geq c_1=1$, we obtain
\[e(\overline{V}_p)=e(\overline{W}_p+\langle v_{\mu_\ell-1}^\ell\rangle_\mathbb{C})
\subset \tilde{e}(\overline{W}_p)+\langle v_{\mu_\ell-1}^\ell\rangle_\mathbb{C}
\subset \overline{W}_{c_p}+\langle v_{\mu_\ell-1}^\ell\rangle_\mathbb{C}=\overline{V}_{c_p}.\]
On the other hand,
by (\ref{classicals-newnew-cdual}), we get in particular
\[|\{q=1,\ldots,2k-1:c_q\geq 2k-1\}|=c_1=1,\]
which forces $c_{2k-1}=2k-1$ (because the sequence $\underline{c}$ is nondecreasing).
Therefore, the condition $e(\overline{V}_{2k-1})\subset \overline{V}_{c_{2k-1}}(=V)$ is trivially satisfied. Altogether,
we have checked that 
$e(\overline{V}_p)\subset\overline{V}_{c_p}$ for all $p\in\{1,\ldots,2k-1\}$,
whence $F\in\mathcal{F}_{e,\underline{d},\underline{c}}^\omega$.
\end{proof}

Relation (\ref{newnewnew-35}) now follows from Claims 1--2 and the equalities
$\mathcal{X}_j=\mathcal{Z}_j'\cap\mathcal{F}_{e,\underline{d},\underline{c}}^\omega$ and 
$\tilde{\mathcal{X}}_j=\tilde{\mathcal{Z}}_j\cap\tilde{\mathcal{F}}_{\tilde{e},\underline{\tilde{d}},\underline{c}}^\omega$.

\subsubsection{Conclusion}
From (\ref{classicals-newnewnew-31})--(\ref{newnewnew-35}),
we obtain that the variety $\mathcal{X}_j$ admits an affine paving for every $j\in\mathbb{Z}$.
In view of (\ref{implication1}), this guarantees that  $(\mathcal{F}_{e,\underline{d},\underline{c}}^\omega)^S$
admits an affine paving, which completes the proof of Proposition \ref{proposition-claim-1} in the case where $c_1=1$.

\subsection{Proof of Proposition \ref{proposition-claim-1} in the case where $c_1=0$}
\label{big-proof}
The proof in this case is more involved.

\subsubsection{The subvarieties $\mathcal{X}_{\underline{j}}\subset(\mathcal{F}_{e,\underline{d},\underline{c}}^\omega)^S$}

\label{section-8-6-1}

For every sequence $\underline{j}=(j_1,\ldots,j_{k-1})\in\mathbb{Z}^{k-1}$,
we define
\begin{eqnarray*}
\mathcal{Z}_{\underline{j}} & = & \big\{(V_0,\ldots,V_{k-1})\in(\mathcal{F}_{\underline{d}}^\omega)^S:
V_p\cap E_{j_p}\not=V_{p-1}\cap E_{j_p}\\ & & \mbox{and}\ V_p\cap E_{j'}=V_{p-1}\cap E_{j'}\ \forall j'<j_p,\ \forall p\in\{1,\ldots,k-1\}\big\}.
\end{eqnarray*}
(Note a difference with the definition of $\mathcal{Z}_j$ in Section \ref{section-8-5-1},
where the second condition is required for $j'>j$.)
Thus $\mathcal{Z}_{\underline{j}}$ is an $L$-stable subvariety of 
$(\mathcal{F}_{\underline{d}}^\omega)^S$, which is empty for all but finitely many sequences $\underline{j}$, and we have
\[(\mathcal{F}_{e,\underline{d},\underline{c}}^\omega)^S=\bigcup_{\underline{j}\in\mathbb{Z}^{k-1}}\mathcal{Z}_{\underline{j}}\cap(\mathcal{F}_{e,\underline{d},\underline{c}}^\omega)^S.\]
Let
\[\mathcal{X}_{\underline{j}}=\mathcal{Z}_{\underline{j}}\cap(\mathcal{F}_{e,\underline{d},\underline{c}}^\omega)^S.\]
By Lemma \ref{lemma-classicals-1}, we get the following implication:
\begin{equation}
\label{newnewnewnew-40}
\mbox{$\mathcal{X}_{\underline{j}}$ admits an affine paving for all $\underline{j}$\quad$\Rightarrow$\quad
$(\mathcal{F}_{e,\underline{d},\underline{c}}^\omega)^S$ admits an affine paving.}
\end{equation}

Hereafter, we fix a sequence $\underline{j}\in\mathbb{Z}^{k-1}$ such that the variety $\mathcal{X}_{\underline{j}}$ is nonempty.
We construct a sequence $\underline{p}=(p_1,\ldots,p_{s})$ (depending on $\underline{j}$) by the following algorithm:
\begin{itemize}
\item Set $p_1=1$;
\item Assume that we have constructed $p_1<p_2<\ldots<p_t\leq k-1$.
\begin{itemize}
\item If there is $p\in\{p_t+1,\ldots,k-1\}$ such that $j_p<j_{p_t}$ and $c_p<p_t$, then denote by $p_{t+1}$ the smallest $p$ with these properties;
\item Otherwise, set $s=t$, $\underline{p}=(p_1,\ldots,p_t)$, and stop the algorithm.
\end{itemize}
\end{itemize}
Finally we get a sequence of integers
\[1=p_1<p_2<\ldots<p_s\leq k-1.\]
We describe some properties of the sequences $\underline{j}$, $\underline{p}$ and the subvariety $\mathcal{X}_{\underline{j}}$.

\medskip
\noindent
{\it Claim 1:}
For all $t\in\{1,\ldots,s\}$, we have
\[j_{p_t}<j_p\ \mbox{ for all $p\in\{1,2,\ldots,p_t-1\}$.}\] 

\begin{proof}[Proof of Claim 1]
We argue by induction on $t\in\{1,\ldots,s\}$ with immediate initialization for $t=1$.
Assume that Claim 1 is valid until the rank $t\in\{1,\ldots,s-1\}$.
From the definition of $p_{t+1}$,
using also that the sequence $\underline{c}$ is nondecreasing, we have
\[c_p\leq c_{p_{t+1}}<p_t\ \mbox{ whenever $p_t\leq p<p_{t+1}$}.\]
Then, the minimality of $p_{t+1}$ (in the definition of the sequence $\underline{p}$) reads as
\[j_p\geq j_{p_t}>j_{p_{t+1}}\ \mbox{ whenever $p_t\leq p<p_{t+1}$}.\]
By induction hypothesis, we also have $j_p> j_{p_t}>j_{p_{t+1}}$ for $1\leq p<p_t$.
Hence Claim~1 is valid until the rank $t+1$.
\end{proof}

\noindent
{\it Claim 2:} For all $t\in\{1,\ldots,s\}$, the following conditions are satisfied: \\
{\rm (a)} For all $(V_0,\ldots,V_{k-1})\in\mathcal{Z}_{\underline{j}}$ we have
\[V_p\cap E_{j'}=0\ \mbox{ whenever $1\leq p<p_t$ and $j'\leq j_{p_t}$,}\quad\mbox{and}\quad
V_{p_t}\cap E_{j_{p_t}}\not=0.\]
{\rm (b)} There is $\ell_t\in\{1,\ldots,r\}$ such that $j_{p_t}=\mu_{\ell_t-1}$ and
$\mu_{\ell_t}\geq 2$, and we have
\begin{equation}
\label{classicals-newnew-6}
v_{\mu_{\ell_t}-1}^{\ell_t}\in V_{p_t}\ \mbox{ for all $F=(V_0,\ldots,V_{k-1})\in\mathcal{X}_{\underline{j}}$}
\end{equation}
(with the notation of Section \ref{classicals-notation}).

\begin{proof}[Proof of Claim 2]
Part {\rm (a)} is an easy consequence of Claim 1, it remains to
show part {\rm (b)}.
Take $v\in V_{p_{t}}\cap E_{j_{p_{t}}}\setminus\{0\}$. 
Since $F\in\mathcal{F}^\omega_{e,\underline{d},\underline{c}}$,
we get $e(v)\in V_{c_{p_{t}}}\cap E_{j_{p_t}+2}$.
In the case where $t=1$, we have by definition ${p_1}=1$ and by assumption $c_1=0$, hence $e(v)\in V_0=\{0\}$. In the case where $t>1$, 
the construction of the sequence $\underline{p}$ guarantees that $c_{p_t}<p_{t-1}$ and $j_{p_t}<j_{p_{t-1}}$. In fact, we may note that the integers $j_{p_t},j_{p_{t-1}}$
have the same parity (both are even if $\omega$ is symmetric and odd if $\omega$ is symplectic, see Section \ref{classicals-notation}). Thus
\[j_{p_t}+2\leq j_{p_{t-1}}.\]
Then, part {\rm (a)} implies that $V_{c_{p_t}}\cap E_{j_{p_t}+2}=0$. In both cases, we conclude that $e(v)=0$, whence
\[v\in V_{p_t}\cap E_{j_{p_t}}\cap \ker e.\]
From Section \ref{classicals-notation}, we know that $E_{j_{p_t}}\cap \ker e$ is nonzero only if $j_{p_t}=\mu_{\ell_t}-1$ for some $\ell_t\in\{1,\ldots,r\}$, and in this case we have $E_{j_{p_t}}\cap\ker e=\langle v_{j_{p_t}}^{\ell_t}\rangle_\mathbb{C}$.
Whence $v_{j_{p_t}}^{\ell_t}\in V_{p_t}$. Since $V_{p_t}$ is an isotropic space, $v_{j_{p_t}}^{\ell_t}$ must be an isotropic vector, thus $j_{p_t}\not=0$ (see (\ref{omega-def})) and so $\mu_{\ell_t}\geq 2$. The proof of Claim 2 is complete.
\end{proof}

Finally, we note that the last term $p_{s}$ of the sequence $\underline{p}$ 
has the following characterization:
\[\mbox{for all $p\in\{p_{s}+1,\ldots,k-1\}$ such that $j_p<j_{p_s}$, we have $c_p\geq p_{s}$}.\]
Since the sequence $\underline{c}$ is nondecreasing, this can be rephrased as follows: 
there is $p_0\in\{p_{s}+1,\ldots,k\}$ such that
\begin{equation}
\label{classicals-newnew-7bis}
 \left\{\begin{array}{ll}
 \mbox{$j_p\geq j_{p_{s}}$ and $c_p<p_{s}$} & \mbox{for $p_{s}\leq p<p_0$,} \\[1mm]
\mbox{$c_p\geq p_{s}$} & \mbox{for $p_0\leq p\leq k-1$.}\end{array}\right.
\end{equation}

\subsubsection{The variety $\tilde{\mathcal{X}}_{\underline{j}}$} We apply the construction of Section \ref{section-classicals-1-1-5} to the number $\ell:=\ell_s$ given in Claim 2\,{\rm (b)} of Section \ref{section-8-6-1}. Specifically, we deal with the space 
$\tilde{V}=\langle v_i^q:(q,i)\not=(\ell_s,\pm(\mu_{\ell_s}-1))\rangle_\mathbb{C}$,
the nilpotent element $\tilde{e}=\pi e\iota$ (where again $\iota:\tilde{V}\to V$ and $\pi:V\to\tilde{V}$ denote the inclusion and the orthogonal projection), and the torus $\tilde{S}=\{\tilde\lambda(t):t\in\mathbb{C}^*\}$ where $\tilde\lambda(t)=\pi\lambda(t)\iota=\lambda(t)|_{\tilde{V}}$.
Note in particular that the definition of $\tilde{e}$ yields
\begin{equation}
\label{eetilde-2}
\mathrm{Im}(e|_{\tilde{V}}-\tilde{e})\subset\langle v_{\mu_{\ell_s}-1}^{\ell_s}\rangle_\mathbb{C}.
\end{equation}
Let $\underline{\tilde{d}}=(\tilde{d}_0=0<\tilde{d}_1<\ldots<\tilde{d}_{p_s-1}\leq\tilde{d}_{p_s}<\ldots<\tilde{d}_{k-1})$ be the sequence given by
\[\tilde{d}_p=\left\{\begin{array}{ll}
d_p & \mbox{if $0\leq p<p_s$},\\
d_p-1 & \mbox{if $p_s\leq p\leq k-1$}.\\
\end{array}\right.\]
Corresponding to this sequence, we consider the varieties $\tilde{\mathcal{F}}_{\underline{\tilde{d}}}^\omega$, $\tilde{\mathcal{F}}_{\tilde{e},\underline{\tilde{d}},\underline{c}}^\omega$, and $(\tilde{\mathcal{F}}_{\tilde{e},\underline{\tilde{d}},\underline{c}}^\omega)^{\tilde{S}}$
(the analogues of $\mathcal{F}_{\underline{d}}^\omega$, $\mathcal{F}_{e,\underline{d},\underline{c}}^\omega$, and $(\mathcal{F}_{e,\underline{d},\underline{c}}^\omega)^S$ for the space $\tilde{V}$; see also Remark \ref{remark-8-new}). By Lemma \ref{lemma-claim-2},
\begin{equation}
\label{newnewnew-43}
\mbox{the variety $(\tilde{\mathcal{F}}_{\tilde{e},\underline{\tilde{d}},\underline{c}}^\omega)^{\tilde{S}}$
admits an affine paving.}
\end{equation}
Set
\begin{eqnarray*}
\tilde{\mathcal{Z}}_{\underline{j}} & = & \{(W_0,\ldots,W_{k-1})\in(\tilde{\mathcal{F}}_{\underline{\tilde{d}}}^\omega)^{\tilde{S}}: W_p\cap E_{j_p}\not=W_{p-1}\cap E_{j_p}\ \forall p\not=p_s, \\
 & & \mbox{and}\ W_p\cap E_{j'}=W_{p-1}\cap E_{j'}\ \forall j'<j_p,\ \forall p\in\{1,\ldots,k-1\}\}.
\end{eqnarray*}
Thus $\tilde{\mathcal{Z}}_{\underline{j}}$ is stable by $\tilde{L}:=Z_{\tilde{G}}(\tilde{S})$ and, in view of (\ref{newnewnew-43}) and Lemma \ref{lemma-classicals-1}, we have:
\begin{equation}
\label{newnewnew-44}
\mbox{the variety $\tilde{\mathcal{X}}_{\underline{j}}:=\tilde{\mathcal{Z}}_{\underline{j}}\cap(\tilde{\mathcal{F}}_{\tilde{e},\underline{\tilde{d}},\underline{c}}^\omega)^{\tilde{S}}$ admits an affine paving.}
\end{equation}

\subsubsection{An isomorphism between $\mathcal{X}_{\underline{j}}$ and $\tilde{\mathcal{X}}_{\underline{j}}$}
By (\ref{classicals-newnew-6}), we have
\[\mathcal{X}_{\underline{j}}=\mathcal{Z}'_{\underline{j}}\cap(\mathcal{F}_{e,\underline{d},\underline{c}}^\omega)^S\]
where
\[\mathcal{Z}'_{\underline{j}}:=\{(V_0,\ldots,V_{k-1})\in\mathcal{Z}_{\underline{j}}:v_{\mu_{\ell_s}-1}^{\ell_s}\in V_{p_s}\}.\]
We define the maps
\[\Phi:\mathcal{Z}'_{\underline{j}}\to\tilde{\mathcal{Z}}_{\underline{j}},\ (V_0,\ldots,V_{k-1})\mapsto (V_0\cap\tilde{V},\ldots,V_{k-1}\cap\tilde{V})\]
and
\[\begin{array}{l}\Psi:\tilde{\mathcal{Z}}_{\underline{j}}\to\mathcal{Z}'_{\underline{j}},\ \\(W_0,\ldots,W_{k-1})\mapsto 
(W_0,\ldots,W_{p_s-1},W_{p_s}+\langle v_{\mu_{\ell_s}-1}^{\ell_s}\rangle_\mathbb{C},\ldots,W_{k-1}+\langle v_{\mu_{\ell_s}-1}^{\ell_s}\rangle_\mathbb{C}).
\end{array}\]
It is straightforward to check that $\Phi$ and $\Psi$ are well defined, algebraic, and in fact mutually inverse isomorphisms. We claim that
\begin{equation}
\label{newnewnew-45}
\Phi(\mathcal{X}_{\underline{j}})=\tilde{\mathcal{X}}_{\underline{j}}\ \ (\mbox{so that $\Phi$ restricts to an isomorphism $\mathcal{X}_{\underline{j}}\stackrel{\sim}{\to}\tilde{\mathcal{X}_{\underline{j}}}$}).
\end{equation}
The remainder of this subsection is devoted to the verification of (\ref{newnewnew-45}). 
Fix $F=(V_0,\ldots,V_{k-1})\in\mathcal{Z}'_{\underline{j}}$ and 
let $\Phi(F)=(W_0,\ldots,W_{k-1})$. Specifically, it is sufficient to show:

\medskip
\noindent
{\it Claim 1:} If $F$ belongs to $\mathcal{F}_{e,\underline{d},\underline{c}}^\omega$, then $\Phi(F)$ belongs to $\tilde{\mathcal{F}}_{\tilde{e},\underline{\tilde{d}},\underline{c}}^\omega$.

\medskip
\noindent
{\it Claim 2:} If $\Phi(F)$ belongs to $\tilde{\mathcal{F}}_{\tilde{e},\underline{\tilde{d}},\underline{c}}^\omega$,
then $F$ belongs to $\mathcal{F}_{e,\underline{d},\underline{c}}^\omega$.

\begin{proof}[Proof of Claim 1]
Let
$(\overline{V}_0,\ldots,\overline{V}_{2k-1})$ and $(\overline{W}_0,\ldots,\overline{W}_{2k-1})$
denote the completions of $F$ and $\Phi(F)$ in the sense of Proposition 
\ref{proposition-classical-4}, that is,
\[(\overline{V}_0,\ldots,\overline{V}_{2k-1}):=(V_0,V_1,\ldots,V_{k-1},V_{k-1}^\perp,\ldots,V_1^\perp,V_0^\perp)\]
and
\[(\overline{W}_0,\ldots,\overline{W}_{2k-1}):=(W_0,W_1,\ldots,W_{k-1},W_{k-1}^{\tilde\perp},\ldots,W_1^{\tilde\perp},W_0^{\tilde\perp}),\]
where $\perp$ and $\tilde\perp$ respectively denote the orthogonals in the spaces $(V,\omega)$ and $(\tilde{V},\omega)$. It is easy to see that
\begin{equation}
\label{newnewnew-46}
\overline{W}_p=\overline{V}_p\cap\tilde{V}=\pi(\overline{V}_p)\ \mbox{ for all $p\in\{0,\ldots,2k-1\}$}
\end{equation}
and
\begin{equation}
\label{newnewnew-47}
\overline{V}_p=\left\{\begin{array}{ll}
\overline{W}_p & \mbox{for $0\leq p<p_s$,} \\
\overline{W}_p\oplus\langle v_{\mu_{\ell_s}-1}^{\ell_s}\rangle_{\mathbb{C}} & \mbox{for $p_s\leq p<2k-p_s$,} \\
\overline{W}_p\oplus\langle v_{\mu_{\ell_s}-1}^{\ell_s},v_{-(\mu_{\ell_s}-1)}^{\ell_s}\rangle_{\mathbb{C}} & \mbox{for $2k-p_s\leq p\leq 2k-1$.}
\end{array}\right.
\end{equation}
The assumption that $F$ belongs to $\mathcal{F}_{e,\underline{d},\underline{c}}^\omega$ means that
\begin{equation}
\label{newnewnew-48}
e(\overline{V}_p)\subset \overline{V}_{c_p}\ \mbox{ for all $p\in\{1,\ldots,2k-1\}$.}
\end{equation}
Using (\ref{newnewnew-46}), (\ref{newnewnew-48}), and the definition of $\tilde{e}=\pi e\iota$, we obtain that
\[\tilde{e}(\overline{W}_p)=\pi(e(\overline{V}_p\cap\tilde{V}))\subset\pi(\overline{V}_{c_p})=\overline{W}_{c_p}\ \mbox{ for all $p\in\{1,\ldots,2k-1\}$.}\]
Whence $\Phi(F)\in\tilde{\mathcal{F}}_{\tilde{e},\underline{\tilde{d}},\underline{c}}^\omega$.
\end{proof}

\begin{proof}[Proof of Claim 2]
Here, we assume that
\begin{equation}
\label{newnewnew-49}
\tilde{e}(\overline{W}_p)\subset \overline{W}_{c_p}\ \mbox{ for all $p\in\{1,\ldots,2k-1\}$}
\end{equation}
and we need to show the inclusion
\begin{equation}
\label{newnewnew-50}
e(\overline{V}_p)\subset \overline{V}_{c_p}\ \mbox{ for all $p\in\{1,\ldots,2k-1\}$.}
\end{equation}
Let $p\in\{1,\ldots,2k-1\}$. 
Recall the element $p_0\in\{p_s+1,\ldots,k\}$ from (\ref{classicals-newnew-7bis}).
We distinguish four cases.

\medskip
\noindent
{\it Case 1:} $1\leq p<p_0$.

By Claim 1 of Section \ref{section-8-6-1} and (\ref{classicals-newnew-7bis}), we have
\[j_q\geq j_{p_s}\ \mbox{ for all $q\in\{1,\ldots,p\}$}.\]
Since $\Phi(F)\in \tilde{\mathcal{Z}}_{\underline{j}}$, this implies
\[\overline{W}_p=W_p\subset\bigoplus_{j'\geq j_{p_s}}E_{j'}\cap \tilde{V}.\]
Since $e$ and $\tilde{e}$ coincide on $\bigoplus_{j'\geq j_{p_s}}E_{j'}\cap \tilde{V}$, using also (\ref{newnewnew-47}), the fact that $v_{\mu_{\ell_s}-1}^{\ell_s}\in\ker e$, and (\ref{newnewnew-49}), we derive\[e(\overline{V}_p)=e(\overline{W}_p)=\tilde{e}(\overline{W}_p)\subset\overline{W}_{c_p}\subset\overline{V}_{c_p}.\]
This shows (\ref{newnewnew-50}) in Case 1.

\medskip
\noindent
{\it Case 2:} $p_0\leq p\leq k-1$.

By (\ref{classicals-newnew-7bis}), we get $c_p\geq p_s$ in this case, whence
\begin{equation}
\label{newnewnew-51}
\langle v_{\mu_{\ell_s}-1}^{\ell_s}\rangle_{\mathbb{C}}\subset V_{p_s}=\overline{V}_{p_s}\subset \overline{V}_{c_p}.
\end{equation}
Combining (\ref{newnewnew-47}), the fact that $v_{\mu_{\ell_s}-1}^{\ell_s}\in\ker e$, (\ref{eetilde-2}), (\ref{newnewnew-49}), and (\ref{newnewnew-51}), we obtain 
\[e(\overline{V}_p)=e(\overline{W}_p)\subset\tilde{e}(\overline{W}_p)+\langle v_{\mu_{\ell_s}-1}^{\ell_s}\rangle_\mathbb{C}\subset\overline{W}_{c_p}+\overline{V}_{c_p}=\overline{V}_{c_p}.\]
This proves (\ref{newnewnew-50})  in Case 2.

\medskip
\noindent
{\it Case 3:} $k\leq p< 2k-p_s$.

We claim that the following alternative is valid:
\begin{equation}
\label{newnewnew-52}
c_k\geq p_s\quad\mbox{or}\quad\mu_{\ell_s}=2.
\end{equation}
We show (\ref{newnewnew-52}) by arguing by contradiction, so assuming that $c_k<p_s$ and $\mu_{\ell_s}\geq 3$ (it was already noticed that $\mu_{\ell_s}\geq 2$, see Claim 2\,{\rm (b)} in Section \ref{section-8-6-1}). 
Take an element $F'=(V'_0,\ldots,V'_{k-1})\in\mathcal{X}_{\underline{j}}$
(it was assumed that the set $\mathcal{X}_{\underline{j}}$ is nonempty).
On the one hand, the fact that $c_k<p_s$ implies 
\begin{equation}
\label{newnewnew-53-bis}
v_{\mu_{\ell_s}-1}^{\ell_s}\notin V'_{c_k}=\overline{V}'_{c_k}
\end{equation}
(by Claim 2 in Section \ref{section-8-6-1}).
It also implies that
\[j_q\geq j_{p_s}\ \mbox{ for all $q\in\{1,\ldots,k-1\}$}\]
(see (\ref{classicals-newnew-7bis}) and Claim 1 in Section \ref{section-8-6-1}), hence
\begin{equation}
\label{newnewnew-53}
V'_{k-1}\subset\bigoplus_{j'\geq j_{p_s}}E_{j'}
\end{equation}
(since $F'\in\mathcal{Z}_{\underline{j}}$).
On the other hand, the fact that $\mu_{\ell_s}\geq 3$ implies that the vector $v_{\mu_{\ell_s}-3}^{\ell_s}$ is orthogonal to $E_{j'}$ for all $j'\geq j_{p_s}$ (see (\ref{omega-def}), recalling that $j_{p_s}=\mu_{\ell_s}-1$), thus (\ref{newnewnew-53}) implies
\[v_{\mu_{\ell_s}-3}^{\ell_s}\in V'^\perp_{k-1}=\overline{V}'_k.\]
Since $F'\in\mathcal{F}_{e,\underline{d},\underline{c}}^\omega$, this yields
\[v_{\mu_{\ell_s}-1}^{\ell_s}=e(v_{\mu_{\ell_s}-3}^{\ell_s})\in e(\overline{V}'_k)\subset\overline{V}'_{c_k},\]
a contradiction to (\ref{newnewnew-53-bis}). Therefore, we have shown (\ref{newnewnew-52}).

We now distinguish two cases depending on the alternative in (\ref{newnewnew-52}).
First, assume that $c_k\geq p_s$. Thus, $c_p\geq p_s$.
This property together with (\ref{newnewnew-47}), the fact that $v_{\mu_{\ell_s}-1}^{\ell_s}\in\ker e$, (\ref{eetilde-2}), and (\ref{newnewnew-49}), implies that 
\[e(\overline{V}_p)=e(\overline{W}_p)\subset\tilde{e}(\overline{W}_p)+\langle v_{\mu_{\ell_s}-1}^{\ell_s}\rangle_\mathbb{C}\subset\overline{W}_{c_p}+\overline{V}_{p_s}\subset\overline{V}_{c_p}.\]
Second, assume that $\mu_{\ell_s}=2$. This property guarantees that the space $\tilde{V}=\langle v_i^\ell:\ell\not=\ell_s\rangle_\mathbb{C}$ is $e$-stable, so $\tilde{e}=e|_{\tilde{V}}$. Thus, by (\ref{newnewnew-47}) and (\ref{newnewnew-49}), we get
\[e(\overline{V}_p)=e(\overline{W}_p)=\tilde{e}(\overline{W}_p)\subset\overline{W}_{c_p}\subset\overline{V}_{c_p}.\]
In both situations, we obtain that $e(\overline{V}_p)\subset\overline{V}_{c_p}$, hence (\ref{newnewnew-50}) holds in Case 3.

\medskip
\noindent
{\it Case 4:} $2k-p_s\leq p\leq 2k-1$.

First, we check the following relation
\begin{equation}
\label{newnewnew-57}
c_{2k-p_s}\geq p_s.
\end{equation}
To show (\ref{newnewnew-57}), choose any element $F'=(V'_0,\ldots,V'_{k-1})\in\mathcal{X}_{\underline{j}}$ (recall that the set $\mathcal{X}_{\underline{j}}$ is assumed to be nonempty). In view of (\ref{omega-def}) and Claim 2 in Section \ref{section-8-6-1}, we then have
\[v_{-(\mu_{\ell_s}-1)}^{\ell_s}\in V'^\perp_{p_s-1}=\overline{V}'_{2k-p_s}.\]
Knowing that $F'\in\mathcal{F}_{e,\underline{d},\underline{c}}^\omega$, this implies
that $e(v_{-(\mu_{\ell_s}-1)}^{\ell_s})\in \overline{V}'_{c_{2k-p_s}}$, whence
\[v_{\mu_{\ell_s}-1}^{\ell_s}=e^{\mu_{\ell_s}-1}(v_{-(\mu_{\ell_s}-1)}^{\ell_s})\in\overline{V}'_{c_{2k-p_s}}\]
(because the space 
$\overline{V}'_{c_{2k-p_s}}$ is in particular $e$-stable). Invoking again Claim 2 in Section \ref{section-8-6-1}, we conclude that (\ref{newnewnew-57}) holds true.

Our next claim is that the following inclusion holds:
\begin{equation}
\label{newnewnew-58}
e(v_{-(\mu_{\ell_s}-1)}^{\ell_s})\in\overline{V}_{c_{2k-p_s}}.
\end{equation}
In the case where $\mu_{\ell_s}=2$, then (\ref{newnewnew-47}) and (\ref{newnewnew-57}) imply
\[e(v_{-(\mu_{\ell_s}-1)}^{\ell_s})=v_{\mu_{\ell_s}-1}^{\ell_s}\in\overline{V}_{p_s}\subset \overline{V}_{c_{2k-p_s}}\]
so that (\ref{newnewnew-58}) holds in this case. It remains to show (\ref{newnewnew-58})
in the case where $\mu_{\ell_s}\geq 3$.
On the one hand, in view of (\ref{newnewnew-52}), this relation forces $c_k\geq p_s$, hence $c_{p_0}\geq p_s$ (see (\ref{classicals-newnew-7bis})).
By (\ref{classicals-newnew-cdual}), we get
\[p_s\leq c_{p_0}=c_{p_0}^*=|\{q=1,\ldots,2k-1:c_q\geq 2k-p_0\}|.\]
The sequence $\underline{c}$ being nondecreasing, we derive
\begin{equation}
\label{newnewnew-59}
c_{2k-p_s}\geq 2k-p_0.
\end{equation}
On the other hand, 
by Claim 1 in Section \ref{section-8-6-1} and (\ref{classicals-newnew-7bis}), it turns out that
$j_q\geq j_{p_s}$ for all $q\in\{1,\ldots,p_0-1\}$,
whence (since $F\in\mathcal{Z}_{\underline{j}}$)
\[V_{p_0-1}\subset\bigoplus_{j'\geq j_{p_s}}E_{j'}.\]
By (\ref{omega-def}), we deduce
\begin{equation}
\label{newnewnew-60}
e(v_{-(\mu_{\ell_s}-1)}^{\ell_s})=v_{-(\mu_{\ell_s}-3)}^{\ell_s}\in V_{p_0-1}^\perp=\overline{V}_{2k-p_0.}
\end{equation}
Relation (\ref{newnewnew-58}) now follows by comparing (\ref{newnewnew-59}) and (\ref{newnewnew-60}).

Finally, using (\ref{newnewnew-47}), the fact that $v_{\mu_{\ell_s}-1}^{\ell_s}\in\ker e$, (\ref{eetilde-2}), (\ref{newnewnew-49}), (\ref{newnewnew-57}), (\ref{newnewnew-58}), and the fact that $c_{2k-p_s}\leq c_p$ (since $2k-p_s\leq p$), we compute
\begin{eqnarray*}
e(\overline{V}_p) & = & e(\overline{W}_p+\langle v_{-(\mu_{\ell_s}-1)}^{\ell_s}\rangle_\mathbb{C})\subset\tilde{e}(\overline{W}_p)+\langle v_{\mu_{\ell_s}-1}^{\ell_s},e(v_{-(\mu_{\ell_s}-1)}^{\ell_s})\rangle_\mathbb{C} \\
 & \subset & \overline{W}_{c_p}+\overline{V}_{c_{2k-p_s}}\subset\overline{V}_{c_p}.
\end{eqnarray*}
This establishes (\ref{newnewnew-50}) in Case 4 and completes the proof of Claim 2.
\end{proof}

\subsubsection{Conclusion}
The proof of Proposition \ref{proposition-claim-1} in the case where $c_1=0$ is achieved 
by combining (\ref{newnewnewnew-40}), (\ref{newnewnew-44}), and (\ref{newnewnew-45}). This completes the proof of Proposition \ref{proposition-claim-1}.
Therefore, Proposition \ref{proposition-section-8} is established, and so the proof of Theorem \ref{theorem-1} is complete.


\end{document}